\documentclass[a4paper,10pt]{article}
\usepackage{amsmath,amssymb,bbm,theorem}

\newcommand{\assign}{:=}

\numberwithin{equation}{section} 

\newcommand{\nocomma}{}
\newcommand{\noplus}{}
\newcommand{\nosymbol}{}
\newcommand{\tmem}[1]{{\em #1\/}}

\newcommand{\tmop}[1]{\ensuremath{\operatorname{#1}}}

\newcommand{\tmstrong}[1]{\textbf{#1}}

\newenvironment{proof}{\noindent\textbf{Proof\ }}{\hspace*{\fill}$\Box$\medskip}
\newtheorem{corollary}{Corollary}
\newtheorem{definition}{Definition}
\newtheorem{lemma}{Lemma}
\newtheorem{proposition}{Proposition}
{\theorembodyfont{\rmfamily}\newtheorem{remark}{Remark}}
\newtheorem{theorem}{Theorem}
\newtheorem{Conjecture}{Conjecture}

\newtheorem{Statement}{Statement}
\newtheorem{Fact}{Fact}

\newcommand{\XXint}[3]{{\setbox}0=\text{\ensuremath{#1 #2 #3 \int}}
{\vcenter{\text{\ensuremath{#2 #3}}}}{\kern}-.5{\tmwd}0}

\newcommand{\opn}[2]{\newcommand{\1}{\}} {\opn}{\Rm{Rm}} {\opn}{\Ric{Ric}}
{\opn}{\Rc{Rc}} {\opn}{\Scal{Sc}} {\opn}{\Tr{Tr}} {\opn}{\Trac{Tr}}
{\opn}detdet {\opn}{\diam{diam}} {\opn}{\dist{dist}} {\opn}{\Im}Im
{\opn}{\div}div {\opn}{\Ker{Ker}} {\opn}expexp {\opn}{\Vol{Vol}}
{\opn}{\exph{exph}} {\opn}{\Herm{Herm}} {\opn}{\End{End}} {\opn}{\Hess{Hess}}
{\opn}{\Vol{Vol}}}

\newcommand{\R}{\mathbb{R}}

\newcommand{\contract}{{\kern}-1.5pt{\vrule} width6.0pt height0.4pt depth0pt
{\vrule} width0.4pt height4.0pt depth0pt}
\newcommand{\retract}{{\kern}-1.5pt{\vrule} width0.4pt height4.0pt depth0pt
{\vrule} width6.0pt height0.4pt depth0pt}
\newcommand{\Openbox}{{\leavevmode} {\hfil}{\vrule} width{\boxrulethickness}
{\vbox} to{\Openboxwidth{{\advance}{\Openboxwidth} -2{\boxrulethickness}
{\hrule} height {\boxrulethickness} width{\Openboxwidth}{\vfil} {\hrule}
height{\boxrulethickness}}}{\vrule} width{\boxrulethickness}{\hfil} }

\begin{document}

\title{On maximal totally real embeddings}\author{\\
{\tmstrong{Nefton Pali}}}\date{}\maketitle

\begin{abstract}
  We consider complex structures with totally real zero section of the
  tangent bundle. We assume that the complex structure tensor is real-analytic
  along the fibers of the tangent bundle. This assumption is quite natural in
  view of a well known result by Bruhat and Whitney \cite{Br-Wh}. We provide explicit
  integrability equations for such complex structures in terms of the
  fiberwise Taylor expansion. In a particular geometric case considered in the literature, we explicit much further the fiberwise Taylor expansion of the complex structure as well as the integrability equations. 
\end{abstract}
{\def\thefootnote{\relax}\footnote{ Linear and non-linear connections over vector bundles.
}} 

\tableofcontents

\section{Introduction and statement of the main result}

Let $\left( E, \pi_E, M \right)$ be a smooth vector bundle over a manifold
$M$. Let $E_p$ be the fiber of $E$ over a point $p \in M$ and let $\eta \in
E_p$. We consider the transition map $\tau_{\eta} \left( v \right) \assign
\eta + v$ acting over $E_p$ and we consider its differential 
$$
d_0 \tau_{\eta}
: T_{E_p, 0} \longrightarrow T_{E_p, \eta}\,,
$$ 
at the point $0$. Composing $d_0
\tau_{\eta}$ with the canonical isomorphism $E_p \simeq T_{E_p, 0}$ we obtain
an isomorphism map
\begin{equation}
  \label{transISO} T_{\eta} : E_p \longrightarrow T_{E_p, \eta}\, .
\end{equation}
We denote by $0_M$ the zero section of $E$. Differentiating the identity
$\tmop{id}_M = \pi_E \circ 0_M$ we obtain $\mathbbm{I}_{T_{M, p}} = d_{0_p}
\pi_E \circ d_p 0_M$. This implies the decomposition
\begin{eqnarray*}
  T_{E, 0_p}  =  d_p \,0_M \left( T_{M, p} \right) \oplus \tmop{Ker} d_{0_p}
  \pi_E \,.
\end{eqnarray*}
We notice also the obvious equalities $\tmop{Ker} d_{\eta} \pi_E = d_0
\tau_{\eta} \left( T_{E_p, 0} \right) = T_{\eta} \left( E_p \right) \simeq
E_p$, for any $\eta \in E_p$. Now applying this to $\eta = 0_p$, using the
previous decomposition and the canonical isomorphism $d_p \,0_M \left( T_{M, p}
\right) \simeq T_{M, p}$, we infer the existence of the canonical isomorphism
$T_{E, 0_p} \simeq T_{M, p} \oplus E_p$, that we rewrite as
\begin{equation}
  \label{FrstISO} T_{E \mid M} \simeq T_M \oplus E\, .
\end{equation}
\begin{definition}
  A real sub-manifold $M$ of an almost complex manifold $\left( X, J \right)$
  is called totally real if $T_{M, p} \cap J \left( T_{M, p} \right) = 0_p$
  for all $p \in M$. A totally real sub-manifold $M$ of an almost complex
  manifold $\left( X, J \right)$ is called maximally totally real if
  $\dim_{_{\mathbbm{R}}} M = \dim_{_{\mathbbm{C}}} X$.
\end{definition}

\subsection{$M$-totally real almost complex structures over $T_M$}\label{totR}

We consider $M$ included inside $T_M$ via the zero section. We know by the isomorphism
(\ref{FrstISO}) with $E = T_M$, that this embedding induces the canonical
isomorphism $T_{T_M \mid M} \simeq T_M \oplus T_M$. The vector bundle $T_{T_M
\mid M}$ is a complex one with the canonical complex structure $J^{\tmop{can}}
: \left( u, v \right) \longmapsto \left( - v, u \right)$ acting on the fibers.

Any almost complex structure which is a continuous extension of
$J^{\tmop{can}}$ in a neighborhood of $M$ inside $T_M$ makes $M$ a maximally
totally real sub-manifold of $T_M$. 

Over an arbitrary small neighborhood of $M$
inside $T_M$ the complex distribution $T_{T_M}^{0, 1}$ is horizontal with
respect to the natural projection $\pi : T_M \longrightarrow M$.

We remind that the data of a smooth complex horizontal distribution over $T_M$
coincides with the one of section 
$$
A \in C^{\infty} \left( T_M, \pi^{\ast}
\mathbbm{C}T^{\ast}_M \otimes_{_{\mathbbm{C}}} \mathbbm{C}T_{T_M} \right),
$$
such that $d \pi \cdot A =\mathbbm{I}_{\pi^{\ast} \mathbbm{C}T_M}$.

For any complex vector field $\xi \in C^{\infty} \left( M, \mathbbm{C}T_M \right)$ we will denote
by abuse of notation $A \left( \xi \right) \equiv A \cdot \left( \xi \circ
\pi \right)$. The section $A$ evaluated at the point $\eta \in T_M$ will be
denoted by $A_{\eta}$.

We notice that we can write $A = \alpha + i \beta$, with 
$$
\alpha, \beta \in
C^{\infty} \left( T_M, \pi^{\ast} T^{\ast}_M \otimes_{_{\mathbbm{R}}} T_{T_M}
\right),
$$ 
such that $d \pi \cdot \alpha =\mathbbm{I}_{\pi^{\ast} T_M}$ and
$\beta_{\eta} = T_{\eta} B_{\eta}$, with $B \in C^{\infty} \left( T_M,
\pi^{\ast} \tmop{End} \left( T_M \right) \right)$. The section $A$ determines
an almost complex structure $J_A$ over $T_M$ such that 
$$
T_{T_M, J_A, \eta}^{0,
1} = A_{\eta} \left( \mathbbm{C}T_{M, \pi \left( \eta \right)} \right) \subset
\mathbbm{C}T_{T_M, \eta}\,,
$$
if and only if
\begin{equation}
  \label{JAcompat} A_{\eta} \left( \mathbbm{C}T_{M, \pi \left( \eta \right)}
  \right) \cap \overline{A_{\eta} \left( \mathbbm{C}T_{M, \pi \left( \eta
  \right)} \right)} = 0\, .
\end{equation}
This condition is equivalent to the property:
\begin{equation}
  \label{Aequality}  \overline{A^{^{}}_{\eta} \left( \bar{\xi}_1 \right)} =
  A_{\eta} \left( \xi_2 \right),
\end{equation}
implies $\xi_1 = \xi_2 = 0$. Taking $d_{\eta} \pi$ in the equality
(\ref{Aequality}) we infer $\xi_1 = \xi_2$. Thus equality (\ref{Aequality}) is
equivalent to $\left( \overline{A} - A \right) \left( \xi_1 \right) = 0$ and
the previous property is equivalent to $\tmop{Ker} \left( \overline{A} - A
\right) = 0$, i.e.
\begin{eqnarray*}
  B  \in  C^{\infty} \left( T_M, \pi^{\ast} \tmop{GL} \left( T_M \right)
  \right) .
\end{eqnarray*}
We notice that with respect to the canonical complex structure of $T_{T_M \mid
M}$ we have \ the equality $\left( u, v \right)^{0, 1} = \left( \xi, i \xi
\right)$, with $\xi \assign \left( u - i v \right) / 2$. Then $J_A$ is an
extension of this complex structure over an open neighborhood $U \subseteq
T_M$ of $M$ if and only if for any point $p \in M$ we have $\alpha_{0_p} = d_p
0_M$ and $B_{0_p} =\mathbbm{I}_{T_{M, p}}$. We denote by
\begin{eqnarray*}
  T  \in  C^{\infty} \left( T_M, \pi^{\ast} T^{\ast}_M
  \otimes_{_{\mathbbm{R}}} T_{T_M} \right),
\end{eqnarray*}
the canonical section which at the point $\eta \in T_M$ takes the value
$T_{\eta}$.

\begin{definition}
  Let $M$ be a smooth manifold. An $M$-totally real almost complex structure
  over an open neighborhood $U \subseteq T_M$ of the image of the zero section $0_M$ is a couple
  $\left( \alpha, B \right)$ with 
  $$
  \alpha \in C^{\infty} \left( U, \pi^{\ast}
  T^{\ast}_M \otimes_{_{\mathbbm{R}}} T_{T_M} \right),
  $$ and 
$$
B \in C^{\infty}
  \left( U, \pi^{\ast} \tmop{GL} \left( T_M \right) \right),
$$ such that $d \pi
  \cdot \alpha =\mathbbm{I}_{\pi^{\ast} T_M}$ over $U$ and such that
  $\alpha_{0_p} = d_p \,0_M$, $B_{0_p} =\mathbbm{I}_{T_{M, p}}$, for all $p \in
  M$. With $A := \alpha + i T B$, the almost complex structure $J_A$ associated
  to $\left( \alpha, B \right)$ is the one which satisfies 
  $$
  T_{T_M, J_A,
  \eta}^{0, 1} = A_{\eta} \left( \mathbbm{C}T_{M, \pi \left( \eta \right)}
  \right) \subset \mathbbm{C}T_{T_M, \eta} \,,
  $$ for all $\eta \in U \subseteq
  T_M$.
\end{definition}

Every almost complex smooth extension of the canonical complex structure
$J^{\tmop{can}}$ of $T_{T_M \mid M}$ over a neighborhood of $M$ inside $T_M$ can be expressed, over a sufficiently small neighborhood $U\subseteq T_M$ of
$M$, as the almost complex structure associated to a unique $M$-totally real
almost complex structure over $U$.

We provide below an explicit formula for the almost complex structure $J_A$.
For this purpose we notice first that for any vector $\xi \in T_{T_M, \eta}$,
\begin{eqnarray*}
  \xi_{J_A}^{0, 1} & = & \frac{1}{2} A_{\eta} \left[ d_{\eta} \pi - i B^{-
  1}_{\eta} T^{- 1}_{\eta} \left( \mathbbm{I}_{T_{T_M}} - \alpha_{\eta}\,
  d_{\eta} \pi \right) \right] \xi\,,\\
  &  & \\
  \xi_{J_A}^{1, 0} & = & \frac{1}{2}  \overline{A}_{\eta} \left[ d_{\eta} \pi
  + i B^{- 1}_{\eta} T^{- 1}_{\eta} \left( \mathbbm{I}_{T_{T_M}} -
  \alpha_{\eta} \,d_{\eta} \pi \right) \right] \xi \,.
\end{eqnarray*}
Indeed $\xi_{J_A}^{0, 1} \in T_{T_M, J_A, \eta}^{0, 1}$, $\xi_{J_A}^{1, 0} \in
T_{T_M, J_A, \eta}^{1, 0}$ and $\xi = \xi_{J_A}^{1, 0} + \xi_{J_A}^{0, 1}$. We
deduce the expression
\begin{equation}
  \label{FormulJA} J_{A, \eta} = - \alpha_{\eta}\, B^{- 1}_{\eta} T^{- 1}_{\eta}
  \left( \mathbbm{I}_{T_{T_M}} - \alpha_{\eta}\, d_{\eta} \pi \right) + T_{\eta}\,
  B_{\eta}\, d_{\eta} \pi .
\end{equation}
This shows that for any $\alpha$-horizontal vector $\xi \in T_{T_M, \eta}$,
i.e. $\xi = \alpha_{\eta}\, d_{\eta} \pi \,\xi$, we have
\begin{eqnarray*}
  J_{A, \eta} \,\xi & = & T_{\eta} \,B_{\eta}\, d_{\eta} \pi\, \xi \,.
\end{eqnarray*}
In equivalent terms
\begin{equation}
  \label{partcActJ} J_{A, \eta} \,\alpha_{\eta} \,v = T_{\eta}\, B_{\eta}\, v\,,
\end{equation}
for any $\eta \in U \subset T_M$ and any $v \in T_{M, \pi \left( \eta
\right)}$. Moreover (\ref{FormulJA}) implies
\begin{equation}
  \label{JAvert} J_{A, \eta \mid \tmop{Ker} d_{\eta} \pi} = - \,\alpha_{\eta}\,
  B^{- 1}_{\eta} T^{- 1}_{\eta} .
\end{equation}
A well known theorem by Bruhat and Whitney \cite{Br-Wh} states that for any
real-analytic manifold $M$ there exist a complex manifold $\left( X, J
\right)$ and a real-analytic embedding of $M$ in $X$ such that as a
sub-manifold of $X$, $M$ is maximally totally real. In addition one can arrange
that $X$ is an open neighborhood $U \subseteq T_M$ of the zero section and
$J_{\mid M} = J^{\tmop{can}}$.

Moreover Bruhat and Whitney show \cite{Br-Wh} that if $X$ is a real-analytic
manifold equipped with two different real-analytic complex structures $J_1$
and $J_2$ which contains a real analytic sub-manifold $M$ which is maximally
totally real with respect to both $J_1$ and $J_2$, then there exist
neighborhoods $U_1$ and $U_2$ of $M$ inside $X$ and a real-analytic
diffeomorphism $\kappa : U_1 \longrightarrow U_2$ which is the identity on $M$
and is a holomorphic mapping of $\left( U_1, J_1 \right)$ onto $\left( U_2,
J_2 \right)$.

In other words the structure $J$ constructed by Bruhat and Whitney in \cite{Br-Wh} is unique
up to complex isomorphisms. 
\\
We state below our results on the integrability conditions for $J$.

\subsection{The integrability equations for $M$-totally real almost complex
structures}

Let $\left( E, \pi_E, M \right)$ be a vector bundle over a manifold $M$. For an arbitrary section $B \in C^{\infty}( E, \pi_E^{\ast} \left( T^{\ast}_M \otimes E
\right))$, we define the derivative along the fiber
\begin{eqnarray*}
  D B  \in  C^{\infty} \left( E, \pi_E^{\ast} \left( E^{\ast} \otimes
  T^{\ast}_M \otimes E \right)_{_{}} \right),
\end{eqnarray*}
by the formula
\begin{eqnarray*}
  D_{\eta} B \left( v \right)  \assign  \frac{d}{d t} _{\mid_{t = 0}}
  B_{\eta + t v} \in T^{\ast}_{M, p} \otimes E_p\,,
\end{eqnarray*}
for any $\eta, v \in E_p$. We denote by $\tmop{Alt}_2$ the alternating
operator (without normalizing coefficient!) which acts on the first two
entries of a tensor. For any morphism $A : T_M \longrightarrow E$ and any
bilinear form $\beta : E \times T_M \longrightarrow E$ we define the
contraction operation 
$$
A \neg \beta \assign \tmop{Alt}_2 \left( \beta \circ A
\right),
$$
where the composition operator $\circ$ act on the first entry of
$\beta$. For a given
  covariant derivative operator $\nabla$ acting on the smooth sections of $T_M$, we denote by $H^{\nabla}$ the linear projection to the associated horizontal distribution. (See lemmas 14, 16 and definition  5 in subsection 9.1 of the appendix for precise 
  definitions and properties of $H^{\nabla}$). 
\begin{theorem}
  \label{GenInteg}Let $M$ be a smooth manifold and let $J_A$ with $A = \alpha
  + i T B$ be an $M$-totally real almost complex structure over an open
  neighborhood $U \subseteq T_M$ of the image of the zero section. Let also $\nabla$ be a
  covariant derivative operator acting on the smooth sections of $T_M$.
  Then $J_A$ is 
  integrable over $U$ if and only if the complex section $S
  \assign T^{-1}(H^{\nabla}-\overline{A}\,)$ satisfies the equation
  \begin{equation}
    \label{MainInteg} H_{\eta}^{\nabla} \neg \left( \nabla^{\tmop{End} \left(
    T_M \right), \pi} S \right)_{\eta} - S_{\eta} \neg D_{\eta} S \noplus
    \noplus \noplus + S_{\eta} \tau^{\nabla} + R^{\nabla} \cdot \eta = 0\,,
  \end{equation}
  for any point $\eta \in U$, where $\nabla^{\tmop{End} \left( T_M \right),
  \pi}$ is the covariant derivative operator acting on the smooth sections of
  $\pi^{\ast} \tmop{End} \left( T_M \right)$ induced by $\nabla$ and where
  $\tau^{\nabla}$ and $R^{\nabla}$ are respectively the torsion and curvature
  forms of $\nabla$.
\end{theorem}

We notice that $S_{\mid M} = i\mathbbm{I}_{T_M}$ by the conditions
$\alpha_{0_p} = H^{\nabla}_{0_p} = d_p \,0_M$ and $B_{0_p} =\mathbbm{I}_{T_{M,
p}}$.
\\
\\
{\tmstrong{Notation for the statement of the main theorem}}.
\\
\\
For any $A \in T^{\ast, \otimes p}_M \otimes \tmop{End}_{_{\mathbbm{C}}}
\left( \mathbbm{C}T_M \right)$ and for any $\theta \in T^{\ast, \otimes q}_M
\otimes \mathbbm{C}T_M$, the product operations of tensors $A \cdot \theta, A
\neg \,\theta \in T^{\ast, \otimes \left( p + q \right)}_M \otimes
\mathbbm{C}T_M$ are defined by
\begin{eqnarray*}
  \left( A \cdot \theta \right) \left( u_1, \ldots, u_p, v_1, \ldots, v_q
  \right) & \assign & A (u_1, \ldots, u_p) \cdot \theta (v_1, \ldots, v_q)\,,\\
  &  & \\
  \left( A \neg \theta \right)  \left( u_1, \ldots, u_p, v_1, \ldots, v_q
  \right) & \assign & \sum_{j = 1}^q \theta (v_1, \ldots, A (u_1, \ldots, u_p)
  \cdot v_j, \ldots, v_q) \,.
\end{eqnarray*}
We will denote for notation simplicity $R^{\nabla} \nosymbol . \, \theta
\assign R^{\nabla} \cdot \theta - R^{\nabla} \neg\, \theta$. We will denote by
$\tmop{Circ}$ the circular operator
\begin{eqnarray*}
  (\tmop{Circ} \theta) \left( v_1, v_2, v_3, \bullet \right) & \assign & \theta
  \left( v_1, v_2, v_3, \bullet \right) + \theta \left( v_2, v_3, v_1, \bullet
  \right) + \theta \left( v_3, v_1, v_2, \bullet \right),
\end{eqnarray*}
acting on the first three entries of any $q$-tensor $\theta$, with $q
\geqslant 3$. We define also the permutation operation $\theta_2 \left( v_1,
v_2, \bullet \right) \assign \theta \left( v_2, v_1, \bullet \right)$.

For any covariant derivative $\nabla$ acting on the smooth sections of
$\mathbbm{C}T_M$ we define the operator
\begin{eqnarray*}
  d_1^{\nabla} : C^{\infty} ( M, T^{\ast, \otimes k}_M
  \otimes_{_{\mathbbm{R}}} \mathbbm{C}T_M )  \longrightarrow 
  C^{\infty} ( M, \Lambda^2 T^{\ast}_M
  \otimes_{_{\mathbbm{R}}} T^{\ast, \otimes \left( k - 1 \right)}_M
  \otimes_{_{\mathbbm{R}}} \mathbbm{C}T_M),
\end{eqnarray*}
with $k \geqslant 1$ as follows
\begin{eqnarray*}
  d_1^{\nabla} A \left( \xi_1, \xi_2, \mu \right) & \assign & \nabla_{\xi_1} A
  \left( \xi_2, \mu \right) - \nabla_{\xi_2} A \left( \xi_1, \mu \right),
\end{eqnarray*}
with $\xi_1, \xi_2 \in T_M$ and with $\mu \in T^{\oplus \left( k - 1
\right)}_M$. 
Moreover for any 
\begin{eqnarray*}
 A&\in& C^{\infty}
(M, T^{\ast, \otimes \left( k + 1 \right)}_M
\otimes_{_{\mathbbm{R}}} \mathbbm{C}T_M)\,,
\\
\\
B&\in&C^{\infty}
(M, T^{\ast, \otimes \left( l + 1 \right)}_M
\otimes_{_{\mathbbm{R}}} \mathbbm{C}T_M)\,,
\end{eqnarray*}
we define the exterior product
\begin{eqnarray*}
  A \wedge_1 B & \in & C^{\infty} ( M, \Lambda^2 T^{\ast}_M
  \otimes_{_{\mathbbm{R}}} T^{\ast, \otimes \left( k + l - 1 \right)}_M
  \otimes_{_{\mathbbm{R}}} \mathbbm{C}T_M),
\end{eqnarray*}
as
\begin{eqnarray*}
  \left( A \wedge_1 B \right) \left( \xi_1, \xi_2, \eta, \mu \right) & \assign
  & A \left( \xi_1, B \left( \xi_2, \eta \right), \mu \right) - A \left(
  \xi_2, B \left( \xi_1, \eta \right), \mu \right),
\end{eqnarray*}
with $\xi_1, \xi_2 \in T_M$, $\eta \in T^{\oplus l}_M$ and $\mu \in T^{\oplus
\left( k - 1 \right)}_M$. We denote by $\tmop{Sym}_{r_1, \ldots, r_s}$ the
symmetrizing operator (without normalizing coefficient!) acting on the entries
$r_1, \ldots, r_s$ of a multi-linear form. 
We use in this paper the common convention that a sum and a product running
over an empty set is equal respectively to 0 and 1. 

With these notation we
can state our main theorem.

\begin{theorem}
  $\left( \text{\tmstrong{Integrability in the fiberwise real analytic case}}\right)$. \label{Maintheorem}
  
  Let $M$ be smooth manifold equipped with a torsion free covariant derivative operator $\nabla$ acting on
  the smooth sections of the tangent bundle $T_M$, let $U \subseteq T_M$ be an open neighborhood
  of the image of the zero section with connected fibers
  let $J_A$ be an $M$-totally real almost complex structure over $U$, real-analytic 
  along the fibers of $U$ and consider the fiberwise Taylor expansion at the origin
  \begin{eqnarray*}
    T^{-1}_{\eta} (H^{\nabla}-\overline{A}\,)_{\eta} \cdot \xi & = & i \xi+\sum_{k \geqslant 1} S_k \left( \xi, \eta^k
    \right),
  \end{eqnarray*}
  with $\eta \in T_M$ in a neighborhood of the image of the zero section, with $\xi \in T_{M, \pi \left( \eta \right)}$ arbitrary,
  with
  $$
  S_k \in C^{\infty} (
  M ,T^{\ast}_M \otimes_{_{\mathbbm{R}}} S^k T^{\ast}_M
  \otimes_{_{\mathbbm{R}}} \mathbbm{C}T_M )\,,
  $$ 
  with $\eta^k \assign \eta^{\times k} \in
  T^{\oplus k}_{M, \pi \left( \eta \right)}$ and let
  $\nabla^{S_1}$ be the complex covariant derivative operator acting on the
  smooth sections of $\mathbbm{C}T_M$ defined by 
  $$
  \nabla^{S_1}_{\xi} \eta
  \assign \nabla_{\xi} \eta + S_1 \left( \xi, \eta \right).
  $$
Then $J_A$ is integrable over $U$ if and only if
 $
    S_1  \in  C^{\infty} \left( M, S^2 T^{\ast}_M
    \otimes_{_{\mathbbm{R}}} \mathbbm{C}T_M \right)
    $,
$($i.e. $\nabla^{S_1}$ is torsion free$)$ and for all $k \geqslant 2$,
  \begin{eqnarray*}
   && S_k  =  \frac{i}{k} \nabla^{S_1} \sigma_{k - 1} + \frac{i}{\left( k + 1
    \right) !}
    \tmop{Sym}_{2, \ldots, k + 1} \beta_{k - 1} \left( \sigma_{k - 2} \right)
    + \sigma_k\,,\\
    &  & \\
   && \sigma_k  \in  C^{\infty} \left( M, S^{k + 1} T^{\ast}_M
    \otimes_{_{\mathbbm{R}}} \mathbbm{C}T_M \right),\\
    &  & \\
   && \tmop{Circ} \beta_{k + 1} \left( \sigma_k \right)  =  0\,,
  \end{eqnarray*}
  where $\sigma_1 \assign 0$, $\beta_1 \left( \sigma_0 \right) \assign
  R^{\nabla^{S_1}}$, $\beta_2 \left( \sigma_1 \right) \assign - \frac{i}{3}
  ( \nabla^{S_1} R^{\nabla^{S_1}})_2$ and for all $k \geqslant 3$,
  \begin{eqnarray*}
    \beta_k \left( \sigma_{k - 1} \right) & \assign & \frac{i}{k}
    R^{\nabla^{S_1}} .\, \sigma_{k - 1} + \frac{1}{\left( k + 1 \right) ! k!} \tmop{Sym}_{3, \ldots, k + 2}  \theta_k \left( \sigma_{k - 1} \right), \\
    &  & \\
   \theta_k \left( \sigma_{k - 1} \right) & \assign & i \sum_{r = 3}^{k - 1}
    \frac{\left( r + 1 \right) !}{r}  ( i\,d_1^{\nabla^{S_1}})^{k - r} ( R^{\nabla^{S_1}} . \,\sigma_{r - 1})\\
    &  & \\
    & - & 2\, i\, ( i\,d_1^{\nabla^{S_1}})^{k - 2} ( \nabla^{S_1}
    R^{\nabla^{S_1}})_2\\
    &  & \\
    & + &  \sum_{r = 4}^{k +1} r! \sum_{p = 2}^{r - 2} (
    i\,d_1^{\nabla^{S_1}} )^{k +1- r} \left( p S_p \wedge_1 S_{r - p}
    \right) .
  \end{eqnarray*}
\end{theorem}

In more explicit terms
\begin{eqnarray}
  S_2 & = & S^0_2 + \sigma_2\,,\label{S2express}
  \\\nonumber
  &  & \\
  S^0_2 \left( \xi_1, \xi_2, \xi_3 \right) & \assign & \frac{i}{6}  \left[
  R^{\nabla^{S_1}} \left( \xi_1, \xi_2 \right) \xi_3 + R^{\nabla^{S_1}} \left(
  \xi_1, \xi_3 \right) \xi_2 \right],\label{S20express}
  \\\nonumber
  &  & \\
  \sigma_2 & \in & C^{\infty} \left( M, S^3 T^{\ast}_M
  \otimes_{_{\mathbbm{R}}} \mathbbm{C}T_M \right),\label{sig2eq}
  \end{eqnarray}
  \begin{eqnarray}
  \tmop{Circ} \beta_3 \left( \sigma_2 \right)  & = & 0\,,\label{sig2circeq}
  \\\nonumber
  &  & \\
  \beta_3 \left( \sigma_2 \right) & \assign & \frac{i}{3} R^{\nabla^{S_1}} .\,
  \sigma_2 + \frac{1}{4!3!} \tmop{Sym}_{3, 4, 5} \theta_3 \left( \sigma_2
  \right),\label{B3express}
  \\\nonumber
  &  & \\
  \theta_3 \left( \sigma_2 \right) & \assign & 2 d_1^{\nabla^{S_1}} (
  \nabla^{S_1} R^{\nabla^{S_1}})_2 + 4! 2 S_2 \wedge_1 S_2 \,.
\end{eqnarray}
The assumption that the complex structure tensor is real-analytic along the
fibers of the tangent bundle is quite natural. Indeed in the case $M$ is real
analytic then the $M$-totally real complex structure constructed by Bruhat and
Whitney \cite{Br-Wh} is also real analytic with respect to the real analytic structure of
the tangent bundle induced by $M$.

In this paper we request from the readers very good knowledge of the geometric
theory of linear connections. Basics of such theory can be found in the
appendix.
\subsection{Application of the main integrability result}

Over a Riemannian manifold $\left( M, g \right)$, we denote by $${\cal V}^g \ni(\eta,t)\longmapsto \Phi_t^g (\eta)\in T_M,$$
the geodesic flow, where ${\cal V}^g\subset T_M\times \R$ is an open neighborhood of $T_M\times \{0\}$.
Let $\nabla^g$ be the Levi-Civita connection of the metric $g$. We denote by $H^g$ the liner projection to the associated horizontal distribution.
We state the following corollary of the main theorem \ref{Maintheorem}.

\begin{corollary}\label{MainCoroll}
  Let $\left( M, g \right)$ be a smooth Riemannian manifold, let $U\subseteq T_M$ be an open 
  neighborhood of the image of the zero section with connected fibers, let $J \equiv
  J_A$ be an $M$-totally real almost complex
  structure over $U$, real
  analytic along the fibers of $U$ and consider the fiberwise Taylor expansion at the origin 
  \begin{eqnarray*}
    T^{-1}_{\eta} (H^g-\overline{A}\,)_{\eta} \cdot \xi  =  i \xi+\sum_{k \geqslant 1} S_k \left( \xi, \eta^k
    \right),
  \end{eqnarray*}
  with $\eta \in T_M$ in a neighborhood of the image of the zero section, with $\xi \in T_{M, \pi \left( \eta \right)}$ arbitrary,
  with
  $$
  S_k \in C^{\infty} (
  M ,T^{\ast}_M \otimes_{_{\mathbbm{R}}} S^k T^{\ast}_M
  \otimes_{_{\mathbbm{R}}} \mathbbm{C}T_M )\,,
  $$ 
  and with $\eta^k \assign \eta^{\times k} \in
  T^{\oplus k}_{M, \pi \left( \eta \right)}$.   Then the statements (a) and (b) below are equivalent.
 
  (a) The almost complex structure $J$ is integrable over $U$ and for any $\eta \in U$ the smooth map
  $\displaystyle{\psi_{\eta} : t + i s \longmapsto s \Phi_t^g \left( \eta \right)}$, defined
  in a neighborhood of $0 \in \mathbbm{C}$ is $J$-holomorphic.
  
  (b) The components $S_k$ satisfy  $S_1 =
  0$,
  \begin{eqnarray*}
    S_k  =  \frac{i}{\left( k + 1 \right) ! k!} \tmop{Sym}_{2, \ldots,
    k + 1} \Theta_k \left( g \right),
  \end{eqnarray*}
  for all $k \geqslant 2$, with $\Theta_2 \left( g \right) \assign 2 R^g$ and with
  \begin{eqnarray*}
    \Theta_k \left( g \right)  \assign  - 2 i ( i d_1^{\nabla^g})^{k - 3} \left( \nabla^g R^g \right)_2   +  \sum_{r = 4}^k r ! \sum_{p = 2}^{r - 2}
    ( i d_1^{\nabla^g})^{k - r} \left( p S_p \wedge_1 S_{r - p} \right),
  \end{eqnarray*}
  for all $k \geqslant 3$ and the equations
  $\tmop{Circ} \tmop{Sym}_{3, \ldots, k + 1} \Theta_k \left( g \right) = 0$, are satisfied
  for all $k \geqslant 4$.
\end{corollary}
It has been 64 years since the existence of complex
structures on Grauert Tubes was proven for the first
time by Bruhat-Whitney \cite{Br-Wh}. 
Still, up to now, the explicit form of the Taylor expansion has remained mysterious.
This is finally clarified in the main theorem 1.5 in \cite{Pal-Sal}, which is based on the statement of corollary \ref{MainCoroll}.
Indeed in \cite{Pal-Sal}, theorem 1.5,  we obtain a rather simple and explicit global expression for the complex
structure on Grauert tubes. 

The expression in theorem 1.5 in  \cite{Pal-Sal} (see also theorem 1.6 there for a more general statement) is important for applications to analytic micro local analysis
over manifolds. It allows indeed an explicit global construction of the complex extension of a
given global Fourier integral operator defined on a real analytic manifold.

The expression in theorem 1.5 in \cite{Pal-Sal} allows also to perform useful explicit global intrinsic operator
computations in the sense of \cite{Pali}. In more explicit terms, given a global intrinsic section over the
Grauert tube, an explicit formula for the complex structure such as the one in theorem 1.5 in \cite{Pal-Sal}, allows to
determine if the section is holomorphic or not. 

The proof of corollary \ref{MainCoroll} will be given in the sub-section \ref{Prooof MainCoroll}. In the case $\left( M, g \right)$ is a compact real 
analytic Riemannian manifold, the complex structure in the statement of corollary \ref{MainCoroll} exist thanks to
the work of Guillemin-Stenzel \cite{Gu-St}, Lempert \cite{Lem}, Lempert-Sz\"oke \cite{Le-Sz1, Le-Sz2},
Sz\"oke \cite{Szo1, Szo2} as well as Bielawski \cite{Bie}. Thus in this case the integrability conditions
  \begin{equation*}
  I_k:=\tmop{Circ} \tmop{Sym}_{3, \ldots, k + 1} \Theta_k \left( g \right) = 0,
\end{equation*}
in the statement of corollary \ref{MainCoroll} are satisfied
  for all $k \geqslant 4$. 
We notice in particular that in the case $k=4$, the equation $I_4:=\tmop{Circ} \tmop{Sym}_{3, 4, 5}
\Theta_4 \left( g \right) = 0$, expands out to
\begin{equation}
  \label{FrstRiemInteg} \tmop{Circ} \tmop{Sym}_{3, 4, 5} \left[ 3
  d_1^{\nabla^g} \left( \nabla^g R^g \right)_2 - 2 \tilde{R}^g
  \wedge_1 \tilde{R}^g  \right] = 0\,,
\end{equation}
with $\tilde{R}^g \assign \tmop{Sym}_{2, 3} R^g$. We will show in a quite general set-up that the
previous equation is an identity. We have indeed the following result which shows the vanishing of $I_4$.
\begin{proposition}\label{Vanish}
  Let $\nabla$ be a torsion free complex covariant derivative operator acting
  on the smooth sections of the bundle $\mathbbm{C}T_M$ with curvature operator
  $R^{\nabla} \left( \cdot, \cdot \right) \cdot \equiv R^{\nabla} \left(
  \cdot, \cdot, \cdot \right)$. Let $\tilde{R}^{\nabla} \assign \tmop{Sym}_{2,
  3} R^{\nabla}$.
  Then
  \begin{equation}
    \label{KeyFrstIntegReduc} \tmop{Circ} \tmop{Sym}_{3, 4, 5} \left[ 3
    d_1^{\nabla} \left( \nabla R^{\nabla} \right)_2 - 2
    \tilde{R}^{\nabla} \wedge_1 \tilde{R}^{\nabla} \right] =0 \,.
  \end{equation}
\end{proposition}
The proof will be provided in section \ref{Vanishing4}. In subsection 5.1 in \cite{Pal-Sal} we provide a shorter proof of the vanishing of $I_4$ in proposition \ref{Vanish} by using some more advanced combinatorial techniques. In subsection 5.2 in \cite{Pal-Sal} we show also the vanishing of $I_5$. 
Using computer algebra (see sections 2 and 5 in \cite{Pa-Sa-Ge}) we can show the vanishing of $I_k$ for $k=4,...7$. In section 5 in \cite{Pa-Sa-Ge} we use the explicit expression in theorems 1.5, 1.6 in \cite{Pal-Sal} and we observe that in the case $k=7$, the computer perform the computation in approximately one second, but we expect that the case $k=8$ would take a computation of approximately two weeks. We feel confident at this point to formulate the following conjecture. 
\begin{Conjecture}\label{MainConjecture}
Let $M$ be a smooth manifold and let $\nabla$ be a torsion free complex covariant derivative operator acting
  on the smooth sections of the bundle $\mathbbm{C}T_M$. Then the sequence of tensors $S_k \in C^{\infty} \left( M, T^{\ast}_M
  \otimes_{_{\mathbb{R}}} S^k T^{\ast}_M \otimes_{_{\mathbb{R}}} \mathbb{C}T_M
  \right)$, $k \geqslant 2$, defined by the inductive rule
  \begin{eqnarray*}
    S_k & := & \frac{i}{(k + 1) ! k!} \tmop{Sym}_{2, \ldots, k + 1}
    \Theta^{\nabla}_k,
  \end{eqnarray*}
  with $\Theta^{\nabla}_2 \assign 2 R^{\nabla}$ and with
  \begin{eqnarray*}
    \Theta^{\nabla}_k & \assign & - 2 i (i d_1^{\nabla})^{k - 3} (\nabla
    R^{\nabla})_2 + \sum_{r = 4}^{k } r ! \sum_{p = 2}^{r - 2} (i
    d_1^{\nabla})^{k  - r} (p S_p \wedge_1 S_{r - p}),
  \end{eqnarray*}
  for all $k \geqslant 3$,  satisfies the identities
  \begin{eqnarray*}
  I_k:=\tmop{Circ} \tmop{Sym}_{3, \ldots, k + 1} \Theta^{\nabla}_k = 0  ,
  \end{eqnarray*}
 for all $k \geqslant 4$. 
 \end{Conjecture}

A general mathematical proof for the vanishing of all the integrability conditions $I_k$ is part of a long and difficult work in progress. A corollary of the solution of the above conjecture and of the main theorem \ref{Maintheorem} will be the following striking result which allow canonical construction of maximal totally real embeddings.

\begin{corollary}\label{CanonicalEmbedding}
$\left( \text{\tmstrong{Canonical maximal totally real embeddings}}\right)$. 
  
  Let $M$ be a real analytic manifold and let $\nabla$ be a torsion free
  complex covariant derivative operator acting on the real analytic sections
  of the complexified tangent bundle $\mathbbm{C}T_M$. Then there exists an
  open neighborhood $U \subseteq T_M$ of the image of the zero section with connected fibers
  and a fiberwise real-analytic section $S$ of $\pi^{\ast} \tmop{End} \left(
  T_M \right)$ over $U$ with fiberwise Taylor expansion at the origin
  \begin{eqnarray*}
    S_{\eta} \cdot \xi & = & \sum_{k \geqslant 2} S_k \left( \xi, \eta^k
    \right),
  \end{eqnarray*}
  for any $\eta \in U$ and any $\xi \in T_{M, \pi \left( \eta \right)}$, with
  $S_k \in C^{\infty} \left( M, T^{\ast}_M
  \otimes_{_{\mathbbm{R}}} S^k T^{\ast}_M \otimes_{_{\mathbbm{R}}}
  \mathbbm{C}T_M \right)$ for all $k \geqslant 2$, $\left( \right.$we denote
  by $\eta^k \assign \eta^{\times k} \in T^{\oplus k}_{M, \pi \left( \eta
  \right)} \left. \right)$ given by the recursive formula
  \begin{eqnarray*}
    S_k & \assign & \frac{i}{\left( k + 1 \right) ! k!} \tmop{Sym}_{2, \ldots,
    k + 1} \Theta^{\nabla}_k,
  \end{eqnarray*}
  with $\Theta^{\nabla}_2 \assign 2 R^{\nabla}$ and
  \begin{eqnarray*}
    \Theta^{\nabla}_k  \assign  - 2 i \left( i d_1^{\nabla^{}} \right)^{k -
    3} \left( \nabla R^{\nabla} \right)_2    +   \sum_{r = 4}^k r ! \sum_{p = 2}^{r - 2}
    \left( i d_1^{\nabla} \right)^{k - r} \left( p S_p \wedge_1 S_{r - p} \right),
  \end{eqnarray*}
  for all $k \geqslant 3$, such that $J_A$ with
  \begin{eqnarray*}
    \overline{A} & \assign & - i T\mathbbm{I}_{T_M} + H^{\nabla} - T S,
  \end{eqnarray*}
  is an $M$-totally real complex structure over $U$ which is real-analytic over
  $U$. 
\end{corollary}
Indeed in the statement of the main theorem \ref{Maintheorem} we set $ \sigma_k=0$ for all $k \geqslant 2$ and we identify the torsion free  complex covariant derivative $\nabla^{S_1}$ with the arbitrary torsion free complex covariant derivative $\nabla$ in the statements of corollary \ref{CanonicalEmbedding} and conjecture 1. Then the integrability equations in the statement of the main theorem \ref{Maintheorem} reduce to the identities $I_k=0$, for all $k \geqslant 4$ in the statement of the conjecture 1.

We notice that the notation $H^{\nabla}$ in the above definition of the
section $A$ is slightly abusive. We mean there by $H^{\nabla}$ the restriction
to $T_M$ of the horizontal map over $\mathbbm{C}T_M$ associated to
the complex covariant derivative operator $\nabla$. We must observe here the obvious
inclusion $T_{\mathbbm{C}T_M \mid T_M} \subset \mathbbm{C}T_{T_M}$.

The expression of $S_k$ above can an should be replaced with the explicit global
expression in the theorems 1.5 and 1.6 in \cite{Pal-Sal}. That expression shows that in the
case $\left( M, \nabla \right)$ with smooth regularity we can assume weaker
conditions on the growth of the covariant derivatives of the curvature and
still obtain convergence along the fibers. 

We obtain in this more general
setting a canonical $M$-totally real complex structure over $U$ which is
real-analytic along the fibers of $U$. This is sufficient for the
applications to micro local analytic analysis over manifolds. 

We wish to point
out that in the general setting of a torsion free complex covariant derivative
operator $\nabla$ acting on the sections of the complexified tangent bundle
$\mathbbm{C}T_M$ there are no geodesics associated to $\nabla$. (Cauchy's
existence theorem does not apply). 

Therefore there exist no geodesic flow
associated to $\nabla$ and the Jacobi field techniques of the authors \cite{Gu-St, Lem, Le-Sz1, Le-Sz2, Szo1, Szo2, Bie} do not apply. 

We
wish also to point out that in mathematics and in theoretical physics there
are many important natural complex differential operators that are defined via
complex connections as above.

The set up of corollary \ref{MainCoroll} is inspired by the articles \cite{Gu-St, Lem, Le-Sz1, Le-Sz2, Szo1, Szo2}. 
The genesis of their approach will be reminded in sub-section \ref{SymplecticAppr} and is needed for the proof of corollary \ref{MainCoroll}.

The long series of articles due to Guillemin-Stenzel \cite{Gu-St}, Lempert \cite{Lem}, Lempert-Sz\"oke \cite{Le-Sz1, Le-Sz2}\\
Sz\"oke \cite{Szo1, Szo2}, Burns \cite{Bu1, Bu2}, 
Burns-Halverscheid-Hind \cite{BHH} as well as Aslam-Burns-Irvine \cite{ABI}  are inspired by the fundamental work of Grauert \cite{Gra}.

Their existence results are needed in a crucial way in analytic
micro-local analysis, in pluri-potential theory (see the work by Zelditch \cite{Zel}) as well as in Hamiltonian
dynamics and in geometric quantization (see the work by Morao-Nunes \cite{Mo-Nu} and Hall-Kirwin \cite{Ha-Ki}).

\section{General connections over vector bundles}

\subsection{Basic definitions}

\begin{definition}
  Let $\left( E, \pi_E, M \right)$ be a smooth vector bundle over a manifold
  $M$. A connection form over $E$ is a section $\gamma \in C^{\infty} \left(
  E, T^{\ast}_E \otimes T_E \right)$ such that $d \pi_E \cdot \gamma = 0$ and
  $\gamma_{\mid \tmop{Ker} d \pi_E} =\mathbbm{I}_{\tmop{Ker} d \pi_E}$.
\end{definition}

We will denote by $\gamma_{\eta}$ the connection form $\gamma$ evaluated at
the point $\eta \in E$.

\begin{lemma}
  For any connection $\gamma \in C^{\infty} \left( E, T^{\ast}_E \otimes T_E \right)$ the 
  map 
 \begin{equation}\label{connproj}
d_{\eta}  \pi_{E \mid \tmop{Ker} \gamma_\eta} :
  \tmop{Ker} \gamma_{\eta} \longrightarrow T_{M, \pi_E \left( \eta \right)}\,,  
 \end{equation}
is an isomorphism for all $\eta \in E$.
\end{lemma}

\begin{proof}
  The assumption $\gamma_{\mid \tmop{Ker} d \pi_E} =\mathbbm{I}_{\tmop{Ker} d
  \pi_E}$ implies $\gamma \cdot \left( \mathbbm{I}_{T_E} - \gamma \right) =
  0$. Thus $\tmop{Im} \left( \mathbbm{I}_{T_E} - \gamma \right) \subseteq
  \tmop{Ker} \gamma$. Then $\tmop{Im} \left( \mathbbm{I}_{T_E} - \gamma
  \right) = \tmop{Ker} \gamma$. Indeed if $\gamma \left( u \right) = 0$ then
  $u = \left( \mathbbm{I}_{T_E} - \gamma \right) u$. On the other hand we notice that the
  condition $d \pi_E \cdot \gamma = 0$ implies $d \pi_E \cdot \left(
  \mathbbm{I}_{T_E} - \gamma \right) = d \pi_E$ and thus
  \begin{equation}
    \label{projIdiot} d_{\eta} \pi_{E \mid \tmop{Ker} \gamma_{\eta}} \cdot
    \left( \mathbbm{I}_{T_E} - \gamma \right) = d \pi_E \,.
  \end{equation}
  This equality shows that the map (\ref{connproj}) is surjective. The injectivity follows from the fact that if
  $u, v \in \tmop{Ker} \gamma_{\eta}$ and $d_{\eta} \pi_E \left( u - v \right)
  = 0$ then $u - v = \gamma \left( u - v \right) = 0$ by the assumption
  $\gamma_{\mid \tmop{Ker} d \pi_E} =\mathbbm{I}_{\tmop{Ker} d \pi_E}$.
\end{proof}

We denote by $H_{\eta}^{\gamma} \assign (d_{\eta} \pi_{E \mid \tmop{Ker}
\gamma_{\eta}})^{- 1}$ the horizontal map. We deduce the existence of a
section
\begin{eqnarray*}
  H^{\gamma}  =  C^{\infty} \left( E \nocomma, \pi^{\ast}_E T^{\ast}_M
  \otimes T_E \right),
\end{eqnarray*}
such that $d \pi_E \cdot H^{\gamma} =\mathbbm{I}_{\pi^{\ast}_E T_M}$. (We
notice that $d \pi_E \in C^{\infty} \left( E \nocomma, T^{\ast}_E \otimes
\pi^{\ast}_E T_M \right)$). Composing both sides of (\ref{projIdiot}) with
$H_{\eta}^{\gamma}$ we infer
\begin{eqnarray*}
  \gamma = \mathbbm{I}_{T_E} - H^{\gamma} \cdot d \pi_E\,,
\end{eqnarray*}
and the smooth vector bundle decomposition $T_E = \tmop{Ker} d \pi_E \oplus
\tmop{Ker} \gamma$.

The data of a connection form $\gamma$ is equivalent with the data of a
horizontal form $H^{\gamma}$. The connection form is called linear if the
horizontal form $H^{\gamma}$ satisfies
\begin{eqnarray*}
  d_{\left( \eta_1, \eta_2 \right)}  \left( s m_{_E} \right) \cdot
  (H^{\gamma}_{\eta_1} \oplus H^{\gamma}_{\eta_2}) & = & H^{\gamma}_{\eta_1 +
  \eta_2}\,,\\
  &  & \\
  H^{\gamma}_{\lambda \eta} & = & d_{\eta} \left( \lambda \mathbbm{I}_E
  \right) \cdot H^{\gamma}_{\eta}\,,
\end{eqnarray*}
where $s m_{_E} : E \oplus E \longrightarrow E$ is the sum bundle map where
$\eta_1, \eta_2, \eta \in E$ with $\pi_E \left( \eta_1 \right) = \pi_E \left(
\eta_2 \right)$, and $\lambda$ is a scalar.

\begin{definition}
  The {\tmstrong{curvature form}} $\theta^{\gamma} \in C^{\infty} \left( E
  \nocomma, \Lambda^2 T^{\ast}_E \otimes T_E \right)$ of a connection form
  $\gamma$ is defined as
  \begin{eqnarray*}
    \theta^{\gamma} (\xi_1, \xi_2)  \assign  - \gamma [ \left(
    \mathbbm{I}_{T_E} - \gamma \right) \xi_1\,, \left( \mathbbm{I}_{T_E} -
    \gamma \right) \xi_2] \,,
  \end{eqnarray*}
  for all $\xi_1, \xi_2 \in C^{\infty} \left( E \nocomma, T_E \right)$.
\end{definition}
The definition is tensorial. Indeed if $f \in C^{\infty} \left( E, \mathbbm{R}
\right)$ then
\begin{eqnarray*}
  [ \left( \mathbbm{I}_{T_E} - \gamma \right) f \xi_1\,, \left(
  \mathbbm{I}_{T_E} - \gamma \right) \xi_2] 
 & = &
 f [ \left( \mathbbm{I}_{T_E}
  - \gamma \right) \xi_1\,, \left( \mathbbm{I}_{T_E} - \gamma \right) \xi_2] 
\\
\\
  &-& [
  \left( \mathbbm{I}_{T_E} - \gamma \right) \xi_2 \,.\, f] \left(
  \mathbbm{I}_{T_E} - \gamma \right) \xi_1 \,.
\end{eqnarray*}
The conclusion follows from the fact that $\gamma \cdot \left(
\mathbbm{I}_{T_E} - \gamma \right) = 0$. We notice that
\begin{eqnarray*}
  \theta^{\gamma} \in  C^{\infty} \left( E \nocomma, \Lambda^2 \left(
  \tmop{Ker} \gamma \right)^{\ast} \otimes \tmop{Ker} d \pi \right),
\end{eqnarray*}
and such element is uniquely determined by the {\tmstrong{curvature field}}
$\Theta^{\gamma}$ defined as
\begin{eqnarray*}
  \Theta^{\gamma} (\xi_1, \xi_2) \left( \eta \right) \assign  T_{\eta}^{-
  1} \theta_{\eta}^{\gamma} (H^{\gamma}_{\eta} \xi_1, H^{\gamma}_{\eta}
  \xi_2)\,,
\end{eqnarray*}
for all $\xi_1, \xi_2 \in T_{M, \pi_E \left( \eta \right)}$. In the case
$\gamma$ is linear then 
$$
\Theta^{\gamma} \in C^{\infty} \left( M, \Lambda^2
T^{\ast}_M \otimes \tmop{End} \left( E \right) \right),
$$ 
is called the
{\tmstrong{curvature operator}}. The terminology is consistent with the fact
that if we denote by $\nabla^{\gamma}$ the covariant derivative associated to
$\gamma$ then the identity $R^{\nabla^{\gamma}} = \Theta^{\gamma}$ holds,
thanks to lemma \ref{TorsCurv} in the appendix.
\\
\\
{\bf{Parallel transport}}.
Given any horizontal form $\alpha \in C^{\infty} \left( E, \pi^{\ast}
T^{\ast}_M \otimes_{_{\mathbbm{R}}} T_E \right)$ over a vector bundle $E$, the
parallel transport with respect to $\alpha$ is defined as follows. We consider
a smooth curve $c : \left( - \varepsilon, \varepsilon \right) \longrightarrow
M$ and the section $\sigma \in C^1 \left( \left( - \varepsilon, \varepsilon
\right), c^{\ast} E \right)$ which satisfies the equation
\begin{eqnarray*}
  \dot{\sigma} = \left( \alpha \circ \sigma \right) \cdot \dot{c}\,,
\end{eqnarray*}
over $\left( - \varepsilon, \varepsilon \right)$ with $\sigma \left( 0 \right)
= \eta \in E_{c \left( 0 \right)}$. We define the parallel transport map
$\tau^{\alpha}_{c, t} : E_{c \left( 0 \right)} \longrightarrow E_{c \left( t
\right)}$, $t \in \left( - \varepsilon, \varepsilon \right)$ along $c$ with
respect to $\alpha$ as $\tau^{\alpha}_{c, t} \left( \eta \right) = \sigma
\left( t \right)$.

We consider now a $C^1$-vector field $\xi$ over $M$ and let $\varphi_{\xi,
t}$ be the associated $1$-parameter sub-group of transformations of $M$. Let
$\Phi^{\alpha}_{\xi, t} : E \longrightarrow E$ be the parallel transport map
along the flow lines of $\varphi_{\xi, t}$. In equivalent terms the map
$\Phi^{\alpha}_{\xi, t}$ is determined by the ODE
\begin{eqnarray*}
  \dot{\Phi}^{\alpha}_{\xi, t} & = & \left( \alpha \circ \Phi^{\alpha}_{\xi,
  t} \right) \cdot \left( \xi \circ \varphi_{\xi, t} \circ \pi_E \right),
\end{eqnarray*}
with initial condition $\Phi^{\alpha}_{\xi, 0} \equiv \mathbbm{I}_E$. We
observe that by definition of parallel transport, the map $\Phi^{\alpha}_{\xi,
t}$ satisfies $\pi_E \circ \Phi^{\alpha}_{\xi, t} = \varphi_{\xi, t} \circ
\pi_E$. This follows also from the equalities
\begin{eqnarray*}
  \left( d \pi_E \circ \Phi^{\alpha}_{\xi, t} \right) \cdot
  \dot{\Phi}^{\alpha}_{\xi, t} & = & \xi \circ \varphi_{\xi, t} \circ \pi_E\\
  &  & \\
  & = & \dot{\varphi}_{\xi, t} \circ \pi_E .
\end{eqnarray*}
Moreover the vector field $\Xi^{\alpha} \assign \alpha \cdot \left( \xi \circ
\pi_E \right)$ over $E$ satisfies $\dot{\Phi}^{\alpha}_{\xi, t} = \Xi^{\alpha}
\circ \Phi^{\alpha}_{\xi, t}$. Indeed
\begin{eqnarray*}
  \Xi^{\alpha} \circ \Phi^{\alpha}_{\xi, t} & = & \left( \alpha \circ
  \Phi^{\alpha}_{\xi, t} \right) \cdot \left( \xi \circ \pi_E \circ
  \Phi^{\alpha}_{\xi, t} \right)\\
  &  & \\
  & = & \left( \alpha \circ \Phi^{\alpha}_{\xi, t} \right) \cdot \left( \xi
  \circ \varphi_{\xi, t} \circ \pi_E \right) .
\end{eqnarray*}
We deduce that $t \longmapsto \Phi^{\alpha}_{\xi, t}$ is also a $1$-parameter
sub-group of transformations of $E$.

\subsection{The geometric meaning of the curvature field}

The following result provides a clear geometric meaning of the curvature
field.

\begin{lemma}
  Let $\left( E, \pi_E, M \right)$ be a smooth vector bundle over a manifold
  $M$ and consider a horizontal form $\alpha \in C^{\infty} \left( E,
  \pi^{\ast} T^{\ast}_M \otimes_{_{\mathbbm{R}}} T_E \right)$ over bundle $E$.
  Then the curvature field $\Theta^{\alpha}$ associated to $\alpha$ satisfies
  \begin{eqnarray*}
    \Theta^{\alpha} (\xi_1, \xi_2) \left( \eta \right) & = & T_{\eta}^{- 1}
    \frac{\partial^2}{\partial t \partial s} _{\mid_{t = s = 0}}  \left(
    \Phi^{\alpha}_{\xi_1, - s} \circ \Phi^{\alpha}_{\xi_2, - t} \circ
    \Phi^{\alpha}_{\xi_1, s} \circ \Phi^{\alpha}_{\xi_2, t} \left( \eta
    \right)_{_{_{_{}}}} \right) .
  \end{eqnarray*}
  for any $\xi_1, \xi_2 \in C^{\infty} \left( M, T_M \right)$ such that
  $[\xi_1, \xi_2] \equiv 0$ and for any $\eta \in E$.
\end{lemma}

\begin{proof}
  We observe first that if we have a family of transformations $\left( \Psi_s
  \right)_s$ over a manifold with $\Psi_0 = \tmop{id}$ and a curve
  {\tmem{$c$}} then
  \begin{eqnarray*}
    \frac{d}{d s} _{\mid_{s = 0}} \Psi_s \left( c_s \right) & = & \dot{\Psi}_0
    \left( c_0 \right) \noplus + d \Psi_0 \left( \dot{c}_0 \right)\\
    &  & \\
    & = & \dot{\Psi}_0 \left( c_0 \right) \noplus + \dot{c}_0 .
  \end{eqnarray*}
  Applying the last equality to $\Psi_s = \varphi_{\xi_2, - s}$ and $c_s
  \assign \varphi_{\xi_1, - t} \circ \varphi_{\xi_2, s} \circ \varphi_{\xi_1,
  t}$, we infer
  \begin{eqnarray*}
    \frac{d}{d s}_{\mid_{s = 0}}  \left( \varphi_{\xi_2, - s} \circ
    \varphi_{\xi_1, - t} \circ \varphi_{\xi_2, s} \circ \varphi_{\xi_1, t}
    \right) & = & - \xi_2 \noplus + \frac{d}{d s} _{\mid_{s = 0}}  \left(
    \varphi_{\xi_1, - t} \circ \varphi_{\xi_2, s} \circ \varphi_{\xi_1, t}
    \right),
  \end{eqnarray*}
  and thus
  \begin{eqnarray*}
    {}[\xi_1, \xi_2] & = & \frac{d}{d t} _{\mid_{t = 0}}  \frac{d}{d s}
    _{\mid_{s = 0}}  \left( \varphi_{\xi_1, - t} \circ \varphi_{\xi_2, s}
    \circ \varphi_{\xi_1, t} \right)\\
    &  & \\
    & = & \frac{d}{d t} _{\mid_{t = 0}}  \frac{d}{d s} _{\mid_{s = 0}} 
    \left( \varphi_{\xi_2, - s} \circ \varphi_{\xi_1, - t} \circ
    \varphi_{\xi_2, s} \circ \varphi_{\xi_1, t} \right) .
  \end{eqnarray*}
  In a similar way
  \begin{eqnarray*}
    {}[\Xi^{\alpha}_2, \Xi^{\alpha}_1] & = & \frac{d}{d t} _{\mid_{t = 0}} 
    \frac{d}{d s} _{\mid_{s = 0}}  \left( \Phi^{\alpha}_{\xi_1, - s} \circ
    \Phi^{\alpha}_{\xi_2, - t} \circ \Phi^{\alpha}_{\xi_1, s} \circ
    \Phi^{\alpha}_{\xi_2, t} \right),
  \end{eqnarray*}
  with $\Xi_j^{\alpha} \assign \alpha \cdot \left( \xi_j \circ \pi_E \right)$,
  $j = 1, 2$. Let $\eta \in E_p$ and observe that 
  $$
  \Phi^{\alpha}_{\xi_1, - s}
  \circ \Phi^{\alpha}_{\xi_2, - t} \circ \Phi^{\alpha}_{\xi_1, s} \circ
  \Phi^{\alpha}_{\xi_2, t} (\eta) \in E_p\,,
  $$ 
  for all parameters $t, s$, since
  $\varphi_{\xi_1, - s} \circ \varphi_{\xi_2, - t} \circ \varphi_{\xi_1, s}
  \circ \varphi_{\xi_2, t} \left( p \right) = p$ thanks to the assumption
  $[\xi_1, \xi_2] \equiv 0$.
We conclude the required
  geometric identity
\end{proof}

\subsection{Comparison of the curvature fields of two connections}

We consider now two connection forms $\gamma_j$, $j = 1, 2$ over $E$ and let
$\alpha_j \assign H^{\gamma_j}$ be the corresponding horizontal forms. The
fact that $d \pi_E \left( \alpha_1 - \alpha_2 \right) = 0$ implies that there
exist a section
\begin{eqnarray*}
  B  \assign  T^{- 1} \left( \alpha_1 - \alpha_2 \right) \in C^{\infty}
  \left( E, \pi_E^{\ast} \left( T^{\ast}_M \otimes E \right)_{_{}}
  \right),
\end{eqnarray*}
which satisfies
\begin{eqnarray*}
  \gamma_1 = \gamma_2 - T B \cdot d \pi_E\, .
\end{eqnarray*}
We want to compare the curvature fields $\Theta_j \assign \Theta^{\gamma_j}$.
We will denote by abuse of notation $\alpha_j \xi \equiv \alpha_j \cdot
\left( \xi \circ \pi_E \right)$ and $B \xi \equiv B \cdot \left( \xi \circ
\pi_E \right)$ for any $\xi \in C^{\infty} \left( M, T_M \right)$.

\begin{lemma}
  In the above set up, the identity
  \begin{eqnarray}
    \Theta_1 \left( \xi_1, \xi_2 \right) & = & \left( \Theta_2 - B \neg D B
    \right) \left( \xi_1, \xi_2 \right)\nonumber\\\nonumber
    &  & \\
    & - & T^{- 1} \left( \left[ \alpha_2 \xi_1, T B_{_{_{_{}}}} \xi_2 \right]
    - \left[ \alpha_2  \xi_2 , T B \xi_1 \right] \right) \noplus +
    B \left[ \xi_1, \xi_2 \right],\label{ComparCurv1,2}
  \end{eqnarray}
  holds for any $\xi_1, \xi_2 \in C^{\infty} \left( M, T_M \right)$.
\end{lemma}

\begin{proof}
  We notice first the equalities
  \begin{eqnarray*}
    T \Theta_1 \left( \xi_1, \xi_2 \right) & = & \theta^{\gamma_1} (\alpha_1
    \xi_1, \alpha_1 \xi_2)\\
    &  & \\
    & = & - \gamma_1 \left[ \alpha_1 \xi_1, \alpha_{1_{_{_{}}}} \xi_2
    \right]\\
    &  & \\
    & = & - \gamma_2 \left[ \alpha_1 \xi_1, \alpha_{1_{_{_{}}}} \xi_2 \right]
    + T B \cdot d \pi_E  \left[ \alpha_1 \xi_1, \alpha_{1_{_{_{}}}} \xi_2
    \right]\\
    &  & \\
    & = & T \Theta_2 \left( \xi_1, \xi_2 \right)\\
    &  & \\
    & - & \gamma_2 \left( \left[ \alpha_2 \xi_1, T B \xi_2
    \right] + \left[ T B \xi_1, \alpha_2 \xi_2 \right] + \left[ T
    B \xi_1, T B \xi_2 \right] \right)\\
    &  & \\
    & + & T B \left[ \xi_1, \xi_2 \right] .
  \end{eqnarray*}
  In the last line we use the well known identity $d \pi_E [\alpha_1 \xi_1,
  \alpha_1 \xi_2] = [\xi_1, \xi_2] \circ \pi_E$, which follows from the fact
  that $d \pi_E \alpha_1 \xi_j = \xi_j \circ \pi_E$, $j = 1, 2$. Let now
  $\Phi_{T B \xi_2, t}$ be the $1$-parameter sub-group of transformations of
  $E$ associated to the vertical vector field $T B \xi_2$. It satisfies $\pi_E
  \circ \Phi_{T B \xi_2, t} = \pi_E$. Using the standard expression of the Lie
  bracket
  \begin{eqnarray*}
    {}[\alpha_2 \xi_1, T B \xi_2] & = & \frac{d}{d t} _{\mid_{t = 0}} 
    \frac{d}{d s} _{\mid_{s = 0}}  \left( \Phi^{\alpha_2}_{\xi_1, - t} \circ
    \Phi_{T B \xi_2, s} \circ \Phi^{\alpha_2}_{\xi_1, t} \right),
  \end{eqnarray*}
  we deduce that this vector field is vertical. In the same way $\left[ T B
  \xi_1, \alpha_2  \xi_2 \right]$ is vertical. It is obvious that
  the vector field $\left[ T B \xi_1, T B_{_{_{_{}}}} \xi_2 \right]$ is also
  vertical. We infer the identity
  \begin{eqnarray*}
    T \Theta_1 \left( \xi_1, \xi_2 \right) & = & T \Theta_2 \left( \xi_1,
    \xi_2 \right) - \left[ T B \xi_1, T B_{_{_{_{}}}} \xi_2 \right]\\
    &  & \\
    & - & \left[ \alpha_2 \xi_1, T B_{_{_{_{}}}} \xi_2 \right] - \left[ T B
    \xi_1, \alpha_2 \xi_2 \right] \noplus + T B \left[ \xi_1,
    \xi_2 \right] .
  \end{eqnarray*}
  The required formula (\ref{ComparCurv1,2}) follows from the identity
  \begin{equation}
    \label{BracketTB}  \left[ T B \xi_1, T B \xi_2 \right] \;=\; T \left( B \neg D
    B \right) \left( \xi_1, \xi_2 \right),
  \end{equation}
  that we show now. We first remind the reader that for any vector space $V$, the
  canonical translation operator $T : C^{\infty} \left( V, V \right)
  \longrightarrow C^{\infty} \left( V, T_V \right)$ defined as $\left( T \xi
  \right) \left( v \right) \assign T_v \xi_v$ is a Lie algebra isomorphism,
  where the Lie algebra structure over $C^{\infty} \left( V, V \right)$ is
  defined by $[\xi, \eta]_v \assign D_v \eta \cdot \xi_v - D_v \xi \cdot
  \eta_v$. Indeed if we define the action of $C^{\infty} \left( V, V \right)$
  over $C^{\infty} \left( V, \mathbbm{R} \right)$ as
  \begin{eqnarray*}
    \left( \xi .f \right) \left( v \right) & \assign & D_v f \cdot \xi_v\\
    &  & \\
    & = & \frac{d}{d t} _{\mid_{t = 0}} f \left( v + t \xi_v \right)\\
    &  & \\
    & = & \left[ \left( T \xi \right) .f \right] \left( v \right),
  \end{eqnarray*}
  then
  \begin{eqnarray*}
    \left( \xi . \eta .f \right) \left( v \right) & = & \frac{d}{d t}
    _{\mid_{t = 0}}  \left( \eta .f \right) \left( v + t \xi_v \right)\\
    &  & \\
    & = & \frac{d}{d t} _{\mid_{t = 0}}  \left( D_{v + t \xi_v} f \cdot
    \eta_{v + t \xi_{v_{_{}}}} \right)\\
    &  & \\
    & = & D_v^2 f \left( \xi_v, \eta_v \right) + D_v f \cdot D_v \eta \cdot
    \xi_v .
  \end{eqnarray*}
  The fact that the bilinear form $D_v^2 f$ is symmetric implies
  \begin{eqnarray*}
    \xi . \eta .f - \eta . \xi .f & = & [\xi, \eta] .f \nosymbol .
  \end{eqnarray*}
  On the other hand by definition
  \begin{eqnarray*}
    T \xi .T \eta .f - T \eta . T \xi .f & = & \xi . \eta .f - \eta . \xi
    .f,\\
    &  & \\
    {}[\xi, \eta] .f \nosymbol & = & T [\xi, \eta] .f \nosymbol .
  \end{eqnarray*}
  We conclude the required identity $[T \xi, T \eta] = T [\xi, \eta]$. We
  apply this remark to our set-up. For any point $p \in M$, we denote by $B
  \xi \left( p \right) \in C^{\infty} \left( E_p, E_p \right)$ the map $\eta
  \in E_p \longmapsto B_{\eta} \xi \left( p \right) \in E_p$ and we denote by
  $T B \xi \left( p \right) \in C^{\infty} \left( E_p, T_{E_p} \right)$ the
  section $\eta \in E_p \longmapsto T_{\eta} B_{\eta} \xi \left( p \right) \in
  T_{E_p, \eta}$. Then for any $\eta \in E_p$
  \begin{eqnarray*}
    \left[ T B \xi_1, T B \xi_2 \right]_{\eta} & = & \left[ T B \xi_1 \left( p
    \right), T B \xi_2 \left( p \right) \right]_{\eta}\\
    &  & \\
    & = & T_{\eta}  \left[ B \xi_1 \left( p \right), B \xi_2 \left( p
    \right) \right]_{\eta} \\
    &  & \\
    & = & T_{\eta}  \left[ D_{\eta} B \left( B_{\eta} \xi_1 \left( p \right)
    \right) \xi_2 \left( p \right) - D_{\eta} B \left( B_{\eta} \xi_2 \left( p
    \right) \right) \xi_1 \left( p \right) \right],
  \end{eqnarray*}
  which shows (\ref{BracketTB}).
\end{proof}

We notice now that for any covariant derivative $\nabla$ over $E$, the
identity (\ref{BraketConn}) can be expressed as
\begin{equation}
  \label{BraketConnMod}  \left[ H^{\nabla} \xi, T \, \pi_E^{\ast}
  s \right] = T \, \pi_E^{\ast} \left( \nabla_{\xi} s \right),
\end{equation}
for any vector field $\xi \in C^{\infty} \left( M, T_M \right)$ and any
section $s \in C^{\infty} \left( M, E \right)$. We need to show the following
more general formula.

\begin{lemma}
  Let $\left( E, \pi_E, M \right)$ be a smooth vector bundle over a manifold
  $M$ and let $\nabla$ be a covariant derivative operator acting on the smooth
  sections of $E$. Then the equality holds
  \begin{equation}
    \label{BraketConnGen}  \left[ H^{\nabla} \xi, T \, \sigma
    \right] = T \nabla_{H^{\nabla} \xi}^{\pi_E} \sigma\,,
  \end{equation}
  for any vector field $\xi \in C^{\infty} \left( M, T_M \right)$ and for any
  section $\sigma \in C^{\infty} \left( E, \pi_E^{\ast} E \right)$.
\end{lemma}

We observe that (\ref{BraketConnGen}) implies (\ref{BraketConnMod}), since
$\nabla_{H^{\nabla} \xi}^{\pi_E} \sigma = \pi_E^{\ast} \left( \nabla_{\xi} s
\right)$, thanks to the functorial property (\ref{functConn}).

\begin{proof}
  In order to show the identity (\ref{BraketConnGen}) we notice first that the assumption
  $\sigma \in C^{\infty} \left( E, \pi_E^{\ast} E \right)$ means that $\sigma$
  is a map $\sigma : E \longrightarrow E$ such that $\pi_E \circ \sigma =
  \pi_E$. Then the $1$-parameter subgroup of transformations of $E$ associated
  to the vector field $T \sigma$ satisfies $\Phi_{T \sigma, t} \left( \eta
  \right) = \eta + t \sigma \left( \eta \right)$. Moreover with the notation
  in the proof of identity (\ref{BraketConn})
  \begin{eqnarray*}
    \left[ H^{\nabla} \xi_{_{_{_{}}}}, T \, \sigma \right] & = & \frac{d}{d t}
    _{\mid_{t = 0}}  \frac{d}{d s} _{\mid_{s = 0}}  \left( \Phi_{\xi, - t}
    \circ \Phi_{T \sigma, s} \circ \Phi_{\xi, t} \right) .
  \end{eqnarray*}
  The fact that $\Phi_{\xi, - t}$ is linear on the fibers of $E$ implies
  \begin{eqnarray*}
    \Phi_{\xi, - t} \circ \Phi_{\Sigma, s} \circ \Phi_{\xi, t} & = &
    \Phi_{\xi, - t} \left[ \Phi_{\xi, t} + s \sigma \circ \Phi_{\xi, t}
    \right]\\
    &  & \\
    & = & \mathbbm{I}_E \noplus + s \Phi_{\xi, - t} \cdot \sigma \circ
    \Phi_{\xi, t} \,.
  \end{eqnarray*}
  We infer
  \begin{eqnarray*}
    \frac{d}{d s} _{\mid_{s = 0}}  \left( \Phi_{\xi, - t} \circ \Phi_{T
    \sigma, s} \circ \Phi_{\xi, t} \right) \left( \eta \right) & = & T_{\eta}
    \Phi_{\xi, - t} \cdot \sigma \circ \Phi_{\xi, t} \left( \eta \right),
  \end{eqnarray*}
  for any $\eta \in E_p$. We observe that $\sigma \circ \Phi_{\xi, t} \left(
  \eta \right) \in E_{\varphi_{\xi, t} \left( p \right)}$. Indeed using the
  property $\pi_E \circ \sigma = \pi_E$ we deduce
  \begin{eqnarray*}
    \pi_E \circ \sigma \circ \Phi_{\xi, t} \left( \eta \right) & = & \pi_E
    \circ \Phi_{\xi, t} \left( \eta \right)\\
    &  & \\
    & = & \varphi_{\xi, t} \left( p \right) .
  \end{eqnarray*}
  We remind now that if $t \longmapsto \eta_t \in E$ is a smooth curve such
  that $c_t \assign \pi_E \left( \eta_t \right)$ then
  \begin{eqnarray*}
    T_{\eta_0}^{- 1} \gamma^{\nabla}_{\eta_0}  \dot{\eta}_0 & = & \frac{d}{d
    t} _{\mid_{t = 0}}  \left( \tau^{- 1}_{c, t} \eta_t \right),
  \end{eqnarray*}
  thanks to formula (\ref{covpar}). We apply the previous identity to the
  curve $\eta_t \assign \sigma \circ \Phi_{\xi, t} \left( \eta \right) \in
  E_{\varphi_{\xi, t} \left( p \right)}$. We obtain
  \begin{eqnarray*}
    T_{\sigma \left( \eta \right)}^{- 1} \gamma_{\sigma \left( \eta
    \right)}^{\nabla}  \frac{d}{d t} _{\mid_{t = 0}}  \left[ \sigma \circ
    \Phi_{\xi, t} \left( \eta \right)_{_{_{_{}}}} \right] & = & \frac{d}{d t}
    _{\mid_{t = 0}}  \left[ \Phi_{\xi, - t} \cdot \sigma \circ \Phi_{\xi, t}
    \left( \eta \right)_{_{_{_{}}}} \right]\\
    &  & \\
    & = & T^{- 1}_{\eta} \left[ H^{\nabla} \xi_{_{_{_{}}}}, T \, \sigma
    \right] \left( \eta \right) .
  \end{eqnarray*}
  Moreover
  \begin{eqnarray*}
    \frac{d}{d t} _{\mid_{t = 0}}  \left[ \sigma \circ \Phi_{\xi, t} \left(
    \eta \right)_{_{_{_{}}}} \right] & = & d_{\eta} \sigma \cdot
    \dot{\Phi}_{\xi, 0} \left( \eta \right)\\
    &  & \\
    & = & d_{\eta} \sigma \cdot H_{\eta}^{\nabla} \xi \left( p \right) .
  \end{eqnarray*}
  We conclude the equality
  \begin{eqnarray*}
    T_{\sigma \left( \eta \right)}^{- 1} \gamma_{\sigma \left( \eta \right)}
    d_{\eta} \sigma \cdot H_{\eta}^{\nabla} \xi \left( p \right) & = & T^{-
    1}_{\eta} \left[ H^{\nabla} \xi_{_{_{_{}}}}, T \, \sigma \right] \left(
    \eta \right),
  \end{eqnarray*}
  which represents the required formula (\ref{BraketConnGen}).
\end{proof}

We can show now the following result.

\begin{lemma}
  Let $\left( E, \pi_E, M \right)$ be a smooth vector bundle over a manifold
  $M$ and let $\nabla$ and $\nabla^{T_M}$ be covariant derivative operators
  acting respectively on the smooth sections of the bundles $E$ and $T_M$.
  
  Then for any section $B \in C^{\infty} \left( E, \pi_E^{\ast} \left(
  T^{\ast}_M \otimes E \right) \right)$ the curvature field
  $\Theta^{\alpha}$ of the horizontal form $\alpha \assign H^{\nabla} + T B$
  satisfies
  \begin{equation}
    \label{curvExpans} \Theta^{\alpha} = - H^{\nabla} \neg \nabla^{T^{\ast}_M
    \otimes E \nocomma, \pi_E} B - B \neg D B \noplus - B \tau^{\nabla^{T_M}}
    + R^{\nabla},
  \end{equation}
  where $\nabla^{T^{\ast}_M \otimes E, \pi_E}$ is the covariant derivative
  acting on the smooth sections of the bundle $\pi_E^{\ast} \left( T^{\ast}_M \otimes E
  \right)$, induced by $\nabla$ and $\nabla^{T_M}$ and where
  $\tau^{\nabla^{T_M}}$ is the torsion form of $\nabla^{T_M}$.
\end{lemma}

\begin{proof}
  In the case $\alpha_2 = H^{\nabla}$ in the identity (\ref{ComparCurv1,2}) we
  can apply the formula (\ref{BraketConnGen}) to the sections $B \xi_j \in
  C^{\infty} \left( E, \pi_E^{\ast} E \right)$. We obtain
  \begin{eqnarray*}
    \Theta_1 \left( \xi_1, \xi_2 \right) & = & \left( R^{\nabla} - B \neg D B
    \right) \left( \xi_1, \xi_2 \right)\\
    &  & \\
    & - & \nabla_{H^{\nabla} \xi_1}^{\pi_E} \left( B \xi_2 \right) +
    \nabla_{H^{\nabla} \xi_2}^{\pi_E} \left( B \xi_1 \right) \noplus + B
    \left[ \xi_1, \xi_2 \right] .
  \end{eqnarray*}
  Using functorial properties of the pull-back we have (with no abuse of
  notation)
  \begin{eqnarray*}
    \nabla_{H^{\nabla} \xi_1}^{\pi_E} \left( B \cdot \pi^{\ast}_E \,\xi_2
    \right) & = & \nabla_{H^{\nabla} \xi_1}^{T^{\ast}_M \otimes E \nocomma,
    \pi_E} B \cdot \pi^{\ast}_E \,\xi_2 + B \cdot \nabla_{H^{\nabla}
    \xi_1}^{T_M, \pi_E} ( \pi^{\ast}_E \,\xi_2)\\
    &  & \\
    & = & \nabla_{H^{\nabla} \xi_1}^{T^{\ast}_M \otimes E \nocomma, \pi_E} B
    \cdot \pi^{\ast}_E\, \xi_2 + B \cdot \pi^{\ast}_E (
    \nabla_{\xi_1}^{T_M} \xi_2 ) \,.
  \end{eqnarray*}
  We conclude by (\ref{ComparCurv1,2}) that if
  $\alpha_1 = \alpha = H^{\nabla} + T B$ then the curvature field
  $\Theta^{\alpha}$ of $\alpha$ satisfies the identity
  \begin{eqnarray*}
    \Theta^{\alpha} \left( \xi_1, \xi_2 \right) & = & \left( R^{\nabla} - B
    \neg D B \right) \left( \xi_1, \xi_2 \right)\\
    &  & \\
    & - & \nabla_{H^{\nabla} \xi_1}^{T^{\ast}_M \otimes E \nocomma, \pi_E} B
    \xi_2 + \nabla_{H^{\nabla} \xi_2}^{T^{\ast}_M \otimes E \nocomma, \pi_E} B
    \xi_1 \noplus - B \tau^{\nabla^{T_M}} \left( \xi_1, \xi_2 \right),
  \end{eqnarray*}
  We infer the required formula (\ref{curvExpans}).
\end{proof}

\section{First reduction of the integrability equations}\label{proofTm1}

{\tmstrong{Proof of theorem \ \ref{GenInteg}.}}

\begin{proof}
  Let $\gamma^A$ be the connection form associated to the horizontal form $A$.
  Then the integrability of $J_A$ is equivalent to the condition
  \begin{equation}
    \label{Aintegrab} \gamma^A [A \xi_1, A \xi_2] = 0,
  \end{equation}
  for all smooth complex vector fields $\xi_1$, $\xi_2$ over $M$. (We remind
  here the use of the abusive notation $A \xi \equiv A \left( \xi \circ \pi
  \right)$). We denote respectively by $\Theta^A$ and $\Theta^{\alpha}$ the
  curvature fields of the horizontal distributions $A$ and $\alpha$. The
  integrability condition (\ref{Aintegrab}) is equivalent to the condition
  $\Theta^A \equiv 0$. Then applying the identity (\ref{ComparCurv1,2}) with
  $\alpha_1 = A$, $\alpha_2 = \alpha$ and separating real and imaginary parts
  we deduce that the integrability of $J_A$ is equivalent to the system
  \begin{equation}
    \label{SpCoolAinteg}  \left\{ \begin{array}{l}
      \Theta^{\alpha} + B \neg D B \;=\; 0\,,\\
      \\
      T B [\xi_1, \xi_2] \; =\; [\alpha \xi_1, T B \xi_2] - [\alpha
      \xi_2, T B \xi_1]\, .
    \end{array} \right.
  \end{equation}
  Let
  $\Gamma \in C^{\infty} \left( U, \pi^{\ast} \tmop{End} \left( T_M \right)
  \right)$ such that $\alpha = H^{\nabla} - T \Gamma$. Using the formula (\ref{curvExpans}) in the case $E = T_M$ and $\nabla =
  \nabla^{T_M}$ we can write the previous equation of the system
  (\ref{SpCoolAinteg}) as
  \begin{eqnarray*}
    H^{\nabla} \neg \nabla^{\tmop{End} \left( T_M \right), \pi} \Gamma -
    \Gamma \neg D \Gamma \noplus \noplus \noplus + \Gamma \tau^{\nabla^{}} + B
    \neg D B + R^{\nabla} = 0 \,.
  \end{eqnarray*}
  We express the second equation of the system (\ref{SpCoolAinteg}) as
  \begin{eqnarray*}
    T B [\xi_1, \xi_2] & = & \left[ H^{\nabla} \xi_1, T B \xi_2
    \right] - [T \Gamma \xi_1, T B \xi_2]\\
    &  & \\
    & - & \left[ H^{\nabla} \xi_2, T B \xi_1 \right] + [T \Gamma
    \xi_2, T B \xi_1] \,.
  \end{eqnarray*}
  Using formula (\ref{BraketConnGen}) we infer
  \begin{eqnarray*}
    B [\xi_1, \xi_2] & = & \nabla_{H^{\nabla} \xi_1}^{\tmop{End} \left( T_M
    \right), \pi} B \xi_2 - \nabla_{H^{\nabla} \xi_2}^{\tmop{End} \left( T_M
    \right) \nocomma, \pi} B \xi_1 \noplus + B \left( \nabla_{\xi_1}^{} \xi_2
    - \nabla_{\xi_2} \xi_1 \right)\\
    &  & \\
    & - & D B \left( \Gamma \xi_1 \right) \xi_2 + D \Gamma \left( B \xi_2
    \right) \xi_1\\
    &  & \\
    & + & D B \left( \Gamma \xi_2 \right) \xi_1 - D \Gamma \left( B \xi_1
    \right) \xi_2\,,
  \end{eqnarray*}
  which can be expressed as
  \begin{eqnarray*}
    H^{\nabla} \neg \nabla^{\tmop{End} \left( T_M \right) \nocomma, \pi} B -
    \Gamma \neg D B \noplus \noplus \noplus - B \neg D \Gamma + B
    \tau^{\nabla} = 0\, .
  \end{eqnarray*}
  We conclude that the system (\ref{SpCoolAinteg}) is equivalent to the system
  \begin{equation}
    \label{HpCoolAinteg}  \left\{ \begin{array}{l}
      H^{\nabla} \neg \nabla^{\tmop{End} \left( T_M \right) \nocomma, \pi}
      \Gamma - \Gamma \neg D \Gamma \noplus \noplus \noplus + \Gamma
      \tau^{\nabla^{}} + B \neg D B + R^{\nabla} \; =\; 0\,,\\
      \\
      H^{\nabla} \neg \nabla^{\tmop{End} \left( T_M \right) \nocomma, \pi} B -
      \Gamma \neg D B \noplus \noplus \noplus - B \neg D \Gamma + B
      \tau^{\nabla} \; =\; 0\, .
    \end{array} \right.
  \end{equation}
  It follows that, using the identification $S = \Gamma + i B$, the system
  (\ref{HpCoolAinteg}) is equivalent to the complex equation
  (\ref{MainInteg}).
\end{proof}

\begin{remark}
  We notice that in the case $\left( \alpha, B \right) = \left( H^{\nabla},
  \mathbbm{I}_{\pi^{\ast} T_M} \right)$, i.e. in the case $J_A =
  J_{H^{\nabla}}$, the system (\ref{HpCoolAinteg}) reduces to
  \[ \left\{ \begin{array}{l}
       R^{\nabla} \; =\; 0\,,\\
       \\
       \tau^{\nabla} \; =\; 0 \,.
     \end{array} \right. \]
  In this way we re-obtain the statement of lemma \ref{affineLm}.
\end{remark}

\begin{lemma}
  \label{IntermediateInteg}Under the assumptions of the theorem
  \ref{Maintheorem} the $M$-totally real almost complex structure $J_A$ is
  integrable over $U$ if and only if
  \begin{eqnarray*}
    S_1 \in C^{\infty} \left( M, S^2 T^{\ast}_M
    \otimes_{_{\mathbbm{R}}} \mathbbm{C}T_M \right), 
  \end{eqnarray*}
$\left( \right.$i.e. $\nabla^{S_1}$ is torsion free$\left. \right)$,
  \begin{equation}
    \label{ZotEq} R^{\nabla^{S_1}} = - 2 i \tmop{Alt}_2 S_2\,,
  \end{equation}
  \begin{equation}
    \label{EqSk}  \left[ d_1^{\nabla^{S_1}} S_k + \sum_{p = 2}^{k - 1} p S_p
    \wedge_1 S_{k - p + 1} + i \left( k + 1 \right) \tmop{Alt}_2 S_{k + 1}
    \right] \left( \xi_1, \xi_2, \eta^k \right) = 0\, .
  \end{equation}
  for all $k \geqslant 2$ and for all $\xi_1, \xi_2, \eta \in T_{M, \pi \left(
  \eta \right)}$.
\end{lemma}

\begin{proof}
Let 
$S\assign T^{-1}(H^{\nabla}-\overline{A}\,)$. In the case the connection $\nabla$ is torsion free the equation
  (\ref{MainInteg}) reduces to
  \begin{equation}
    \label{IntegrabEq} H^{\nabla} \neg \nabla^{\tmop{End} \left( T_M \right)
    \nocomma, \pi} S - S \neg D S \noplus \noplus \noplus + R^{\nabla} = 0\, .
  \end{equation}
 The identification $\mathcal{S}_{k, \eta} \cdot \xi \equiv S_k \left(
  \xi, \eta^k \right)$ shows that $\mathcal{S}_{k, \eta} \in T^{\ast}_{M, \pi
  \left( \eta \right)} \otimes \mathbbm{C} T_{M, \pi \left( \eta \right)}$, i.e.
  \begin{eqnarray*}
    \mathcal{S}_k \in C^{\infty} \left( T_M, \pi^{\ast} \left( T^{\ast}_M
    \otimes_{_{\mathbbm{R}}} \mathbbm{C}T_M \right)_{_{}} \right),
  \end{eqnarray*}
  and
  \begin{equation}
    \label{SSeriesExpa} S = \sum_{k \geqslant 0} \mathcal{S}_k \,.
  \end{equation}
  We remind the reader of the formula
  \begin{eqnarray*}
    \nabla_{H^{\nabla} \xi_1}^{\pi}  \left( \mathcal{S}_k \cdot \xi_2 \right)
    = \nabla_{H^{\nabla} \xi_1}^{\tmop{End} \left( T_M \right) \nocomma,
    \pi} \mathcal{S}_k \cdot \xi_2 +\mathcal{S}_k \cdot \nabla_{\xi_1} \xi_2\,,
  \end{eqnarray*}
  for any vector field $\xi_1$, $\xi_2$ over $M$. On the other hand, by
  definition
  \begin{eqnarray*}
    &&\nabla_{H^{\nabla} \xi_1}^{\pi}  \left( \mathcal{S}_k \cdot \xi_2
    \right)_{\mid \eta}\\
    &  & \\
    & = & T_{S_k \left( \xi_2, \eta^k \right)}^{- 1}
    \gamma_{S_k \left( \xi_2, \eta^k \right)}^{\nabla} d_{\eta} \left(
    \mathcal{S}_k \cdot \xi_2 \right) \left( H^{\nabla} \xi_1 \right) \\
    &  & \\
    & = & T_{S_k \left( \xi_2, \eta^k \right)}^{- 1} \gamma_{S_k \left(
    \xi_2, \eta^k \right)}^{\nabla}  \frac{d}{d t} _{\mid_{t = 0}}  \left[ S_k
    \left( \xi_2 \circ \varphi_{\xi_1, t} \circ \pi \left( \eta \right),
    \Phi_{\xi_1, t} \left( \eta \right)_{_{}}^k \right)_{_{_{_{}}}} \right] .
  \end{eqnarray*}
  Let now $\bf{\eta}$ be the vector field over $\tmop{Im} \left(
  \varphi_{\xi_1, \bullet} \circ \pi \left( \eta \right) \right)$ defined by
  $$
  \bf{\eta} \left( \varphi_{\xi_1, t} \circ \pi \left( \eta \right)
  \right) = \Phi_{\xi_1, t} \left( \eta \right).
  $$ Then
  \begin{eqnarray*}
    \nabla_{H^{\nabla} \xi_1}^{\pi}  \left( \mathcal{S}_k \cdot \xi_2
    \right)_{\mid \eta} & = & \nabla_{\xi_1} \left[ S_k \left( \xi_2,
    {\bf\eta}^k \right) \right]_{\mid \pi \left( \eta
    \right)}\\
    &  & \\
    & = & \nabla_{\xi_1} S_k  \left( \xi_2 \circ \pi \left( \eta
    \right), \eta^k \right) + S_k \left( \nabla_{\xi_1} \xi_{2
    \mid \pi \left( \eta \right)} \nocomma, \eta_{_{}}^k \right),
  \end{eqnarray*}
  since $\nabla_{\xi_1} {\bf\eta}= 0$. We conclude the identity
  \begin{eqnarray*}
    \left( \nabla_{H^{\nabla} \xi_1}^{\tmop{End} \left( T_M \right), \pi}
    \mathcal{S}_k \right)_{\mid \eta} \cdot \xi_2 = \nabla_{\xi_1} S_k 
    \left( \xi_2, \eta^k \right),
  \end{eqnarray*}
  $\xi_1, \xi_2 \in T_{M, \pi \left( \eta \right)}$. We infer the formula
  \begin{equation}
    \label{d1exterior} H^{\nabla} \neg \nabla^{\tmop{End} \left( T_M \right)
    \nocomma, \pi} \mathcal{S}_k = d_1^{\nabla} S_k\,,
  \end{equation}
  We notice now the equalities
  \begin{eqnarray*}
    D_{\eta} \mathcal{S}_k  \left( v \right) \cdot \xi & = & \frac{d}{d t}
    _{\mid_{t = 0}}  \left[ S_k \left( \xi_{_{_{_{}}}}, \left( \eta + t v
    \right)^k \right)_{_{_{_{}}}} \right]\\
    &  & \\
    & = & \sum_{j = 1}^k S_k \left( \xi, \eta^{j - 1}, v, \eta^{k - j}
    \right)\\
    &  & \\
    & = & k S_k  \left( \xi, v, \eta^{k - 1} \right),
  \end{eqnarray*}
  and
  \begin{eqnarray*}
    \left( \mathcal{S}_l \neg D\mathcal{S}_k \right)_{\mid \eta} \left( \xi_1,
    \xi_2 \right) = k S_k \left( \xi_2, \mathcal{S}_{l, \eta} \cdot \xi_1,
    \eta^{k - 1} \right) - k S_k \left( \xi_1, \mathcal{S}_{l, \eta} \cdot
    \xi_2, \eta^{k - 1} \right) .
  \end{eqnarray*}
  We infer the equality
  \begin{equation}
    \label{Exterior1}  \left( \mathcal{S}_l \neg D\mathcal{S}_k \right)_{\mid
    \eta} \left( \xi_1, \xi_2 \right) = - k \left( S_k \wedge_1 S_l \right)
    \left( \xi_1, \xi_2, \eta^{k + l - 1} \right) .
  \end{equation}
  Let $W \subset U$ be any set containing the zero section of $T_M$ such that
  $W \cap T_{M, p}$ is a neighborhood of $0_p$ for any $p \in M$ and such that
  the fiberwise expansion (\ref{SSeriesExpa}) converges over $W \cap T_{M,
  p}$. The fact that by assumption $U \cap T_{M, p}$ is connected implies by
  the fiberwise real analyticity of $S$ that $S$ is a solution of
  (\ref{IntegrabEq}) over $U$ if and only if it satisfies (\ref{IntegrabEq})
  over $W$.
  
  Using (\ref{d1exterior}) we can write the equation (\ref{IntegrabEq}) under
  the form
  \begin{equation}
    \label{IntegrabEqDec}  \sum_{k \geqslant 1} d_1^{\nabla} S_k - \sum_{l, p
    \geqslant 0} \left( \mathcal{S}_l \neg D\mathcal{S}_p \right) + R^{\nabla}
    = 0\,,
  \end{equation}
  over $W$. We decompose the sum
  \begin{eqnarray*}
    &&\sum_{l, p \geqslant 0} \left( \mathcal{S}_l \neg D\mathcal{S}_p \right) \\
    &  & \\
    &= & \sum_{l \geqslant 0, p \geqslant 1} \left( \mathcal{S}_l \neg
    D\mathcal{S}_p \right)\\
    &  & \\
    & = & \sum_{l, p \geqslant 1} \left( \mathcal{S}_l \neg D\mathcal{S}_p
    \right) + i \sum_{k \geqslant 0} \left( \mathbbm{I}_{T_M} \neg
    D\mathcal{S}_{k + 1} \right)\\
    &  & \\
    & = & \sum_{k \geqslant 1} \sum_{p = 1}^k \left( \mathcal{S}_{k - p + 1}
    \neg D\mathcal{S}_p \right) + i \sum_{k \geqslant 0} \left(
    \mathbbm{I}_{T_M} \neg D\mathcal{S}_{k + 1} \right)\\
    &  & \\
    & = & - \sum_{k \geqslant 1} \sum_{p = 1}^k p \left( S_p \wedge_1 S_{k -
    p + 1} \right) - i \sum_{k \geqslant 0} \left( k + 1 \right) \left( S_{k +
    1} \wedge_1 \mathbbm{I}_{T_M} \right),
  \end{eqnarray*}
  thanks to the equality (\ref{Exterior1}). If we denote by $\deg_{\eta}$the
  degree with respect to the fibre variable $\eta \in E_{\pi \left( \eta
  \right)}$ we have
  \begin{eqnarray*}
    \deg_{\eta} d_1^{\nabla} S_k  & = & \deg_{\eta} \left( S_p \wedge_1 S_{k -
    p + 1} \right) = k,\\
    &  & \\
    \deg_{\eta}  \left( S_{k + 1} \wedge_1 \mathbbm{I}_{T_M} \right) & = &
    k,\\
    &  & \\
    \deg_{\eta} R^{\nabla} & = & 1 .
  \end{eqnarray*}
  Thus by homogeneity the equation (\ref{IntegrabEqDec}) is equivalent to the
  countable system
  \begin{equation}
    \label{Denom-integ}  \left\{ \begin{array}{l}
      S_1 \wedge_1 \mathbbm{I}_{T_M} \; =\; 0\,,\\
      \\
      d_1^{\nabla} S_1 + S_1 \wedge_1 S_1 \noplus \noplus \noplus + 2 i S_2
      \wedge_1 \mathbbm{I}_{T_M} + R^{\nabla} \; =\; 0\,,\\
      \\
      \left[ d_1^{\nabla} S_k + \sum_{p = 1}^k p \left( S_p \wedge_1 S_{k - p
      + 1} \right) \noplus \noplus \noplus + i \left( k + 1 \right) S_{k + 1}
      \wedge_1 \mathbbm{I}_{T_M} \right] \left( \xi_1, \xi_2, \eta^k \right)
      \;=\; 0\,, 
      \\
      \\
      \forall  k \geqslant 2\,,\; \forall \xi_1, \xi_2, \eta \in T_M\,.
    \end{array} \right.
  \end{equation}The first equation in the system means
  $S_1 \in C^{\infty} \left( M, S^2 T^{\ast}_M
  \otimes_{_{\mathbbm{R}}} \mathbbm{C}T_M \right)$, i.e. the complex
  connection $\nabla^{S_1}$ is torsion free.
The second equation in the system (\ref{Denom-integ}) rewrites as
  (\ref{ZotEq}).
We show now that the equation for $k \geqslant 2$ in the system
  (\ref{Denom-integ}) rewrites as (\ref{EqSk}). Indeed using the formula
  \begin{eqnarray*}
    \nabla_{\xi}^{\Gamma} \theta \left( v_1, \ldots, v_p \right) & = &
    \nabla_{\xi} \theta \left( v_1, \ldots, v_p \right) + \Gamma \left( \xi,
    \theta \left( v_1, \ldots, v_p \right) \right)\\
    &  & \\
    & - & \sum_{j = 1}^p \theta \left( v_1, \ldots, v_{j - 1}, \Gamma \left(
    \xi, v_j \right), v_{j + 1}, \ldots, v_p \right),
  \end{eqnarray*}
  where $\Gamma \in C^{\infty}( M, T^{\ast, \otimes 2}_M \otimes_{_{\mathbbm{R}}} \mathbbm{C}T_M) $, $\theta \in C^{\infty}
  \left( M, T^{\ast, \otimes p}_M \otimes_{_{\mathbbm{R}}} \mathbbm{C}T_M
  \right)$ and $\xi, v_k \in T_M$, we infer
  \begin{eqnarray*}
    &&d_1^{\nabla^{S_1}} S_k  \left( \xi_1, \xi_2, \eta^k \right) \\
    &  & \\
    & = &
    \nabla_{\xi_1} S_k  \left( \xi_2, \eta^k \right) - \nabla_{\xi_2} S_k 
    \left( \xi_1, \eta^k \right)\\
    &  & \\
    & + & S_1 \left( \xi_1, S_k  \left( \xi_2, \eta^k \right) \right) - S_k 
    \left( S_1  \left( \xi_1, \xi_2 \right), \eta^k \right) - k S_k  \left(
    \xi_2, S_1  \left( \xi_1, \eta \right), \eta^{k - 1} \right)\\
    &  & \\
    & - & S_1 \left( \xi_2, S_k  \left( \xi_1, \eta^k \right) \right) + S_k 
    \left( S_1  \left( \xi_2, \xi_1 \right), \eta^k \right) + k S_k  \left(
    \xi_1, S_1  \left( \xi_2, \eta \right), \eta^{k - 1} \right)\\
    &  & \\
    & = & \left[ d_1^{\nabla} S_k + S_1 \wedge_1 S_k + k_{_{_{_{_{}}}}} S_k
    \wedge_1 S_1 \right] \left( \xi_1, \xi_2, \eta^k \right),
  \end{eqnarray*}
  since $S_1$ is symmetric and $S_k$ is symmetric in the last $k$ variables.
  We conclude (\ref{EqSk}).
\end{proof}

\begin{remark}
  In the case $S_k = 0$, for all $k \geqslant 2$, the previous system reduces
  to the equation
  \begin{equation}
    \label{LinearInteg} d_1^{\nabla} S_1 + S_1 \wedge_1 S_1 + R^{\nabla} = 0 \,.
  \end{equation}
  The equation (\ref{LinearInteg}) means that the complex connection
  $\nabla^{S_1}$ acting on sections of $\mathbbm{C}T_M$ is flat. In the case
  $B =\mathbbm{I}_{\pi^{\ast} T_M}$, the second equation in the system
  (\ref{Denom-integ}) implies
  \begin{eqnarray*}
    d_1^{\nabla} \Gamma_1 + \Gamma_1 \wedge_1 \Gamma_1 \noplus \noplus \noplus
    + R^{\nabla} = 0\,,
  \end{eqnarray*}
  with $\Gamma_1 \assign S_1$. This means that the real connection
  $\nabla^{\Gamma_1}$ is flat.
\end{remark}
\section{Second reduction of the integrability equations}

In this section we will prove the following result.

\begin{proposition}
  \label{MainProp}Under the assumptions of the theorem \ref{Maintheorem} the
  $M$-totally real almost complex structure $J_A$ is integrable over $U$ if and only if
  \begin{eqnarray*}
    S_1 & \in & C^{\infty} \left( M, S^2 T^{\ast}_M
    \otimes_{_{\mathbbm{R}}} \mathbbm{C}T_M \right), \; \text{i.e.
    $\nabla^{S_1}$ is torsion free},\\
    &  & \\
    S_2 & = & S^0_2 + \sigma_2\,,\\
    &  & \\
    S^0_2 \left( \xi_1, \xi_2, \xi_3 \right) & \assign & \frac{i}{6}  \left[
    R^{\nabla^{S_1}} \left( \xi_1, \xi_2 \right) \xi_3 + R^{\nabla^{S_1}}
    \left( \xi_1, \xi_3 \right) \xi_2 \right],\\
    &  & \\
    \sigma_2 & \in & C^{\infty} \left( M_{_{_{_{}}}}, S^3 T^{\ast}_M
    \otimes_{_{\mathbbm{R}}} \mathbbm{C}T_M \right),\\
    &  & \\
    S_3 & = & \frac{i}{3} \nabla^{S_1} \sigma_2 + \frac{1}{4! 3}
    \tmop{Sym}_{2, 3, 4} ( \nabla^{S_1} R^{\nabla^{S_1}} )_2+ \sigma_3\,,
    \\
    &  & \\  
    \sigma_3 & \in & C^{\infty} \left( M_{_{_{_{}}}}, S^4 T^{\ast}_M
    \otimes_{_{\mathbbm{R}}} \mathbbm{C}T_M \right),
  \end{eqnarray*}
$( \nabla^{S_1} R^{\nabla^{S_1}})_2 \left( \xi_1, \xi_2, \xi_3,
    \xi_4 \right)  : = \nabla_{\xi_2}^{S_1} R^{\nabla^{S_1}} \left( \xi_1,
    \xi_3 \right) \xi_4$, for all $\xi_1, \xi_2, \xi_3, \xi_4 \in T_{M, \pi \left( \xi_1 \right)}$  and for all $k \geqslant 3$,
  \begin{eqnarray*}
    \left[ d_1^{\nabla^{S_1}} S_k + \sum_{p = 2}^{k - 1} p S_p \wedge_1 S_{k -
    p + 1} \noplus \noplus \noplus + i \left( k + 1 \right) \tmop{Alt}_2 S_{k
    + 1} \right] \left( \xi_1, \xi_2, \eta^k \right) = 0\,,
  \end{eqnarray*}
  for all $\xi_1, \xi_2, \eta \in T_{M, \pi \left( \eta \right)}$.
\end{proposition}

We first remind the reader that for any complex connection $\nabla$ acting over the
sections of $\mathbbm{C}T_M$ its torsion $\tau^{\nabla}$ satisfies the
identity
\begin{eqnarray*}
  \tau^{\nabla} = d^{\nabla} \mathbbm{I}_{T_M}\,,
\end{eqnarray*}
where $d^{\nabla}$ is the covariant exterior differentiation and 
$\mathbbm{I}_{T_M} \in C^{\infty} \left( M, T^{\ast}_M \otimes T_M \right)$.
Then
\begin{eqnarray*}
  d^{\nabla} \tau^{\nabla}= R^{\nabla} \wedge \mathbbm{I}_{T_M}\,,
\end{eqnarray*}
and
\begin{eqnarray*}
  (R^{\nabla} \wedge \mathbbm{I}_{T_M}) \left( \xi_1, \xi_2, \xi_3 \right) = R^{\nabla} \left( \xi_1, \xi_2 \right) \xi_3 + R^{\nabla} \left( \xi_2,
  \xi_3 \right) \xi_1 + \noplus R^{\nabla} \left( \xi_3, \xi_1 \right) \xi_2 \,.
\end{eqnarray*}
We conclude that if a connection is torsion free then then its curvature
operator satisfies the algebraic Bianchi identity.

We denote by $\tmop{Alt}_p$
the alternating operator (without normalizing
coefficients!) acting on the first $p \geqslant 2$ entries of a tensor,
counted from the left to the right. We notice the following very elementary
fact.
\begin{lemma}
  \label{exactSequence}Let $V$ be a vector space over a field $\mathbbm{K}$ of
  characteristic zero. Then for any integer $p \geqslant 2$, the sequence
  \[ 0 \longrightarrow S^{p + 1} V^{\ast} \longrightarrow V^{\ast} \otimes S^p
     V^{\ast} \xrightarrow{\tmop{Alt}_2 } \Lambda^2 V^{\ast} \otimes S^{p - 1}
     V^{\ast} \xrightarrow{\tmop{Alt}_3 } \Lambda^3 V^{\ast} \otimes S^{p - 2}
     V^{\ast}, \]
  is exact.
\end{lemma}

\begin{proof}
  The equality
  \begin{eqnarray*}
    S^{p + 1} V^{\ast} & = & \tmop{Ker} \left( V^{\ast} \otimes S^p V^{\ast}
    \xrightarrow{\tmop{Alt}_2 } \Lambda^2 V^{\ast} \otimes S^{p - 1} V^{\ast}
    \right),
  \end{eqnarray*}
  is obvious. We show now the equality
\begin{eqnarray}
  &&\tmop{Im} \left( V^{\ast} \otimes S^p V^{\ast}
    \xrightarrow{\tmop{Alt}_2 } \Lambda^2 V^{\ast} \otimes S^{p - 1} V^{\ast}
    \right)\nonumber \\\nonumber
    \\
    &=& \tmop{Ker} \left( \Lambda^2 V^{\ast} \otimes S^{p - 1} V^{\ast}
    \xrightarrow{\tmop{Alt}_3 } \Lambda^3 V^{\ast} \otimes S^{p - 2} V^{\ast}
    \right) . \label{Alt23}
\end{eqnarray}
We show first the inclusion $\subseteq$ in (\ref{Alt23}). We notice the
  equality
  \begin{eqnarray*}
    &&\left( \Lambda^2 V^{\ast} \otimes S^{p - 1} V^{\ast}
    \xrightarrow{\tmop{Alt}_3 } \Lambda^3 V^{\ast} \otimes S^{p - 2} V^{\ast}
    \right) \\
    \\
    &=&\left( \Lambda^2 V^{\ast} \otimes S^{p - 1} V^{\ast}
    \xrightarrow{2 \tmop{Circ}} \Lambda^3 V^{\ast} \otimes S^{p - 2} V^{\ast}
    \right) .
  \end{eqnarray*}
  Let now $\beta \assign \tmop{Alt}_2 \alpha$, with $\alpha \in V^{\ast}
  \otimes S^p V^{\ast}$. Then summing up the two equalities
  \begin{eqnarray*}
    \beta \left( v_1, v_2 ; v_3, v_4, \ldots, v_{p + 1} \right) = \alpha
    \left( v_1 ; v_2, v_3, v_4, \ldots, v_{p + 1} \right) - \alpha \left( v_2
    ; v_1, v_3, v_4, \ldots, v_{p + 1} \right),\\
    &  & \\
    - \beta \left( v_1, v_3 ; v_2, v_4, \ldots, v_{p + 1} \right) = -
    \alpha \left( v_1 ; v_3, v_2, v_4, \ldots, v_{p + 1} \right) + \alpha
    \left( v_3 ; v_1, v_2, v_4, \ldots, v_{p + 1} \right),
  \end{eqnarray*}
  we obtain
  \begin{eqnarray*}
    &  & \beta \left( v_1, v_2 ; v_3, v_4, \ldots, v_{p + 1} \right) - \beta
    \left( v_1, v_3 ; v_2, v_4, \ldots, v_{p + 1} \right)\\
    &  & \\
    & = & -\; \alpha \left( v_2 ; v_3, v_1, v_4, \ldots, v_{p + 1} \right) +
    \alpha \left( v_3 ; v_2, v_1, v_4, \ldots, v_{p + 1} \right)\\
    &  & \\
    & = & - \;\beta \left( v_2, v_3 ; v_1, v_4, \ldots, v_{p + 1} \right),
  \end{eqnarray*}
  which rewrites as
  \begin{eqnarray*}
    &&\beta \left( v_1, v_2 ; v_3, v_4, \ldots, v_{p + 1} \right) \\
    \\
    &+&
    \beta \left( v_2, v_3 ; v_1, v_4, \ldots, v_{p + 1} \right)
    \\
    \\
    &+& 
    \beta \left(
    v_3, v_1 ; v_2, v_4, \ldots, v_{p + 1} \right) \;=\; 0\,,
  \end{eqnarray*}
  i.e. $\tmop{Circ} \beta = 0$, which shows the inclusion $\subseteq$ in
  (\ref{Alt23}). In order to show the reverse inclusion in (\ref{Alt23}) we
  consider $\beta \in \Lambda^2 V^{\ast} \otimes S^{p - 1} V^{\ast}$ with
  $\tmop{Circ} \beta = 0$ and we will prove that $\beta = C_p \tmop{Alt}_2 \alpha$, with 
  $$\alpha \assign \tmop{Sym}_{2, \ldots, p + 1} \beta \in
  V^{\ast} \otimes S^p V^{\ast}\,,
  $$ and with $C_p \assign p / \left( p + 1
  \right) !$. Indeed
  \begin{eqnarray*}
    \frac{1}{\left( p - 1 \right) !} \alpha \left( v_1 ; v_2, \ldots, v_{p +
    1} \right) & = & \sum_{j = 2}^{p + 1} \beta \left( v_1, v_j ; v_2, \ldots,
    \hat{v}_j, \ldots, v_{p + 1} \right)
  \end{eqnarray*}
  and
  \begin{eqnarray*}
    &  & \frac{1}{\left( p - 1 \right) !}  \left( \tmop{Alt}_2 \alpha \right)
    \left( v_1, v_2 ; \ldots, v_{p + 1} \right)\\
    &  & \\
    & = & \frac{1}{\left( p - 1 \right) !} \alpha \left( v_1 ; v_2, \ldots,
    v_{p + 1} \right) - \frac{1}{\left( p - 1 \right) !} \alpha \left( v_2 ;
    v_1, \hat{v}_2, \ldots, v_{p + 1} \right)\\
    &  & \\
    & = & \sum_{j = 2}^{p + 1} \beta \left( v_1, v_j ; v_2, \ldots,
    \hat{v}_j, \ldots, v_{p + 1} \right)\\
    &  & \\
    & - & \sum_{{\begin{array}{l}
      j = 1\\
      j \neq 2
    \end{array}}}^{p + 1} \beta \left( v_2, v_j ; v_1, \hat{v}_2, \ldots,
    \hat{v}_j, \ldots, v_{p + 1} \right)\\
    &  & \\
    & = & \beta \left( v_1, v_2 ; v_3, \ldots, v_{p + 1} \right) + \sum_{j =
    3}^{p + 1} \beta \left( v_1, v_j ; v_2, v_3, \ldots, \hat{v}_j, \ldots,
    v_{p + 1} \right)\\
    &  & \\
    & + & \beta \left( v_1, v_2 ; v_3, \ldots, v_{p + 1} \right) + \sum_{j =
    3}^{p + 1} \beta \left( v_j, v_2 ; v_1, \hat{v}_2, v_3, \ldots, \hat{v}_j,
    \ldots, v_{p + 1} \right)\,.
  \end{eqnarray*}
  Using the circular identity $\tmop{Circ} \beta = 0$, we obtain
  \begin{eqnarray*}
    &  & \frac{1}{\left( p - 1 \right) !}  \left( \tmop{Alt}_2 \alpha \right)
    \left( v_1, v_2 ; \ldots, v_{p + 1} \right)\\
    &  & \\
    & = & 2\, \beta \left( v_1, v_2 ; v_3, \ldots, v_{p + 1} \right) - \sum_{j
    = 3}^{p + 1} \beta \left( v_2, v_1 ; v_j, v_3, \ldots, \hat{v}_j, \ldots,
    v_{p + 1} \right) .
  \end{eqnarray*}
  This combined with the fact that $\beta \in \Lambda^2 V^{\ast} \otimes S^{p
  - 1} V^{\ast}$ implies
  \begin{eqnarray*}
    &  &  \frac{1}{\left( p - 1 \right) !}  \left( \tmop{Alt}_2 \alpha
    \right) \left( v_1, v_2 ; \ldots, v_{p + 1} \right)\\
    &  & \\
    & = & 2\, \beta \left( v_1, v_2 ; v_3, \ldots, v_{p + 1} \right) + \left( p
    - 1 \right) \beta \left( v_1, v_2 ; v_3, \ldots, v_{p + 1} \right)\\
    &  & \\
    & = & \left( p + 1 \right) \beta \left( v_1, v_2 ; \ldots, v_{p + 1}
    \right),
  \end{eqnarray*}
  which shows the required identity.
\end{proof}

A direct consequence of the proof of lemma \ref{exactSequence} is the
following fact.
\begin{corollary}
  \label{KeylmRmSm}Let $R \in C^{\infty} \left( M, \Lambda^2 T^{\ast}_M
  \otimes_{_{\mathbbm{R}}} T^{\ast}_M \otimes_{_{\mathbbm{R}}} \mathbbm{C}T_M
  \right)$ satisfying the algebraic Bianchi identity. Then a tensor $S \in
  C^{\infty} \left( M, T^{\ast}_M \otimes_{_{\mathbbm{R}}} S^2
  T^{\ast}_M \otimes_{_{\mathbbm{R}}} \mathbbm{C}T_M \right)$ satisfies $3 R =
  \tmop{Alt}_2 S$ if and only if $S = \tmop{Sym}_{2, 3} R + \sigma$, with $\sigma \in
  C^{\infty} ( M, S^3 T^{\ast}_M \otimes_{_{\mathbbm{R}}}
  \mathbbm{C}T_M)$.
\end{corollary}
We infer by corollary \ref{KeylmRmSm} that the equation (\ref{ZotEq}) is satisfied by
$S_2 = S^0_2 + \sigma_2$, with
\begin{equation}
  \label{ExprS2} S^0_2 \left( \xi_1, \xi_2, \xi_3 \right) = \frac{i}{6} 
  \left[ R^{\nabla^{S_1}} \left( \xi_1, \xi_2 \right) \xi_3 + R^{\nabla^{S_1}}
  \left( \xi_1, \xi_3 \right) \xi_2 \right],
\end{equation}
and with $\sigma_2 \in C^{\infty} \left( M, S^3 T^{\ast}_M
\otimes_{_{\mathbbm{R}}} \mathbbm{C}T_M \right)$. We consider now the equation
(\ref{EqSk}) for $k = 2$, which writes as
\begin{equation}
  \label{EqS2}  \left[ d_1^{\nabla^{S_1}} S_2 + 3 i \tmop{Alt}_2 S_3 \right]
  \left( \xi_1, \xi_2, \eta^2 \right) = 0 .
\end{equation}
The fact that the tensor 
$$
d_1^{\nabla^{S_1}} S_2 + 3 i \tmop{Alt}_2 S_3\,,
$$ is
symmetric in the last two variables implies that the equation (\ref{EqS2}) is
equivalent to the equation
\[ d_1^{\nabla^{S_1}} S_2 + 3 i \tmop{Alt}_2 S_3 = 0\,, \]
that we can rewrite under the form
\begin{equation}
  \label{EqS2gn} d_1^{\nabla^{S_1}} S^0_2 + 3 i \tmop{Alt}_2 \hat{S}_3 = 0\,,
\end{equation}
with 
$$
\hat{S}_3 \assign S_3 \noplus - \frac{i}{3} \nabla^{S_1} \sigma_2\,.
$$
Then
using the expression (\ref{ExprS2}) we can rewrite equation (\ref{EqS2gn}) in
the explicit form
\begin{eqnarray}
  &  & \nabla^{S_1}_{\xi_1} R^{\nabla^{S_1}} \left( \xi_2, \xi_3 \right)
  \xi_4 + \nabla^{S_1}_{\xi_1} R^{\nabla^{S_1}} \left( \xi_2, \xi_4 \right)
  \xi_3\nonumber
  \\\nonumber
  &  & \\
  & - & \nabla^{S_1}_{\xi_2} R^{\nabla^{S_1}} \left( \xi_1, \xi_3 \right)
  \xi_4 - \nabla^{S_1}_{\xi_2} R^{\nabla^{S_1}} \left( \xi_1, \xi_4 \right)
  \xi_3\nonumber
  \\\nonumber
  &  & \\
  & = & - \;18 \left[ \hat{S}_3 \left( \xi_1, \xi_2, \xi_3, \xi_4 \right) -
  \hat{S}_3 \left( \xi_2, \xi_1, \xi_3, \xi_4 \right) \right] .\label{Rm3Sm}
\end{eqnarray}
We notice that the fact that the complex connection $\nabla^{S_1}$ is torsion
free implies that the tensor $\rho$ given by $\rho \left( \xi_1, \xi_2, \xi_3,
\xi_4 \right) \assign \nabla^{S_1}_{\xi_1} R^{\nabla^{S_1}} \left( \xi_2,
\xi_3 \right) \xi_4$ satisfies the circular identity with respect to the first
and last three entries. Moreover $\rho$ is obviously skew-symmetric with
respect to the variables $\xi_2, \xi_3$.

\begin{lemma}
  \label{AltSymCurvature}Let $\rho$ be a $4$-linear form which satisfies the
  circular identity with respect to the first and last three entries and which
  is skew-symmetric with respect to the second and third variables. Then a
  $4$-linear form $S$ which is symmetric with respect to the last three
  entries satisfies the equation
  \begin{equation}
    \label{EqAlSm} \tmop{Alt}_2 [8 \tmop{Sym}_{3, 4} \rho - S] = 0\,,
  \end{equation}
  if and only if
  \begin{eqnarray*}
    S & = & - 2 \tmop{Sym}_{2, 3, 4} \rho_2 + \sigma\\
    &  & \\
    & = & 2 \tmop{Sym}_{2, 3, 4} \rho_3 + \sigma\,,
  \end{eqnarray*}
  with $\rho_2 \left( \xi_1, \xi_2, \xi_3, \xi_4 \right) : = \rho \left(
  \xi_2, \xi_1, \xi_3, \xi_4 \right)$, with $\rho_3 \left( \xi_1, \xi_2,
  \xi_3, \xi_4 \right) : = \rho \left( \xi_2, \xi_3, \xi_1, \xi_4 \right)$,
  for all $\xi_1, \xi_2, \xi_3, \xi_4 \in T_{M, \pi \left( \xi_1 \right)}$ and
  with $\sigma$ a $4$-linear form which is symmetric with respect to all its
  entries. 
\end{lemma}
\begin{proof}
  We observe first that the assumptions on $\rho$ imply $\tmop{Circ}
  \tmop{Alt}_2 \tmop{Sym}_{3, 4} \rho = 0$. Indeed
  \begin{eqnarray*}
    (\tmop{Alt}_2 \tmop{Sym}_{3, 4} \rho) \left( \xi_1, \xi_2, \xi_3, \xi_4
    \right) & = & \rho \left( \xi_1, \xi_2, \xi_3, \xi_4 \right) + \rho \left(
    \xi_1, \xi_2, \xi_4, \xi_3 \right)\\
    &  & \\
    & - & \rho \left( \xi_2, \xi_1, \xi_3, \xi_4 \right) - \rho \left( \xi_2,
    \xi_1, \xi_4, \xi_3 \right),
  \end{eqnarray*}
  and
  \begin{eqnarray*}
    &  & (\tmop{Circ} \tmop{Alt}_2 \tmop{Sym}_{3, 4} \rho) \left( \xi_1,
    \xi_2, \xi_3, \xi_4 \right)\\
    &  & \\
    & = & \rho \left( \xi_1, \xi_2, \xi_3, \xi_4 \right)_{_1} + \rho \left(
    \xi_1, \xi_2, \xi_4, \xi_3 \right)_{_4} - \rho \left( \xi_2, \xi_1, \xi_3,
    \xi_4 \right)_{_2} - \rho \left( \xi_2, \xi_1, \xi_4, \xi_3 \right)_{_5}\\
    &  & \\
    & + & \rho \left( \xi_2, \xi_3, \xi_1, \xi_4 \right)_{_2} + \rho \left(
    \xi_2, \xi_3, \xi_4, \xi_1 \right)_{_5} - \rho \left( \xi_3, \xi_2, \xi_1,
    \xi_4 \right)_{_3} - \rho \left( \xi_3, \xi_2, \xi_4, \xi_1 \right)_{_6}\\
    &  & \\
    & + & \rho \left( \xi_3, \xi_1, \xi_2, \xi_4 \right)_{_3} + \rho \left(
    \xi_3, \xi_1, \xi_4, \xi_2 \right)_{_6} - \rho \left( \xi_1, \xi_3, \xi_2,
    \xi_4 \right)_{_1} - \rho \left( \xi_1, \xi_3, \xi_4, \xi_2 \right)_{_4},
  \end{eqnarray*}
  where we denote by $\rho \left( \cdot, \cdot, \cdot, \cdot \right)_{_j}$ the
  terms we group together. Using the assumption $\rho$ is skew-symmetric with
  respect to the second and third variables we infer
  \begin{eqnarray*}
    &  & (\tmop{Circ} \tmop{Alt}_2 \tmop{Sym}_{3, 4} \rho) \left( \xi_1,
    \xi_2, \xi_3, \xi_4 \right)\\
    &  & \\
    & = & 2 \rho \left( \xi_1, \xi_2, \xi_3, \xi_4 \right)_{_1} + \rho \left(
    \xi_1, \xi_2, \xi_4, \xi_3 \right)_{_4} + \rho \left( \xi_2, \xi_4, \xi_1,
    \xi_3 \right)_{_5}\\
    &  & \\
    & + & 2 \rho \left( \xi_2, \xi_3, \xi_1, \xi_4 \right)_{_1} + \rho \left(
    \xi_2, \xi_3, \xi_4, \xi_1 \right)_{_5} + \rho \left( \xi_3, \xi_4, \xi_2,
    \xi_1 \right)_{_6}\\
    &  & \\
    & + & 2 \rho \left( \xi_3, \xi_1, \xi_2, \xi_4 \right)_{_1} + \rho \left(
    \xi_3, \xi_1, \xi_4, \xi_2 \right)_{_6} + \rho \left( \xi_1, \xi_4, \xi_3,
    \xi_2 \right)_{_4} .
  \end{eqnarray*}
Using the circular assumptions on $\rho$ we infer
  \begin{eqnarray*}
    &  & (\tmop{Circ} \tmop{Alt}_2 \tmop{Sym}_{3, 4} \rho) \left( \xi_1,
    \xi_2, \xi_3, \xi_4 \right)\\
    &  & \\
    & = & - \rho \left( \xi_1, \xi_3, \xi_2, \xi_4 \right) - \rho \left(
    \xi_2, \xi_1, \xi_3, \xi_4 \right) - \rho \left( \xi_3, \xi_2, \xi_1,
    \xi_4 \right)\\
    &  & \\
    & = & 0 \,.
  \end{eqnarray*}
  Then by the proof of lemma \ref{exactSequence} in the case $p = 3$, we infer
  that a $4$-linear form $S$ which is symmetric with respect to the last three
  entries satisfies the equation (\ref{EqAlSm}) if and only if
  \begin{eqnarray*}
    S & = & \tmop{Sym}_{2, 3, 4} \tmop{Alt}_2 \tmop{Sym}_{3, 4} \rho + \sigma\,,
  \end{eqnarray*}
  with $\sigma$ any $4$-linear form which is symmetric with respect to all its
  entries, satisfies (\ref{EqAlSm}). We write now
  \begin{eqnarray*}
    &  & (\tmop{Sym}_{2, 3, 4} \tmop{Alt}_2 \tmop{Sym}_{3, 4} \rho) \left(
    \xi_1, \xi_2, \xi_3, \xi_4 \right)\\
    &  & \\
    & = & \rho \left( \xi_1, \xi_2, \xi_3, \xi_4 \right)_{_1} + \rho \left(
    \xi_1, \xi_2, \xi_4, \xi_3 \right)_{_2} - \rho \left( \xi_2, \xi_1, \xi_3,
    \xi_4 \right)_{_3} - \rho \left( \xi_2, \xi_1, \xi_4, \xi_3 \right)_{_4}\\
    &  & \\
    & + & \rho \left( \xi_1, \xi_2, \xi_4, \xi_3 \right)_{_1} + \rho \left(
    \xi_1, \xi_2, \xi_3, \xi_4 \right)_{_2} - \rho \left( \xi_2, \xi_1, \xi_4,
    \xi_3 \right)_{_4} - \rho \left( \xi_2, \xi_1, \xi_3, \xi_4 \right)_{_3}\\
    &  & \\
    & + & \rho \left( \xi_1, \xi_3, \xi_2, \xi_4 \right)_{_1} + \rho \left(
    \xi_1, \xi_3, \xi_4, \xi_2 \right)_{_2} - \rho \left( \xi_3, \xi_1, \xi_2,
    \xi_4 \right)_{_5} - \rho \left( \xi_3, \xi_1, \xi_4, \xi_2 \right)_{_6}\\
    &  & \\
    & + & \rho \left( \xi_1, \xi_3, \xi_4, \xi_2 \right)_{_1} + \rho \left(
    \xi_1, \xi_3, \xi_2, \xi_4 \right)_{_2} - \rho \left( \xi_3, \xi_1, \xi_4,
    \xi_2 \right)_{_6} - \rho \left( \xi_3, \xi_1, \xi_2, \xi_4 \right)_{_5}\\
    &  & \\
    & + & \rho \left( \xi_1, \xi_4, \xi_2, \xi_3 \right)_{_1} + \rho \left(
    \xi_1, \xi_4, \xi_3, \xi_2 \right)_{_2} - \rho \left( \xi_4, \xi_1, \xi_2,
    \xi_3  \right)_{_7} - \rho \left( \xi_4, \xi_1, \xi_3, \xi_2
    \right)_{_8}\\
    &  & \\
    & + & \rho \left( \xi_1, \xi_4, \xi_3, \xi_2 \right)_{_1} + \rho \left(
    \xi_1, \xi_4, \xi_2, \xi_3 \right)_{_2} - \rho \left( \xi_4, \xi_1, \xi_3,
    \xi_2 \right)_{_8} - \rho \left( \xi_4, \xi_1, \xi_2, \xi_3 \right)_{_7} .
  \end{eqnarray*}
  The fact that $\rho$ is skew-symmetric with respect to the second and third
  variables implies that $\tmop{Sym}_{2, 3, 4} \rho = 0$. We infer
  \begin{eqnarray*}
    &  & (\tmop{Sym}_{2, 3, 4} \tmop{Alt}_2 \tmop{Sym}_{3, 4} \rho) \left(
    \xi_1, \xi_2, \xi_3, \xi_4 \right)\\
    &  & \\
    & = & - 2 \rho \left( \xi_2, \xi_1, \xi_3, \xi_4 \right) - 2 \rho \left(
    \xi_2, \xi_1, \xi_4, \xi_3 \right)\\
    &  & \\
    & - & 2 \rho \left( \xi_3, \xi_1, \xi_2, \xi_4 \right) - 2 \rho \left(
    \xi_3, \xi_1, \xi_4, \xi_2 \right)\\
    &  & \\
    & - & 2 \rho \left( \xi_4, \xi_1, \xi_2, \xi_3  \right) - 2 \rho \left(
    \xi_4, \xi_1, \xi_3, \xi_2 \right)\\
    &  & \\
    & = & 2 \rho \left( \xi_2, \xi_3, \xi_1, \xi_4 \right) + 2 \rho \left(
    \xi_2, \xi_4, \xi_1, \xi_3 \right)\\
    &  & \\
    & + & 2 \rho \left( \xi_3, \xi_2, \xi_1, \xi_4 \right) + 2 \rho \left(
    \xi_3, \xi_4, \xi_1, \xi_2 \right)\\
    &  & \\
    & + & 2 \rho \left( \xi_4, \xi_2, \xi_1, \xi_3  \right) + 2 \rho \left(
    \xi_4, \xi_3, \xi_1, \xi_2 \right),
  \end{eqnarray*}
  which shows the required expressions for $S$.
\end{proof}

By the equation (\ref{Rm3Sm}) we can apply lemma \ref{AltSymCurvature} to the
tensor $\rho \assign \nabla^{S_1} R^{\nabla^{S_1}}$. We infer the equation
\begin{equation}
  \label{EndInteg} S_3 = \frac{1}{4! 3} \tmop{Sym}_{2, 3, 4}  (
  \nabla^{S_1} R^{\nabla^{S_1}})_2 \noplus + \frac{i}{3} \nabla^{S_1}
  \sigma_2 + \sigma_3\,,
\end{equation}
We deduce that the equation (\ref{EqS2gn}) is equivalent to the equation
(\ref{EndInteg}). This concludes the proof of the proposition \ref{MainProp}
thanks to lemma \ref{IntermediateInteg}.

\section{Third reduction of the integrability equations and proof of the main
theorem}

In this section we will prove the following result.

\begin{lemma}
  \label{thrdReductInteg}Under the assumptions of the theorem
  \ref{Maintheorem} the $M$-totally real almost complex structure $J_A$ is
  integrable over $U$ if and only if
  \begin{eqnarray*}
    S_1 & \in & C^{\infty} \left( M, S^2 T^{\ast}_M
    \otimes_{_{\mathbbm{R}}} \mathbbm{C}T_M \right), \hspace{1em} \text{i.e.
    $\nabla^{S_1}$ is torsion free},\\
    &  & \\
    S_2 & = & S^0_2 + \sigma_2,\\
    &  & \\
    S^0_2 \left( \xi_1, \xi_2, \xi_3 \right) & \assign & \frac{i}{6}  \left[
    R^{\nabla^{S_1}} \left( \xi_1, \xi_2 \right) \xi_3 + R^{\nabla^{S_1}}
    \left( \xi_1, \xi_3 \right) \xi_2 \right],\\
    &  & \\
    \sigma_2 & \in & C^{\infty} \left( M_{_{_{_{}}}}, S^3 T^{\ast}_M
    \otimes_{_{\mathbbm{R}}} \mathbbm{C}T_M \right),
  \end{eqnarray*}
  and for all $k \geqslant 3$,
  \begin{eqnarray*}
    S_k & = & \frac{i}{k} \nabla^{S_1} \sigma_{k - 1} + \frac{i}{\left( k + 1\right) !} 
    \tmop{Sym}_{2, \ldots, k + 1} \beta_{k - 1} + \sigma_k\,,\\
    &  & \\
    \sigma_k & \in & C^{\infty} \left( M_{_{_{_{}}}}, S^{k + 1} T^{\ast}_M
    \otimes_{_{\mathbbm{R}}} \mathbbm{C}T_M \right),\\
    &  & \\
    \beta_k & \assign & \frac{i}{k} d_1^{\nabla^{S_1}} \nabla^{S_1} \sigma_{k
    - 1} + \frac{i}{\left( k + 1 \right) !} d_1^{\nabla^{S_1}} \tmop{Sym}_{2, \ldots, k + 1}
    \beta_{k - 1} \\
    &  & \\
    & + & \frac{1}{k!} \tmop{Sym}_{3, \ldots, k + 2} \left( \sum_{p = 2}^{k -
    1} p S_p \wedge_1 S_{k - p + 1} \right),\\
    &  & \\
    \beta_2 & \assign & - \,\frac{i}{3}  ( \nabla^{S_1} R^{\nabla^{S_1}})_2\,,\\
    &  & \\
    \tmop{Circ} \beta_k & = & 0 \,.
  \end{eqnarray*}
\end{lemma}
\begin{proof}
  We show that the statement of proposition \ref{MainProp} is equivalent to
  the statement of lemma \ref{thrdReductInteg}. We show indeed by induction on
  $k \geqslant 3$ the following statement.
\begin{Statement}  
The tensors $S_h$, $h = 3, \ldots, k+1$, satisfy the equations
  \begin{equation}
    \label{ExtH}  \left[ d_1^{\nabla^{S_1}} S_h + \sum_{p = 2}^{h - 1} p S_p
    \wedge_1 S_{h - p + 1} \noplus \noplus \noplus + i \left( h + 1 \right)
    \tmop{Alt}_2 S_{h + 1} \right] \left( \xi_1, \xi_2, \eta^h \right) = 0\,,
  \end{equation}
for all $h = 3, \ldots, k$, for all $\xi_1, \xi_2, \eta \in T_{M, \pi \left( \eta \right)}$ and 
\begin{eqnarray*}
    S_3 & = & \frac{i}{3} \nabla^{S_1} \sigma_2 + \frac{1}{4! 3}
    \tmop{Sym}_{2, 3, 4}  ( \nabla^{S_1} R^{\nabla^{S_1}} )_2
    \noplus + \sigma_3\,,
  \end{eqnarray*}  
with 
  $
  \sigma_3 \in
  C^{\infty} \left( M, S^4 T^{\ast}_M \otimes_{_{\mathbbm{R}}}
  \mathbbm{C}T_M \right)
  $, if and only if the tensors $S_h$ satisfy for all $h = 3,
  \ldots, k + 1$, the identities
  \begin{equation}
    \label{SHexpr} S_h = \frac{i}{h} \nabla^{S_1} \sigma_{h - 1} + \frac{i}{\left( h + 1 \right) !} \tmop{Sym}_{2, \ldots, h + 1} \beta_{h - 1} + \sigma_h\,,
  \end{equation}
  with $\sigma_h \in C^{\infty} \left( M, S^{h + 1} T^{\ast}_M
  \otimes_{_{\mathbbm{R}}} \mathbbm{C}T_M \right)$ and where for all $r = 3,
  \ldots, k$,
 \begin{eqnarray*}
\beta_r &\assign& \frac{i}{r} d_1^{\nabla^{S_1}} \nabla^{S_1} \sigma_{r - 1} +
     \frac{i}{\left( r + 1 \right) !} d_1^{\nabla^{S_1}} \tmop{Sym}_{2, \ldots, r + 1} \beta_{r
     - 1} 
     \\
     \\
     &+& \frac{1}{r!} \tmop{Sym}_{3, \ldots, r + 2} \left( \sum_{p = 2}^{r
     - 1} p S_p \wedge_1 S_{r - p + 1} \right),  
 \end{eqnarray*}
  with $\beta_2 \assign - \frac{i}{3}  ( \nabla^{S_1} R^{\nabla^{S_1}})_2$ satisfies the equation $\tmop{Circ} \beta_r = 0$.
\end{Statement}  
The statement 1 follows directly from the following fact.
\begin{Fact}
Let $S_h$, for some $h = 3, \ldots, k$, be the tensor
  given by (\ref{SHexpr}). Then the tensor $S_{h + 1}$ satisfies the equation
  (\ref{ExtH}) if and only if $S_{h + 1}$ satisfies the identity
  (\ref{SHexpr}), with $h$ replaced by $h + 1$ and $\beta_h$ satisfies the
  equation $\tmop{Circ} \beta_h = 0$.
\end{Fact}  
In order to show the fact 1 we observe first that (\ref{ExtH}) rewrites as
  $$
    d_1^{\nabla^{S_1}} S_h + \frac{1}{h!} \tmop{Sym}_{3, \ldots, h + 2} \left(
    \sum_{p = 2}^{h - 1} p S_p \wedge_1 S_{h - p + 1}  \right) \noplus \noplus
    \noplus + i \left( h + 1 \right) \tmop{Alt}_2 S_{h + 1} = 0\, .
 $$
  Using the expression (\ref{SHexpr}) for $S_h$ and the
  definition of $\beta_h$, we can rewrite the previous identity as
  \begin{equation}
    \label{EqBeta} \beta_h = - \tmop{Alt}_2 \left[ \nabla^{S_1} \sigma_h + i
    \left( h + 1 \right)S_{h + 1} \right] .
  \end{equation}
  By the proof of lemma \ref{exactSequence} we deduce $\tmop{Circ} \beta_h =
  0$ and
  $$
- \nabla^{S_1} \sigma_h - i \left( h + 1 \right)_{_{_{_{}}}} S_{h + 1}  =
     C_{h + 1} \tmop{Sym}_{2, \ldots, h + 2} \beta_h - i \left( h + 1
    \right) \sigma_{h + 1} \,.  
  $$
Therefore the identity (\ref{EqBeta}) is equivalent to; $\tmop{Circ} \beta_h
  = 0$ and $S_{h + 1}$ satisfies (\ref{SHexpr}), with $h$ replaced by $h + 1$.
  This concludes the proof fact 1. We infer the required conclusion of lemma \ref{thrdReductInteg}. 
\end{proof}
\\
\\
{\tmstrong{Proof of the main theorem}}
\\
\\
\begin{proof}
  We show that the recursive definition of $\beta_k$ in the statement of lemma
  \ref{thrdReductInteg} yields the formula
  \begin{eqnarray}
   \beta_k & = & \frac{i}{k} d_1^{\nabla^{S_1}} \nabla^{S_1} \sigma_{k - 1} +
    \frac{1}{\left( k + 1 \right) ! k!} \tmop{Sym}_{3, \ldots, k + 2}
    \theta_k\,,\label{bkcool}
    \\\nonumber
    &  & \\
    \theta_k & \assign & \sum_{r = 2}^{k - 2} \frac{\left( r + 2 \right) !}{r
    + 1}  ( i\,d_1^{\nabla^{S_1}} )^{k - r} \nabla^{S_1} \sigma_r +
    3! \left( i\,d_1^{\nabla^{S_1}} \right)^{k - 2} \beta_2\nonumber
    \\\nonumber
    &  & \\\nonumber
    & + & \sum_{r = 3}^k \left( r + 1 \right) ! \sum_{p = 2}^{r - 1} ( i\,
    d_1^{\nabla^{S_1}})^{k - r} \left( p S_p \wedge_1 S_{r - p + 1}
    \right)\,, \nonumber
  \end{eqnarray}
  for all $k \geqslant 3$. We show (\ref{bkcool}) by induction on $k$. We
  notice first that the recursive definition of $\beta_k$ rewrites as
 $$
    \beta_k = \frac{i}{k} d_1^{\nabla^{S_1}} \nabla^{S_1} \sigma_{k - 1} +
    \tmop{Sym}_{3, \ldots, k + 2} \left[\frac{i}{\left( k + 1 \right) !} d_1^{\nabla^{S_1}}
    \beta_{k - 1} \noplus + \frac{1}{k!} \sum_{p = 2}^{k - 1} p S_p \wedge_1
    S_{k - p + 1} \right],
 $$
  and we write
  \begin{eqnarray*}
    \beta_{k + 1} & = & \frac{i}{k + 1} d_1^{\nabla^{S_1}} \nabla^{S_1}
    \sigma_k + \tmop{Sym}_{3, \ldots, k + 3} \left[ \frac{i}{\left( k + 2
    \right) !} d_1^{\nabla^{S_1}} \beta_k \right],\\
    &  & \\
    & + & \tmop{Sym}_{3, \ldots, k + 3} \left[ \frac{1}{\left( k + 1 \right)
    !} \sum_{p = 2}^k p S_p \wedge_1 S_{k - p + 2} \right] .
  \end{eqnarray*}
  Using the inductive assumption (\ref{bkcool}) we infer the expressions
 \begin{eqnarray*}
    \frac{i}{\left( k + 2 \right) !} d_1^{\nabla^{S_1}} \beta_k & = &
    \frac{1}{\left( k + 2 \right) ! k}  ( i\,d_1^{\nabla^{S_1}})^2\,
    \nabla^{S_1} \sigma_{k - 1} 
    \\
    \\
    &+& \frac{1}{\left( k + 2 \right) ! \left( k + 1
    \right) ! k!} \tmop{Sym}_{4, \ldots, k + 3} i\,d_1^{\nabla^{S_1}}
    \theta_k,\\
    &  & \\
    i\,d_1^{\nabla^{S_1}} \theta_k & = & \sum_{r = 2}^{k - 2} \frac{\left( r +
    2 \right) !}{r + 1}  ( i\,d_1^{\nabla^{S_1}} )^{k + 1 - r}\,
    \nabla^{S_1} \sigma_r + 3! ( i\,d_1^{\nabla^{S_1}} )^{k - 1}
    \beta_2\\
    &  & \\
    & + & \sum_{r = 3}^k \left( r + 1 \right) ! \sum_{p = 2}^{r - 1}( i\,
    d_1^{\nabla^{S_1}})^{k + 1 - r} \left( p S_p \wedge_1 S_{r - p + 1}
    \right) \,.
  \end{eqnarray*}
  This combined with the identity $\tmop{Sym}_{3, \ldots, k + 3}
  \tmop{Sym}_{4, \ldots, k + 3} = k! \tmop{Sym}_{3, \ldots, k + 3}$, yields
  \begin{eqnarray*}
   && \beta_{k + 1} 
   \\
   \\
   & = & \frac{i}{k + 1} d_1^{\nabla^{S_1}} \nabla^{S_1}
    \sigma_k + \frac{1}{\left( k + 2 \right) ! k} \tmop{Sym}_{3, \ldots, k +
    3}  ( i\,d_1^{\nabla^{S_1}} )^2\, \nabla^{S_1} \sigma_{k - 1}\\
    &  & \\
    & + & \frac{1}{\left( k + 2 \right) ! \left( k + 1 \right) !}
    \tmop{Sym}_{3, \ldots, k + 3}  \sum_{r = 2}^{k - 2} \frac{\left( r + 2
    \right) !}{r + 1}  ( i\,d_1^{\nabla^{S_1}})^{k + 1 - r}\,
    \nabla^{S_1} \sigma_r\\
    &  & \\
    & + & \frac{3!}{\left( k + 2 \right) ! \left( k + 1 \right) !}
    \tmop{Sym}_{3, \ldots, k + 3}  ( i\,d_1^{\nabla^{S_1}})^{k - 1}
    \beta_2\\
    &  & \\
    & + & \frac{1}{\left( k + 2 \right) ! \left( k + 1 \right) !}
    \tmop{Sym}_{3, \ldots, k + 3} \sum_{r = 3}^k \left( r + 1 \right) !
    \sum_{p = 2}^{r - 1} ( i\,d_1^{\nabla^{S_1}})^{k + 1 - r}
    \left( p S_p \wedge_1 S_{r - p + 1} \right) \\
    &  & \\
    & + & \frac{1}{\left( k + 1 \right) !} \tmop{Sym}_{3, \ldots, k + 3}
    \sum_{p = 2}^k p S_p \wedge_1 S_{k - p + 2} .
  \end{eqnarray*} 
  Putting the terms together we obtain (\ref{bkcool}) for $\beta_{k + 1}$.
  Then the obvious identity $d_1^{\nabla} \nabla = \tmop{Alt}_2 \nabla^2$,
  combined with the formula (\ref{CurvOPERAT}) below, allows to conclude the
  required expression of $\beta_k \equiv \beta_k \left( \sigma_{k - 1}
  \right)$ in the statement of the main theorem. (We perform the change of indices $r':=r+1$ in the above expression of $\theta_k$). This concludes the proof of
  the main theorem.
\end{proof}

We remind first the following elementary and well known fact.

\begin{lemma}
  For any covariant derivative operator $\nabla$ acting on the smooth sections of $\mathbbm{C}T_M$ and for any tensor $\theta \in
  C^{\infty} \left( X, T^{\ast, \otimes q}_M \otimes \mathbbm{C}T_M \right)$
  holds the identity
  \begin{equation}
    \label{CurvOPERAT} \tmop{Alt}_2 \nabla^2 \theta = R^{\nabla} \nosymbol .
    \, \theta .
  \end{equation}
\end{lemma}

\section{The symplectic approach}
\subsection{General facts}\label{SymplecticAppr}
Let $M$ be a smooth manifold and let $\theta \in C^{\infty} ( T^{\ast}_M,
T^{\ast}_{T^{\ast}_M})$ be the canonical $1$-form on the total space of
the cotangent bundle defined as $\theta_{\lambda} \assign \lambda \cdot
d_{\lambda} \pi_{T^{\ast}_M}$, for any $\lambda \in T_M^{\ast}$. The canonical
symplectic form over the total space $T_M^{\ast}$ is defined as $\Omega
\assign - d \theta$. Let now $g$ be a Riemann metric over $M$ viewed as a
vector bundle map $g : T_M \longrightarrow T^{\ast}_M$. We define also the
forms $\theta^g \assign g^{\ast} \theta$ and $\Omega^g \assign g^{\ast} \Omega
= - d \theta^g$ over the total space of the tangent bundle. In explicit terms
$\theta_{\eta}^g = g \left( \eta \right) \cdot d_{\eta} \pi_{T_M}$, for all
$\eta \in T_M$, i.e.
\begin{eqnarray*}
  \theta_{\eta}^g \left( \xi \right)  =  g_{\pi_{T_M} \left( \eta \right)}
  (\eta, d_{\eta} \pi_{T_M} \cdot \xi)\,,
\end{eqnarray*}
for all $\xi \in T_{T_M, \eta}$. Let $\nabla^g$ be the Levi-Civita connection,
defined as
\begin{eqnarray*}
  2\, \nabla_{\xi}^g \eta  \assign  g^{- 1} \left[ \xi \,\neg\, d \left( g \eta
  \right) + \eta \,\neg\, d \left( g \xi \right)_{_{_{_{}}}} \noplus + d \langle
  \xi, \eta \rangle_g \right] + [\xi, \eta]\,,
\end{eqnarray*}
for any $\xi, \eta \in C^{\infty} \left( M, T_M \right)$. Let also $\gamma^g
\in C^{\infty} (T_M, T_{T_M}^{\ast} \otimes T_{T_M})$ be the Levi-Civita
$1$-form, which is determined along any section $\eta \in C^{\infty} \left( M,
T_M \right)$, by the identity $\gamma_{\eta}^g \cdot d \eta = T_{\eta}
\nabla^g \eta$.

For any curve $\eta : t \longmapsto \eta_t \in T_M$, we define the covariant
derivative
\begin{eqnarray*}
  \frac{\nabla^g \eta}{d t}  \assign  T_{\eta_t}^{- 1} \gamma_{\eta_t}^g 
  \dot{\eta}_t \in T_{M, \pi \left( \eta_t \right)} \,.
\end{eqnarray*}
We consider now two curves $\eta_j : t \longmapsto \eta_{j, t} \in T_M$, $j =
1, 2$, such that $\pi_{T_M} \left( \eta_{1, t} \right) = \pi_{T_M} \left(
\eta_{2, t} \right) = x_t$. Then
\begin{eqnarray*}
  \frac{d}{d t} g_{\mid x_t} \left( \eta_{1, t}\,,\, \eta_{2, t} \right)  = 
  g_{\mid x_t} \left( \frac{\nabla^g \eta_1}{d t}\,,\, \eta_{2, t} \right) +
  g_{\mid x_t} \left( \eta_{1, t}\,,\, \frac{\nabla^g \eta_2}{d t} \right) .
\end{eqnarray*}
With the previous notation hold the following well known lemma (see also
Klingenberg's book \cite{Kli} for a proof using local coordinates).

\begin{lemma}
  \label{symplectLCivi}The formula
  \begin{eqnarray*}
    \Omega_{\eta}^g \left( \xi_1, \xi_2 \right) & = & g_p (d_{\eta} \pi_{T_M}
    \xi_1, T_{\eta}^{- 1} \gamma_{\eta}^g \xi_2) - g_p  (d_{\eta} \pi_{T_M}
    \xi_2, T_{\eta}^{- 1} \gamma_{\eta}^g \xi_1)\,,
  \end{eqnarray*}
  hold for any $\eta \in T_M \nocomma$, $p = \pi_{T_M} \left( \eta \right)$
  and for any $\xi_1, \xi_2 \in T_{T_M, \eta}$.
\end{lemma}

\begin{proof}
  With respect to a local coordinate trivialization of the tangent bundle we
  can extend in a linear way the vectors $\xi_1, \xi_2$ in to vector fields
  $\Xi_1$, $\Xi_2$ in a neighborhood of $T_{M, p}$ inside $T_M$. In this way
  $[\Xi_1, \Xi_2] = 0$ and thus $\Omega^g \left( \Xi_1, \Xi_2 \right) = \Xi_2
  . \theta^g \left( \Xi_1 \right) - \Xi_1 . \theta^g \left( \Xi_2 \right)$. We
  denote by $\eta_{j, t}$, $j = 1, 2$ the corresponding flow lines starting
  from $\eta$. Then
  \begin{eqnarray*}
    \Omega_{\eta}^g \left( \xi_1, \xi_2 \right) & = & \frac{d}{d t} _{\mid_{t
    = 0}}  \left[ g_{\pi_{T_M} \left( \eta_{2, t} \right)} \left( \eta_{2, t}\,,
    d_{\eta_{2, t}} \pi_{T_M} \cdot \Xi_1 \left( \eta_{2, t}
    \right)_{_{_{_{}}}} \right) \right] \\
    &  & \\
    & - & \frac{d}{d t} _{\mid_{t = 0}}  \left[ g_{\pi_{T_M} \left( \eta_{1,
    t} \right)} \left( \eta_{1, t}\,, d_{\eta_{1, t}} \pi_{T_M} \cdot \Xi_2
    \left( \eta_{1, t} \right)_{_{_{_{}}}} \right) \right] .
  \end{eqnarray*}
  We distinguish two cases.
  
  $\bullet$ In the case when $d_{\eta} \pi_{T_M} \xi_j = 0$ for some $j$, say
  $j = 1$, then $d_{\eta_{2, t}} \pi_{T_M} \Xi_1 \left( \eta_{2, t} \right) =
  0$ and
  \[  \frac{d}{d t} d_{\eta_{1, t}} \pi_{T_M} \Xi_2 \left( \eta_{1, t} \right)
     = 0\,, \]
  by the linear nature of the local extension. Then
  \begin{eqnarray*}
    \Omega_{\eta}^g \left( \xi_1, \xi_2 \right) = - g_p  (T_{\eta}^{- 1}
    \gamma_{\eta}^g \xi_1 \nocomma, d_{\eta} \pi_{T_M} \xi_2) \,.
  \end{eqnarray*}
  The case $j = 2$ is quite similar.
  
  $\bullet$ In the case when $d_{\eta} \pi_{T_M} \xi_j$, do not vanish for $j
  = 1, 2$, then the vector fields $\zeta_j \assign d \pi_{T_M} \Xi_j$ are well
  defined and $[\zeta_1, \zeta_2] = 0$. Then
  \begin{eqnarray*}
    \Omega_{\eta}^g \left( \xi_1, \xi_2 \right) & = & g_p (T_{\eta}^{- 1}
    \gamma_{\eta}^g \xi_2, d_{\eta} \pi_{T_M} \xi_1) + g_p (\eta,
    \nabla_{\zeta_2 \left( p \right)}^g \zeta_1 - \nabla_{\zeta_1 \left( p
    \right)}^g \zeta_2)\\
    &  & \\
    & - & g_p (T_{\eta}^{- 1} \gamma_{\eta}^g \xi_1, d_{\eta} \pi_{T_M}
    \xi_2)\\
    &  & \\
    & = & g_p (T_{\eta}^{- 1} \gamma_{\eta}^g \xi_2, d_{\eta} \pi_{T_M}
    \xi_1) + g_p (\eta, [\zeta_1, \zeta_2] \left( p \right))\\
    &  & \\
    & - & g_p (T_{\eta}^{- 1} \gamma_{\eta}^g \xi_1, d_{\eta} \pi_{T_M}
    \xi_2)\,,
  \end{eqnarray*}
  which implies the required conclusion.
\end{proof}

We need to remind in detail also the following very well known lemma (see also \cite{Kli}).

\begin{lemma}
  Let $2\, \zeta^g \assign \Omega^{g, - 1} d\, | \cdot |^2_g$ and let $\Phi_t^g$
  be the corresponding $1$-parameter sub-group of transformations of $T_M$.
  Then for any $\eta \in T_M$ the curve $c_t \assign \pi_{T_M} \circ \Phi_t^g
  \left( \eta \right)$ is the geodesic with initial speed $\dot{c}_0 = \eta$
  and $\dot{c}_t = \Phi_t^g \left( \eta \right)$.
\end{lemma}

\begin{proof}
  For any $\eta \in T_M$ and for any $\xi \in T_{T_M, \eta}$, let $t
  \longmapsto \eta_t \in T_M$ be the curve such that $\dot{\eta}_0 = \xi$.
  Then
  \begin{eqnarray*}
    \xi \,.\, | \cdot |_g^2 & = & \frac{d}{d t}_{\mid_{t = 0}}  \left[
    g_{\pi_{T_M} \left( \eta_t \right)} (\eta_t\,, \eta_t) \right]\\
    &  & \\
    & = & 2 \,g_p \left( \eta\,, T_{\eta}^{- 1} \gamma_{\eta}^g \xi \right),
  \end{eqnarray*}
  and thus 
  $$
  \Omega^g_{\eta} \left( \zeta_{\eta}^g\,,\, \xi \right) = g_p \left(
  \eta\,, T_{\eta}^{- 1} \gamma_{\eta}^g \xi \right),
  $$ by the definition of the
  vector field $\zeta_{\eta}^g$. Using lemma \ref{symplectLCivi} we infer
  \begin{equation}
    \label{Mix} g_p (d_{\eta} \pi_{T_M} \zeta^g_{\eta}, T_{\eta}^{- 1}
    \gamma_{\eta}^g \xi) - g_p  (d_{\eta} \pi_{T_M} \xi, T_{\eta}^{- 1}
    \gamma_{\eta}^g \zeta^g_{\eta}) = g_p \left( \eta, T_{\eta}^{- 1}
    \gamma_{\eta}^g \xi \right) .
  \end{equation}
  In the case $d_{\eta} \pi_{T_M} \xi = 0$, the identity (\ref{Mix}) yields
  \begin{eqnarray*}
    g_p (d_{\eta} \pi_{T_M} \zeta^g_{\eta}, T_{\eta}^{- 1} \xi) = g_p
    \left( \eta, T_{\eta}^{- 1} \xi \right),
  \end{eqnarray*}
  and thus $d_{\eta} \pi_{T_M} \zeta^g_{\eta} = \eta$. In the case
  $\gamma_{\eta}^g \,\xi = 0$, the identity (\ref{Mix}) yields
  \begin{eqnarray*}
    g_p  (d_{\eta} \pi_{T_M} \xi\,, T_{\eta}^{- 1} \gamma_{\eta}^g
    \zeta^g_{\eta}) = 0\,,
  \end{eqnarray*}
  and thus $\gamma_{\eta}^g\, \zeta^g_{\eta} = 0$. We deduce the formula
  \begin{equation}
    \label{GenGFw} \zeta^g_{\eta} = H_{\eta}^g \cdot \eta \,.
  \end{equation}
  Thus the flow line $\eta_t \assign \Phi_t^g \left( \eta \right)$ satisfies
  the identity
  \begin{equation}
    \label{flowLine}  \dot{\eta}_t = H_{\eta_t}^g \cdot \eta_t \,.
  \end{equation}
  We deduce
  \begin{eqnarray*}
    \dot{c}_t & = & d_{\eta_t} \pi_{T_M} \cdot \dot{\eta}_t\\
    &  & \\
    & = & d_{\eta_t} \pi_{T_M} \cdot H_{\eta_t}^g \cdot \eta_t\\
    &  & \\
    & = & \eta_t\,,
  \end{eqnarray*}
  and $\ddot{c}_t = H_{\dot{c}_t}^g \cdot \dot{c}_t$, which is the geodesic
  equation.
\end{proof}

We provide now a proof of the following well known result due to
Lempert-Sz\"oke \cite{Le-Sz1}. See also Guillemin-Stenzel \cite{Gu-St}, Burns \cite{Bu1,Bu2} and Burns-Halverscheid-Hind \cite{BHH}.

\begin{corollary}
  \label{sympHol}Let $\left( M, g \right)$ be a smooth Riemannian manifold. A
  complex structure $J$ over the total space of the tangent bundle $T_M$
  satisfies the conditions
  \begin{equation}
    \label{JcanM} J_{\mid M} = J^{\tmop{can}},
  \end{equation}
  \begin{equation}
    \label{CanFm} 2 \theta^g = d | \cdot |_g^2 \cdot J .
  \end{equation}
  if and only if for any $\eta \in T_M$, the smooth map $\psi_{\eta} : t +
  i s \longmapsto s \Phi_t^g \left( \eta \right)$, defined in a neighborhood
  of $0 \in \mathbbm{C}$, is $J$-holomorphic.
\end{corollary}

\begin{proof}
  We define the Reeb vector field $\Xi \assign \Omega^{g, - 1} \theta^g$. This
  vector field is independent of the metric $g$. Indeed by lemma
  \ref{symplectLCivi} the identity
  \begin{equation}
    \label{metricIDEN} g_p (\eta\,, d_{\eta} \pi_{T_M} \xi) = g_p (d_{\eta}
    \pi_{T_M} \Xi_{\eta}\,, T_{\eta}^{- 1} \gamma_{\eta}^g \xi) - g_p  (d_{\eta}
    \pi_{T_M} \xi\,, T_{\eta}^{- 1} \gamma_{\eta}^g \Xi_{\eta})\,,
  \end{equation}
  holds for any $\xi_{} \in T_{T_M, \eta}$. Thus if $d_{\eta} \pi_{T_M} \xi = 0$ we
  deduce the equality 
  $$
  g_p (d_{\eta} \pi_{T_M} \Xi_{\eta}\,, T_{\eta}^{- 1} \xi)
  = 0\,,
$$ and thus $d_{\eta} \pi_{T_M} \Xi_{\eta} = 0$. Then the identity
  (\ref{metricIDEN}) reduces as
  \begin{eqnarray*}
    g_p (\eta\,, d_{\eta} \pi_{T_M} \xi) = -\; g_p  (d_{\eta} \pi_{T_M} \xi\,,
    T_{\eta}^{- 1} \Xi_{\eta})\,,
  \end{eqnarray*}
  for any $\xi_{} \in T_{T_M, \eta}$. We infer the formula
  \begin{equation}
    \label{Reeb} \Xi_{\eta} = - \;T_{\eta} \cdot \eta\,,
  \end{equation}
  for all $\eta \in T_M$. We notice now that the identity (\ref{CanFm}) is
  equivalent to the identity
  \begin{eqnarray*}
    \Omega^g \left( 2\, \Xi\,, \xi \right)  =  d\, | \cdot |^2_g \,J \xi\,,
  \end{eqnarray*}
  and is also equivalent to the identity $\theta^g = - d_{_{J^{}}}^c | \cdot
  |_g^2 .$ Thus 
$$
\Omega^g = d d_{_{J^{}}}^c | \cdot |_g^2 = i
  \partial_{_{J^{}}} \overline{\partial}_{_{J^{}}} | \cdot |_g^2\,,
$$
thanks to
  the fact that $J^g$ is integrable. We infer that the symplectic form
  $\Omega^g$ is $J^{}$-invariant. Thus
  \begin{eqnarray*}
    \Omega^g \left( 2\, J\, \Xi\,, J \xi \right)  =  d\, | \cdot |^2_g\, J \xi\,,
  \end{eqnarray*}
  i.e.
  \begin{equation}
    \label{ReebID} J\, \Xi = \zeta^g .
  \end{equation}
  This combined with (\ref{Reeb}) and with (\ref{GenGFw}) implies that
  (\ref{CanFm}) is equivalent to the identity
  \begin{equation}
    \label{Jcan} J_{\eta}\, H_{\eta}^g \cdot \eta = T_{\eta} \cdot \eta \,.
  \end{equation}
  We show now that the later combined with (\ref{JcanM}) is equivalent to the
  $J$-holomorphy of the maps $\psi_{\eta}$. For this purpose we observe that
  the differential of such maps is given by
  \begin{eqnarray*}
    d_{t_0 + i s_0} \psi_{\eta} \left( a \frac{\partial}{\partial t} \noplus +
    b \frac{\partial}{\partial s} \right) =  a\, d \left( s_0
    \mathbbm{I}_{T_M} \right) \dot{\Phi}^g_{t_0} \left( \eta \right) + b\,
    T_{s_0 \Phi^g_{t_0} \left( \eta \right)} \Phi^g_{t_0} \left( \eta \right)
    .
  \end{eqnarray*}
  But
  \begin{eqnarray*}
    \dot{\Phi}^g_{t_0} \left( \eta \right) & = & \zeta^g \circ \Phi^g_{t_0}
    \left( \eta \right)\\
    &  & \\
    & = & H_{\Phi^g_{t_0} \left( \eta \right)}^g \cdot \Phi^g_{t_0} \left(
    \eta \right),
  \end{eqnarray*}
  thanks to (\ref{GenGFw}). Then using the property (\ref{multHoriz}) of the
  linear connection $\nabla^g$ we infer
  \begin{equation}
    \label{expDifJhol} d_{t_0 + i s_0} \psi_{\eta} \left( a
    \frac{\partial}{\partial t} \noplus + b \frac{\partial}{\partial s}
    \right) = \left( a H^g_{s_0 \Phi^g_{t_0} \left( \eta \right)} + b T_{s_0
    \Phi^g_{t_0} \left( \eta \right)} \right) \cdot \Phi^g_{t_0} \left( \eta
    \right) .
  \end{equation}
  The smooth map $\psi_{\eta}$ is $J$-holomorphic if and only if
  \begin{eqnarray*}
    d_{t_0 + i s_0} \psi_{\eta} \left( - b \frac{\partial}{\partial t} \noplus
    + a \frac{\partial}{\partial s} \right)  =  J d_{t_0 + i s_0}
    \psi_{\eta} \left( a \frac{\partial}{\partial t} \noplus + b
    \frac{\partial}{\partial s} \right),
  \end{eqnarray*}
  thus, if and only if
  \begin{eqnarray*}
    \left( - b H^g_{s_0 \Phi^g_{t_0} \left( \eta \right)} + a T_{s_0
    \Phi^g_{t_0} \left( \eta \right)} \right) \cdot \Phi^g_{t_0} \left( \eta
    \right) = J \left( a H^g_{s_0 \Phi^g_{t_0} \left( \eta \right)} + b
    T_{s_0 \Phi^g_{t_0} \left( \eta \right)} \right) \cdot \Phi^g_{t_0} \left(
    \eta \right) .
  \end{eqnarray*}
  For $s_0 \neq 0$ this is equivalent to (\ref{Jcan}). For $s_0 = 0$ this
  is equivalent to (\ref{JcanM}). We deduce the required conclusion.
\end{proof}

The condition (\ref{JcanM}) implies that $J$ is an $M$-totally real complex
structure. We can provide now the proof of corollary \ref{MainCoroll}.

\subsection{Proof of corollary \ref{MainCoroll}}\label{Prooof MainCoroll}

\begin{proof}
  If we write $A = \alpha + i T B$ and $\alpha = H^g - T \Gamma$, then 
  $$ S\assign T^{-1}(H^{\nabla}-\overline{A}\,)=\Gamma + i B\,.
  $$
  We set $S_k = \Gamma_k + i B_k$. From the proof of
  corollary \ref{sympHol} we know that in the case $J$ is integrable over $U$,
  the curve $\psi_{\eta}$ is $J$-holomorphic if and only if hold (\ref{Jcan}).
  The later rewrites as
  \begin{eqnarray*}
    H_{\eta}^g \cdot \eta  =  - J_{\eta} T_{\eta} \cdot \eta \,.
  \end{eqnarray*}
  Using (\ref{JAvert}) we infer that the previous identity is equivalent to
  \begin{equation}
    \label{interKey} H_{\eta}^g \cdot \eta = \alpha_{\eta} B^{- 1}_{\eta}
    \cdot \eta \,.
  \end{equation}
  Taking $d_{\eta} \pi$ on both sides of (\ref{interKey}) we deduce $\eta =
  B^{- 1}_{\eta} \cdot \eta$. Therefore (\ref{interKey}) is equivalent to the
  system
  \begin{equation}
    \label{systKey}  \left\{ \begin{array}{l}
      B_{\eta} \cdot \eta  = \eta \hspace{0.25em},\\
      \\
      H_{\eta}^g \cdot \eta  = \alpha_{\eta}
      \cdot \eta \hspace{0.25em} .
    \end{array} \right.
  \end{equation}
  Then the system (\ref{systKey}) rewrites as
  \[ \left\{ \begin{array}{l}
       \sum_{k \geqslant 1} B_k \left( \eta^{k + 1} \right)   =
       0 \hspace{0.25em},\\
       \\
       \sum_{k \geqslant 1} \Gamma_k \left( \eta^{k + 1} \right)
        = 0 \hspace{0.25em} .
     \end{array} \right. \]
  and thus as $S_k \left( \eta^{k + 1} \right) = 0$ for all
  $k \geqslant 1$. We remind now that, according to theorem \ref{Maintheorem},
  the integrability of the structure $J$ implies the condition $S_1 \in
  C^{\infty} ( M, S^2 T^{\ast}_M \otimes_{_{\mathbbm{R}}}
  \mathbbm{C}T_M)$. We infer $S_1 = 0$. We notice that, with the
  notation of the statement of theorem \ref{Maintheorem}, the equation
  $\tmop{Circ} \beta_k = 0$ hold for all $k \geqslant 1$. This combined with
  the identity 
  $$
  [\tmop{Circ}, \tmop{Sym}_{2, \ldots, k + 2}] = 0\,,
  $$ implies
  \begin{equation}
    \label{CrcSmBeta} \tmop{Circ} \tmop{Sym}_{2, \ldots, k + 2} \beta_k = 0\,,
  \end{equation}
  for all $k \geqslant 1$. So if \ we apply the $\tmop{Circ}$ operator to both
  sides of the definition of $S_2$ in the statement of theorem
  \ref{Maintheorem} we infer $\tmop{Circ} S_2 = \tmop{Circ} \sigma_2 = 3
  \sigma_2$. If we evaluate this equality to $\eta^3$ we infer $S_2 \left(
  \eta^3 \right) = \sigma_2 \left( \eta^3 \right)$, which implies $\sigma_2 =
  0$. We show now by induction that $\sigma_k = 0$ for all $k \geqslant 2$.
  Indeed by the inductive assumption
  \begin{eqnarray*}
    S_{k + 1} = \frac{i}{\left( k + 2 \right) !} \tmop{Sym}_{2, \ldots, k
    + 2} \beta_k + \sigma_{k + 1} \,.
  \end{eqnarray*}
  Applying the $\tmop{Circ}$ operator to both sides of this identity and using
  the equation (\ref{CrcSmBeta}), we infer $\tmop{Circ} S_{k + 1} =
  \tmop{Circ} \sigma_{k + 1} = 3 \sigma_{k + 1}$, which evaluated at $\eta^{k
  + 2}$ gives $S_{k + 1} \left( \eta^{k + 2} \right) = \sigma_{k + 1} \left(
  \eta^{k + 2} \right)$. We deduce $\sigma_{k + 1} = 0$. Using the identity
  \begin{equation}
    \label{SmSm} \tmop{Sym}_{2, \ldots, k + 1} \tmop{Sym}_{3, \ldots, k + 1} =
    \left( k - 1 \right) ! \tmop{Sym}_{2, \ldots, k + 1}\,,
  \end{equation}
  we infer from the statement of theorem \ref{Maintheorem} and with the
  notation there
  \begin{eqnarray*}
    S_k & = & \frac{i}{\left( k + 1 \right) ! k!} \tmop{Sym}_{2, \ldots, k +
    1} \theta_{k - 1},
  \end{eqnarray*}
  for $k \geqslant 2$, with $\theta_1 \assign 2 R^g$ and
  \begin{eqnarray*}
    \theta_k & \assign & - 2 i ( i d_1^{\nabla^g})^{k - 2} \left(
    \nabla^g R^g \right)_2\\
    &  & \\
    & + &  \sum_{r = 4}^{k+1}  r! \sum_{p = 2}^{r - 2} (
    i d_1^{\nabla^g} )^{k+1 - r} \left( p S_p \wedge_1 S_{r - p}
    \right),
  \end{eqnarray*}
  for all $k \geqslant 2$. Moreover we observe that the equation $\tmop{Circ} \beta_k = 0$, $k
  \geqslant 3$ rewrites as 
  $$
  \tmop{Circ} \tmop{Sym}_{3, \ldots, k + 2} \theta_k
  = 0\,.
  $$ 
If we set $\Theta_k \left( g \right) \assign \theta_{k - 1}$, for all
  $k \geqslant 2$ we obtain the required expansion.
  
On the other hand if the expansion in the statement of the lemma under
  consideration hold then $J$ is integrable thanks to theorem
  \ref{Maintheorem} and $\tmop{Circ} S_k = 0$ for all $k \geqslant 2$, ($S_1 =
  0$). Indeed for $k = 2, 3$ this equality follows from the identities
  $\tmop{Circ} \Theta_k \left( g \right) = 0$ and
  \begin{equation}
    \label{ComCircSm}  [\tmop{Circ}, \tmop{Sym}_{2, \ldots, k + 1}] = 0 \,.
  \end{equation}
  For $k \geqslant 4$, we use the identities (\ref{ComCircSm}), (\ref{SmSm})
  and the integrability equations satisfied by the metric $g$. We deduce $S_k \left( \eta^{k
  + 1} \right)  = 0$, for all $k \geqslant 1$, which is
  equivalent to (\ref{interKey}) and so to the fact that the curves
  $\psi_{\eta}$ are $J$-holomorphic.
\end{proof}
\section{Proof of the proposition \ref{Vanish}}\label{Vanishing4}
\begin{proof}
  We expand first the term
  \begin{eqnarray*}
    &  & d_1^{\nabla} \left( \nabla R^{\nabla} \right)_2 \left( \xi_1, \xi_2,
    \xi_3, \xi_4, \xi_5 \right)\\
    &  & \\
    & = & \nabla_{\xi_1} \left( \nabla R^{\nabla} \right)_2 \left( \xi_2,
    \xi_3, \xi_4, \xi_5 \right) - \nabla_{\xi_2} \left( \nabla R^{\nabla}
    \right)_2 \left( \xi_1, \xi_3, \xi_4, \xi_5 \right)\\
    &  & \\
    & = & \nabla^2 R^{\nabla} \left( \xi_1, \xi_3, \xi_2, \xi_4, \xi_5
    \right) - \nabla^2 R^{\nabla} \left( \xi_2, \xi_3, \xi_1, \xi_4, \xi_5
    \right)\\
    &  & \\
    & = & \nabla^2 R^{\nabla} \left( \xi_1, \xi_3, \xi_2, \xi_4, \xi_5
    \right) + \nabla^2 R^{\nabla} \left( \xi_2, \xi_3, \xi_4, \xi_1, \xi_5
    \right)\\
    &  & \\
    & = & \nabla^2 R^{\nabla} \left( \xi_3, \xi_1, \xi_2, \xi_4, \xi_5
    \right) + \nabla^2 R^{\nabla} \left( \xi_3, \xi_2, \xi_4, \xi_1, \xi_5
    \right)\\
    &  & \\
    & + & (R^{\nabla} . R^{\nabla}) \left( \xi_1, \xi_3, \xi_2, \xi_4, \xi_5
    \right) + (R^{\nabla} . R^{\nabla}) \left( \xi_2, \xi_3, \xi_4, \xi_1,
    \xi_5 \right),
  \end{eqnarray*}
  thanks to formula (\ref{CurvOPERAT}). Using the differential Bianchi identity we infer
  \begin{eqnarray*}
    &  & d_1^{\nabla} \left( \nabla R^{\nabla} \right)_2 \left( \xi_1, \xi_2,
    \xi_3, \xi_4, \xi_5 \right)\\
    &  & \\
    & = & - \nabla^2 R^{\nabla} \left( \xi_3, \xi_4, \xi_1, \xi_2, \xi_5
    \right)\\
    &  & \\
    & + & (R^{\nabla} . R^{\nabla}) \left( \xi_1, \xi_3, \xi_2, \xi_4, \xi_5
    \right) + (R^{\nabla} . R^{\nabla}) \left( \xi_2, \xi_3, \xi_4, \xi_1,
    \xi_5 \right)\,.
  \end{eqnarray*}
  In order to simplify the notation in the computations that will follow we
  will use from now on the identification 
  $$
  \theta \left( \xi_1, \xi_2, \xi_3,
  \xi_4, \xi_5 \right) \equiv \theta \left( 12345 \right)\,,
  $$
  for any tensor
  $\theta$. We expand now the term 
  $$
  \tmop{Circ} \tmop{Sym}_{3, 4, 5} d_1^{\nabla} \left( \nabla
  R^{\nabla} \right)_2\,.
  $$ We let 
  $$
  \theta \left( 12345 \right) \assign \nabla^2
  R^{\nabla} \left( 34125 \right)\,,
  $$ and we observe the identities
  \begin{eqnarray*}
    \left( \tmop{Sym}_{3, 4, 5} \theta \right) \left( 12345 \right) & = &
    \nabla^2 R^{\nabla} \left( 34125 \right) + \nabla^2 R^{\nabla} \left(
    35124 \right) + \nabla^2 R^{\nabla} \left( 43125 \right)\\
    &  & \\
    & + & \nabla^2 R^{\nabla} \left( 45123 \right) + \nabla^2 R^{\nabla}
    \left( 53124 \right) + \nabla^2 R^{\nabla} \left( 54123 \right),\\
    &  & \\
    \left( \tmop{Sym}_{3, 4, 5} \theta \right) \left( 23145 \right) & = &
    \nabla^2 R^{\nabla} \left( 14235 \right) + \nabla^2 R^{\nabla} \left(
    15234 \right) + \nabla^2 R^{\nabla} \left( 41235 \right)\\
    &  & \\
    & + & \nabla^2 R^{\nabla} \left( 45231 \right) + \nabla^2 R^{\nabla}
    \left( 51234 \right) + \nabla^2 R^{\nabla} \left( 54231 \right),\\
    &  & \\
    \left( \tmop{Sym}_{3, 4, 5} \theta \right) \left( 31245 \right) & = &
    \nabla^2 R^{\nabla} \left( 24315 \right) + \nabla^2 R^{\nabla} \left(
    25314 \right) + \nabla^2 R^{\nabla} \left( 42315 \right)\\
    &  & \\
    & + & \nabla^2 R^{\nabla} \left( 45312 \right) + \nabla^2 R^{\nabla}
    \left( 52314 \right) + \nabla^2 R^{\nabla} \left( 54312 \right),
  \end{eqnarray*}
  Summing up we obtain
  \begin{eqnarray*}
   && \left( \tmop{Circ} \tmop{Sym}_{3, 4, 5} \theta \right) \left( 12345
    \right) \\
    &  & \\
    & = & \nabla^2 R^{\nabla} \left( 34125 \right) + \nabla^2
    R^{\nabla} \left( 14235 \right) + \nabla^2 R^{\nabla} \left( 24315
    \right)\\
    &  & \\
    & + & \nabla^2 R^{\nabla} \left( 35124 \right) + \nabla^2 R^{\nabla}
    \left( 15234 \right) + \nabla^2 R^{\nabla} \left( 25314 \right)\\
    &  & \\
    & + & \nabla^2 R^{\nabla} \left( 43125 \right)_{_1} + \nabla^2 R^{\nabla}
    \left( 41235 \right)_{_1} + \nabla^2 R^{\nabla} \left( 42315
    \right)_{_1}\\
    &  & \\
    & + & \nabla^2 R^{\nabla} \left( 45123 \right)_{_2} + \nabla^2 R^{\nabla}
    \left( 45231 \right)_{_2} + \nabla^2 R^{\nabla} \left( 45312
    \right)_{_2}\\
    &  & \\
    & + & \nabla^2 R^{\nabla} \left( 53124 \right)_{_3} + \nabla^2 R^{\nabla}
    \left( 51234 \right)_{_3} + \nabla^2 R^{\nabla} \left( 52314
    \right)_{_3}\\
    &  & \\
    & + & \nabla^2 R^{\nabla} \left( 54123 \right)_{_4} + \nabla^2 R^{\nabla}
    \left( 54231 \right)_{_4} + \nabla^2 R^{\nabla} \left( 54312 \right)_{_4},
  \end{eqnarray*}
  where we denote by $\nabla^2 R^{\nabla} \left( \cdot \cdot \cdot \cdot \cdot
  \right)_{_j}$ the terms that summed up together equal zero thanks to the
  differential Bianchi identity for $j=1,3$ and thanks to the
  algebraic Bianchi identity for $j=2,4$. Using formula (\ref{CurvOPERAT}) we infer
  \begin{eqnarray*}
   && \left( \tmop{Circ} \tmop{Sym}_{3, 4, 5} \theta \right) \left( 12345
    \right) 
    \\
    &  & \\
    & = & \nabla^2 R^{\nabla} \left( 43125 \right)_{_1} + \nabla^2
    R^{\nabla} \left( 41235 \right)_{_1} + \nabla^2 R^{\nabla} \left( 42315
    \right)_{_1}\\
    &  & \\
    & + & \nabla^2 R^{\nabla} \left( 53124 \right)_{_2} + \nabla^2 R^{\nabla}
    \left( 51234 \right)_{_2} + \nabla^2 R^{\nabla} \left( 52314
    \right)_{_2}\\
    &  & \\
    & + & (R^{\nabla} . R^{\nabla}) \left( 34125 \right) + (R^{\nabla} .
    R^{\nabla}) \left( 14235 \right) + (R^{\nabla} . R^{\nabla}) \left( 24315
    \right)\\
    &  & \\
    & + & (R^{\nabla} . R^{\nabla}) \left( 35124 \right) + (R^{\nabla} .
    R^{\nabla}) \left( 15234 \right) + (R^{\nabla} . R^{\nabla}) \left( 25314
    \right),
  \end{eqnarray*}
  where as before we denote by $\nabla^2 R^{\nabla} \left( \cdot \cdot \cdot
  \cdot \cdot \right)_{_j}$ the terms that summed up together equal zero
  thanks to the differential Bianchi identity. We deduce the expression
  \begin{eqnarray}
    &&\left( \tmop{Circ} \tmop{Sym}_{3, 4, 5} \theta \right) \left( 12345
    \right) \nonumber
    \\\nonumber
    &  & \\
    & = & (R^{\nabla} . R^{\nabla}) \left( 34125 \right) + (R^{\nabla}
    . R^{\nabla}) \left( 14235 \right) + (R^{\nabla} . R^{\nabla}) \left(
    24315 \right)\nonumber
    \\\nonumber
    &  & \\
    & + & (R^{\nabla} . R^{\nabla}) \left( 35124 \right) + (R^{\nabla} .
    R^{\nabla}) \left( 15234 \right) + (R^{\nabla} . R^{\nabla}) \left( 25314
    \right)\, .\label{CrcSmTeta}
  \end{eqnarray}
  We set now for notation simplicity $\rho \assign R^{\nabla} . R^{\nabla}$
  and let 
  $$
  \Theta \left( 12345 \right) \assign \rho \left( 13245 \right) +
  \rho \left( 23415 \right)\,.
  $$
We observe that, by definition, the tensor
$$
\rho
  \in C^{\infty} \left( M, \Lambda^2 T^{\ast}_M
  \otimes_{_{\mathbbm{R}}} \Lambda^2 T^{\ast}_M \otimes_{_{\mathbbm{R}}}
  T^{\ast}_M \otimes_{_{\mathbbm{R}}} \mathbbm{C}T_M \right)\,,
  $$ 
  satisfies the
  circular identity with respect to its last three entries. We expand now the term
  $$
  \tmop{Circ} \tmop{Sym}_{3, 4, 5} \Theta\,.
  $$
  We observe the identities
  \begin{eqnarray*}
    \left( \tmop{Sym}_{3, 4, 5} \Theta \right) \left( 12345 \right) & = & \rho
    \left( 13245 \right) + \rho \left( 23415 \right)\\
    &  & \\
    & + & \rho \left( 13254 \right) + \rho \left( 23514 \right)\\
    &  & \\
    & + & \rho \left( 14235 \right) + \rho \left( 24315 \right)\\
    &  & \\
    & + & \rho \left( 14253 \right) + \rho \left( 24513 \right)\\
    &  & \\
    & + & \rho \left( 15234 \right) + \rho \left( 25314 \right)\\
    &  & \\
    & + & \rho \left( 15243 \right) + \rho \left( 25413 \right),
  \end{eqnarray*}
  \begin{eqnarray*}
    \left( \tmop{Sym}_{3, 4, 5} \Theta \right) \left( 23145 \right) & = & \rho
    \left( 21345 \right) + \rho \left( 31425 \right)\\
    &  & \\
    & + & \rho \left( 21354 \right) + \rho \left( 31524 \right)\\
    &  & \\
    & + & \rho \left( 24315 \right) + \rho \left( 34125 \right)\\
    &  & \\
    & + & \rho \left( 24351 \right) + \rho \left( 34521 \right)\\
    &  & \\
    & + & \rho \left( 25314 \right) + \rho \left( 35124 \right)\\
    &  & \\
    & + & \rho \left( 25341 \right) + \rho \left( 35421 \right),
  \end{eqnarray*}
  \begin{eqnarray*}
    \left( \tmop{Sym}_{3, 4, 5} \Theta \right) \left( 31245 \right) & = & \rho
    \left( 32145 \right) + \rho \left( 12435 \right)\\
    &  & \\
    & + & \rho \left( 32154 \right) + \rho \left( 12534 \right)\\
    &  & \\
    & + & \rho \left( 34125 \right) + \rho \left( 14235 \right)\\
    &  & \\
    & + & \rho \left( 34152 \right) + \rho \left( 14532 \right)\\
    &  & \\
    & + & \rho \left( 35124 \right) + \rho \left( 15234 \right)\\
    &  & \\
    & + & \rho \left( 35142 \right) + \rho \left( 15432 \right) .
  \end{eqnarray*}
  Summing up we obtain
  \begin{eqnarray*}
    &  & \left( \tmop{Circ} \tmop{Sym}_{3, 4, 5} \Theta \right) \left( 12345
    \right)\\
    &  & \\
    & = & \rho \left( 13245 \right)_{_1} + \rho \left( 23415 \right)_{_2} +
    \rho \left( 21345 \right)_{_3} + \rho \left( 31425 \right)_{_1} 
   \\
    &  & \\ 
    &+& \rho
    \left( 32145 \right)_{_2} + \rho \left( 12435 \right)_{_3}\\
    &  & \\
    & + & \rho \left( 13254 \right)_{_4} + \rho \left( 23514 \right)_{_5} +
    \rho \left( 21354 \right)_{_6} + \rho \left( 31524 \right)_{_4} 
    \\
    &  & \\
    &+& \rho
    \left( 32154 \right)_{_5} + \rho \left( 12534 \right)_{_6}\\
    &  & \\
    & + & \rho \left( 14235 \right)_{_7} + \rho \left( 24315 \right)_{_8} +
    \rho \left( 24315 \right)_{_8} + \rho \left( 34125 \right)_{_9} 
    \\
    &  & \\
    &+& \rho
    \left( 34125 \right)_{_9} + \rho \left( 14235 \right)_{_7}\\
    &  & \\
    & + & \rho \left( 14253 \right)_{_7} + \rho \left( 24513 \right)_{_8} +
    \rho \left( 24351 \right)_{_8} + \rho \left( 34521 \right)_{_9} 
  \\
    &  & \\  
    &+& \rho
    \left( 34152 \right)_{_9} + \rho \left( 14532 \right)_{_7}\\
    &  & \\
    & + & \rho \left( 15234 \right)_{_{10}} + \rho \left( 25314
    \right)_{_{11}} + \rho \left( 25314 \right)_{_{11}} + \rho \left( 35124
    \right)_{_{12}} 
  \\
    &  & \\  
    &+& \rho \left( 35124 \right)_{_{12}} + \rho \left( 15234
    \right)_{_{10}}\\
    &  & \\
    & + & \rho \left( 15243 \right)_{_{10}} + \rho \left( 25413
    \right)_{_{11}} + \rho \left( 25341 \right)_{_{11}} + \rho \left( 35421
    \right)_{_{12}} 
   \\
    &  & \\ 
    &+& \rho \left( 35142 \right)_{_{12}} + \rho \left( 15432
    \right)_{_{10}}\,,
  \end{eqnarray*}
  where we denote by $\rho \left( \cdot \cdot \cdot \cdot \cdot \right)_{_j}$
  the terms that we sum up together using the symmetries of $\rho$. We obtain
  \begin{eqnarray*}
    &  & \left( \tmop{Circ} \tmop{Sym}_{3, 4, 5} \Theta \right) \left( 12345
    \right)\\
    &  & \\
    & = & 2 \rho \left( 13245 \right) + 2 \rho \left( 23415 \right) + 2 \rho
    \left( 12435 \right)\\
    &  & \\
    & + & 2 \rho \left( 13254 \right) + 2 \rho \left( 23514 \right) + 2 \rho
    \left( 12534 \right)\\
    &  & \\
    & + & 3 \rho \left( 14235 \right) + 3 \rho \left( 24315 \right) + 3 \rho
    \left( 34125 \right)\\
    &  & \\
    & + & 3 \rho \left( 15234 \right) + 3 \rho \left( 25314 \right) + 3 \rho
    \left( 35124 \right) \,.
  \end{eqnarray*}
  We conclude the expression
  \begin{eqnarray}
    &  & \left[ \tmop{Circ}_{_{_{_{_{}}}}} \tmop{Sym}_{3, 4, 5} d_1^{\nabla}
    \left( \nabla R^{\nabla} \right)_2 \right] \left( 12345 \right)\nonumber
    \\\nonumber
    &  & \\
    & = & 2 \rho \left( 13245 \right) + 2 \rho \left( 23415 \right) + 2 \rho
    \left( 12435 \right)\nonumber
    \\\nonumber
    &  & \\
    & + & 2 \rho \left( 13254 \right) + 2 \rho \left( 23514 \right) + 2 \rho
    \left( 12534 \right)\nonumber
    \\\nonumber
    &  & \\
    & + & 2 \rho \left( 14235 \right) + 2 \rho \left( 24315 \right) + 2 \rho
    \left( 34125 \right)\nonumber
    \\\nonumber
    &  & \\
    & + & 2 \rho \left( 15234 \right) + 2 \rho \left( 25314 \right) + 2 \rho
    \left( 35124 \right)\, .\label{crcSmdDR}
  \end{eqnarray}
We expand now the term 
$$
\tmop{Circ}_{_{_{_{_{}}}}} \tmop{Sym}_{3, 4, 5} 
  ( \tilde{R}^{\nabla} \wedge_1 \tilde{R}^{\nabla} )\,.
  $$ From now on
  we will denote for notation simplicity $\left( 123 \right) \equiv R^{\nabla}
  \left( 123 \right)$ and 
  $$\left[ 123 \right] \assign \left( 123 \right) +
  \left( 132 \right)\,.
  $$ We observe the identities
  \begin{eqnarray*}
    \left[ \tmop{Sym}_{3, 4, 5}  ( \tilde{R}^{\nabla} \wedge_1
    \tilde{R}^{\nabla} )_{_{_{}}} \right] \left( 12345 \right) & = &
    \left[ 1 \left[ 234 \right] 5 \right] - \left[ 2 \left[ 134 \right] 5
    \right]\\
    &  & \\
    & + & \left[ 1 \left[ 235 \right] 4 \right] - \left[ 2 \left[ 135 \right]
    4 \right]\\
    &  & \\
    & + & \left[ 1 \left[ 243 \right] 5 \right] - \left[ 2 \left[ 143 \right]
    5 \right]\\
    &  & \\
    & + & \left[ 1 \left[ 245 \right] 3 \right] - \left[ 2 \left[ 145 \right]
    3 \right]\\
    &  & \\
    & + & \left[ 1 \left[ 253 \right] 4 \right] - \left[ 2 \left[ 153 \right]
    4 \right]\\
    &  & \\
    & + & \left[ 1 \left[ 254 \right] 3 \right] - \left[ 2 \left[ 154 \right]
    3 \right]\,,
  \end{eqnarray*}
  \begin{eqnarray*}
    \left[ \tmop{Sym}_{3, 4, 5}  ( \tilde{R}^{\nabla} \wedge_1
    \tilde{R}^{\nabla} )_{_{_{}}} \right] \left( 23145 \right) & = &
    \left[ 2 \left[ 314 \right] 5 \right] - \left[ 3 \left[ 214 \right] 5
    \right]\\
    &  & \\
    & + & \left[ 2 \left[ 315 \right] 4 \right] - \left[ 3 \left[ 215 \right]
    4 \right]\\
    &  & \\
    & + & \left[ 2 \left[ 341 \right] 5 \right] - \left[ 3 \left[ 241 \right]
    5 \right]\\
    &  & \\
    & + & \left[ 2 \left[ 345 \right] 1 \right] - \left[ 3 \left[ 245 \right]
    1 \right]\\
    &  & \\
    & + & \left[ 2 \left[ 351 \right] 4 \right] - \left[ 3 \left[ 251 \right]
    4 \right]\\
    &  & \\
    & + & \left[ 2 \left[ 354 \right] 1 \right] - \left[ 3 \left[ 254 \right]
    1 \right]\,,
  \end{eqnarray*}
  \begin{eqnarray*}
    \left[ \tmop{Sym}_{3, 4, 5}  ( \tilde{R}^{\nabla} \wedge_1
    \tilde{R}^{\nabla} )_{_{_{}}} \right] \left( 31245 \right) & = &
    \left[ 3 \left[ 124 \right] 5 \right] - \left[ 1 \left[ 324 \right] 5
    \right]\\
    &  & \\
    & + & \left[ 3 \left[ 125 \right] 4 \right] - \left[ 1 \left[ 325 \right]
    4 \right]\\
    &  & \\
    & + & \left[ 3 \left[ 142 \right] 5 \right] - \left[ 1 \left[ 342 \right]
    5 \right]\\
    &  & \\
    & + & \left[ 3 \left[ 145 \right] 2 \right] - \left[ 1 \left[ 345 \right]
    2 \right]\\
    &  & \\
    & + & \left[ 3 \left[ 152 \right] 4 \right] - \left[ 1 \left[ 352 \right]
    4 \right] \\
    &  & \\
    & + & \left[ 3 \left[ 154 \right] 2 \right] - \left[ 1 \left[ 354 \right]
    2 \right]\, .
  \end{eqnarray*}
  Summing up using the symmetries of $[\cdots]$ and $(\cdots)$ we obtain 
  \begin{eqnarray*}
    &&\left[ \tmop{Circ} \tmop{Sym}_{3, 4, 5}  ( \tilde{R}^{\nabla}
    \wedge_1 \tilde{R}^{\nabla} )_{_{_{}}} \right] \left( 12345 \right)
   \\
    &  & \\ 
    & = & 6 \left[ 1 \left( 234 \right) 5 \right] + 6 \left[ 2 \left( 314
    \right) 5 \right] + 6 \left[ 3 \left( 124 \right) 5 \right]\\
    &  & \\
    & + & 6 \left[ 1 \left( 235 \right) 4 \right] + 6 \left[ 2 \left( 315
    \right) 4 \right] + 6 \left[ 3 \left( 125 \right) 4 \right]\\
    &  & \\
    & + & 2 \left[ 1 \left[ 245 \right] 3 \right]_{_1} - 2 \left[ 2 \left[
    145 \right] 3 \right]_{_2}\\
    &  & \\
    & + & 2 \left[ 2 \left[ 345 \right] 1 \right]_{_3} - 2 \left[ 3 \left[
    245 \right] 1 \right]_{_1}\\
    &  & \\
    & + & 2 \left[ 3 \left[ 145 \right] 2 \right]_{_2} - 2 \left[ 1 \left[
    345 \right] 2 \right]_{_3} \,.
 \end{eqnarray*}   
We combine now the terms $[\cdot[\cdots]\cdot]_{_j}$ for each $j=1,2,3$ and we explicit and simplify them by using the algebraic Bianchi identity. We obtain 
 \begin{eqnarray*}   
  &&\left[ \tmop{Circ} \tmop{Sym}_{3, 4, 5}  ( \tilde{R}^{\nabla}
    \wedge_1 \tilde{R}^{\nabla} )_{_{_{}}} \right] \left( 12345 \right)
   \\
    &  & \\ 
    & = & 6 \left[ 1 \left( 234 \right) 5 \right] + 6 \left[ 2 \left( 314
    \right) 5 \right] + 6 \left[ 3 \left( 124 \right) 5 \right]\\
    &  & \\
    & + & 6 \left[ 1 \left( 235 \right) 4 \right] + 6 \left[ 2 \left( 315
    \right) 4 \right] + 6 \left[ 3 \left( 125 \right) 4 \right]\\
    &  & \\
    & + & 6 \left( 13 \left[ 245 \right] \right) + 6 \left( 32 \left[ 145
    \right] \right) + 6 \left( 21 \left[ 345 \right] \right) \,.
  \end{eqnarray*}
  Expanding further we obtain the complete expansion 
  \begin{eqnarray*}
   && \left[ \tmop{Circ} \tmop{Sym}_{3, 4, 5}  ( \tilde{R}^{\nabla}
    \wedge_1 \tilde{R}^{\nabla} )_{_{_{}}} \right] \left( 12345 \right)
  \\
    &  & \\  
    & = & 6 \left( 1 \left( 234 \right) 5 \right) + 6 \left( 15 \left( 234
    \right) \right)\\
    &  & \\
    & + & 6 \left( 2 \left( 314 \right) 5 \right) + 6 \left( 25 \left( 314
    \right) \right)\\
    &  & \\
    & + & 6 \left( 3 \left( 124 \right) 5 \right) + 6 \left( 35 \left( 124
    \right) \right)\\
    &  & \\
    & + & 6 \left( 1 \left( 235 \right) 4 \right) + 6 \left( 14 \left( 235
    \right) \right)\\
    &  & \\
    & + & 6 \left( 2 \left( 315 \right) 4 \right) + 6 \left( 24 \left( 315
    \right) \right)\\
    &  & \\
    & + & 6 \left( 3 \left( 125 \right) 4 \right) + 6\left( 34 \left( 125
    \right) \right)\\
    &  & \\
    & + & 6 \left( 13 \left( 245 \right) \right) + 6 \left( 13 \left( 254
    \right) \right)\\
    &  & \\
    & + & 6 \left( 32 \left( 145 \right) \right) + 6 \left( 32 \left( 154
    \right) \right)\\
    &  & \\
    & + & 6 \left( 21 \left( 345 \right) \right) + 6 \left( 21 \left( 354
    \right) \right) \,.
  \end{eqnarray*}
  Expanding the terms $\rho$ present in the expression (\ref{crcSmdDR}) we obtain the complete expansion of the term
  $$
  \left\{ \tmop{Circ} \tmop{Sym}_{3, 4, 5} \left[ 3 d_1^{\nabla}
    \left( \nabla R^{\nabla} \right)_2 -_{_{_{_{_{}}}}} 2 \tilde{R}^{\nabla}
    \wedge_1 \tilde{R}^{\nabla} \right] \right\} \left( 12345 \right)\,,
  $$
  given by
  \begin{eqnarray*}
    &  & \left\{ \tmop{Circ} \tmop{Sym}_{3, 4, 5} \left[ 3 d_1^{\nabla}
    \left( \nabla R^{\nabla} \right)_2 -_{_{_{_{_{}}}}} 2 \tilde{R}^{\nabla}
    \wedge_1 \tilde{R}^{\nabla} \right] \right\} \left( 12345 \right)\\
    &  & \\
    & = & 6 \left( 13 \left( 245 \right) \right)_{_1} - 6 \left( \left( 132
    \right) 45 \right)_{_2} - 6 \left( 2 \left( 134 \right) 5 \right)_{_3} - 6
    \left( 24 \left( 135 \right) \right)_{_4}\\
    &  & \\
    & + & 6 \left( 23 \left( 415 \right) \right)_{_5} - 6 \left( \left( 234
    \right) 15 \right)_{_6} - 6 \left( 4 \left( 231 \right) 5 \right)_{_2} - 6
    \left( 41 \left( 235 \right) \right)_{_7}\\
    &  & \\
    & + & 6 \left( 12 \left( 435 \right) \right)_{_8} - 6 \left( \left( 124
    \right) 35 \right)_{_9} - 6 \left( 4 \left( 123 \right) 5 \right)_{_2} - 6
    \left( 43 \left( 125 \right) \right)_{_{10}}\\
    &  & \\
    & + & 6 \left( 13 \left( 254 \right) \right)_{_{11}} - 6 \left( \left(
    132 \right) 54 \right)_{_{12}} - 6 \left( 2 \left( 135 \right) 4
    \right)_{_{13}} - 6 \left( 25 \left( 134 \right) \right)_{_{14}}\\
    &  & \\
    & + & 6 \left( 23 \left( 514 \right) \right)_{_{15}} - 6 \left( \left(
    235 \right) 14 \right)_{_{16}} - 6 \left( 5 \left( 231 \right) 4
    \right)_{_{12}} - 6 \left( 51 \left( 234 \right) \right)_{_{17}}\\
    &  & \\
    & + & 6 \left( 12 \left( 534 \right) \right)_{_{18}} - 6 \left( \left(
    125 \right) 34 \right)_{_{19}} - 6 \left( 5 \left( 123 \right) 4
    \right)_{_{12}} - 6 \left( 53 \left( 124 \right) \right)_{_{20}} \\
    &  & \\
    & + & 6 \left( 14 \left( 235 \right) \right)_{_{7}} - 6 \left( \left(
    142 \right) 35 \right)_{_9} - 6 \left( 2 \left( 143 \right) 5 \right)_{_3}
    - 6 \left( 23 \left( 145 \right) \right)_{_5}\\
    &  & \\
    & + & 6 \left( 24 \left( 315 \right) \right)_{_{4}} - 6 \left( \left(
    243 \right) 15 \right)_{_6} - 6 \left( 3 \left( 241 \right) 5 \right)_{_9}
    - 6 \left( 31 \left( 245 \right) \right)_{_1}\\
    &  & \\
    & + & 6 \left( 34 \left( 125 \right) \right)_{_{10}} - 6 \left( \left(
    341 \right) 25 \right)_{_3} - 6 \left( 1 \left( 342 \right) 5 \right)_{_6}
    - 6 \left( 12 \left( 345 \right) \right)_{_8}\\
    &  & \\
    & + & 6 \left( 15 \left( 234 \right) \right)_{_{17}} - 6 \left( \left( 152
    \right) 34 \right)_{_{19}} - 6 \left( 2 \left( 153 \right) 4
    \right)_{_{13}} - 6 \left( 23 \left( 154 \right) \right)_{_{15}}\\
    &  & \\
    & + & 6 \left( 25 \left( 314 \right) \right)_{_{14}} - 6 \left( \left( 253
    \right) 14 \right)_{_{16}} - 6 \left( 3 \left( 251 \right) 4
    \right)_{_{19}} - 6 \left( 31 \left( 254 \right) \right)_{_{11}}\\
    &  & \\
    & + & 6 \left( 35 \left( 124 \right) \right)_{_{20}} - 6 \left( \left( 351
    \right) 24 \right)_{_{13}} - 6 \left( 1 \left( 352 \right) 4
    \right)_{_{16}} - 6 \left( 12 \left( 354 \right) \right)_{_{18}}\\
    &  & \\
    & - & 12 \left( 1 \left( 234 \right) 5 \right)_{_6} - 12 \left( 15 \left(
    234 \right) \right)_{_{17}}\\
    &  & \\
    & - & 12 \left( 2 \left( 314 \right) 5 \right)_{_3} - 12 \left( 25 \left(
    314 \right) \right)_{_{14}}\\
    &  & \\
    & - & 12 \left( 3 \left( 124 \right) 5 \right)_{_9} - 12 \left( 35 \left(
    124 \right) \right)_{_{20}}\\
    &  & \\
    & - & 12 \left( 1 \left( 235 \right) 4 \right)_{_{16}} - 12 \left( 14
    \left( 235 \right) \right)_{_7}\\
    &  & \\
    & - & 12 \left( 2 \left( 315 \right) 4 \right)_{_{13}} - 12 \left( 24
    \left( 315 \right) \right)_{_4}\\
    &  & \\
    & - & 12 \left( 3 \left( 125 \right) 4 \right)_{_{19}} - 12 \left( 34
    \left( 125 \right) \right)_{_{10}}
     -  12 \left( 13 \left( 245 \right) \right)_{_1} - 12 \left( 13 \left(
    254 \right) \right)_{_{11}}\\
    &  & \\
    & - & 12 \left( 32 \left( 145 \right) \right)_{_5} - 12 \left( 32 \left(
    154 \right) \right)_{_{15}}
 -  12 \left( 21 \left( 345 \right) \right)_{_8} - 12 \left( 21 \left(
    354 \right) \right)_{_{18}} \,,
  \end{eqnarray*}
where as before we denote by $\left( \cdot \cdot \cdot \cdot \cdot \right)_{_j}$
  the terms that we sum up together using the symmetries of the curvature tensor $R^{\nabla}$. 
All the terms summed up together cancel up. This is obvious for all the sub indexes $j$ with the exception of $j = 3, 6, 9, 13, 16, 19$ for which me must provide the detail of the computation. Indeed for $j = 3$ we have
  \begin{eqnarray*}
    &  &  - 6 \left( \left( 341 \right)
    25 \right) - 6 \left( 2 \left( 134 \right) 5 \right) - 6 \left( 2 \left(
    143 \right) 5 \right) - 12 \left( 2 \left( 314 \right) 5 \right)\\
    &  & \\
    & = &  6 \left( 2 \left( 341
    \right) 5 \right) + 6 \left( 2 \left( 413 \right) 5 \right)-6 \left( 2 \left( 314
    \right) 5 \right)\\
    &  & \\
    & = &  - 6 \left( 2 \left( 134
    \right) 5 \right)- 6 \left( 2 \left( 314
    \right) 5 \right)\\
    &  & \\
    & = & 0 \,.
  \end{eqnarray*}
  For $j = 6$ we have
  \begin{eqnarray*}
    &  &  - 12 \left( 1 \left( 234
    \right) 5 \right) - 6 \left( \left( 234 \right) 15 \right) - 6 \left(
    \left( 243 \right) 15 \right) - 6 \left( 1 \left( 342 \right) 5 \right)\\
    &  & \\
    & = &  6  \left( 1 \left(
    243 \right) 5 \right)  + 6 \left( 1 \left( 432 \right) 5 \right)-6 \left( 1 \left( 234
    \right) 5 \right)\\
    &  & \\
    & = &  - 6 \left( 1 \left( 324
    \right) 5 \right)-6 \left( 1 \left( 234
    \right) 5 \right)\\
    &  & \\
    & = &  0 \,.
  \end{eqnarray*}
  For $j = 9$ we have
  \begin{eqnarray*}
    &  &  - 6 \left( \left( 124 \right)
    35 \right) - 6 \left( \left( 142 \right) 35 \right) - 6 \left( 3 \left(
    241 \right) 5 \right) - 12 \left( 3 \left( 124 \right) 5 \right)\\
    &  & \\
    & = &  6 \left( 3 \left( 142
    \right) 5 \right) + 6 \left( 3 \left( 421 \right) 5 \right)-6 \left( 3 \left( 124
    \right) 5 \right) \\
    &  & \\
    & = & - 6 \left( 3 \left( 214
    \right) 5 \right)-6 \left( 3 \left( 124
    \right) 5 \right) \\
    &  & \\
    & = & 0\,.
  \end{eqnarray*}
  For $j = 13$ we have
  \begin{eqnarray*}
    &  & - 6 \left( \left( 351 \right)
    24 \right) - 6 \left( 2 \left( 135 \right) 4 \right) - 6 \left( 2 \left(
    153 \right) 4 \right) - 12 \left( 2 \left( 315 \right) 4 \right)\\
    &  & \\
    & = &  6 \left( 2 \left( 351
    \right) 4 \right) + 6 \left( 2 \left( 513 \right) 4 \right)-6 \left( 2 \left( 315
    \right) 4 \right)\\
    &  & \\
    & = & - 6 \left( 2 \left( 135
    \right) 4 \right)-6 \left( 2 \left( 315
    \right) 4 \right)\\
    &  & \\
    & = & 0 \,.
  \end{eqnarray*}
  For $j = 16$ we have
  \begin{eqnarray*}
    &  &  - 12 \left( 1 \left( 235
    \right) 4 \right) - 6 \left( \left( 235 \right) 14 \right) - 6 \left(
    \left( 253 \right) 14 \right) - 6 \left( 1 \left( 352 \right) 4 \right)\\
    &  & \\
    & = &  6 \left( 1 \left( 253
    \right) 4 \right) + 6 \left( 1 \left( 532 \right) 4 \right)-6 \left( 1 \left( 235
    \right) 4 \right) \\
    &  & \\
    & = &  - 6 \left( 1 \left( 325
    \right) 4 \right)-6 \left( 1 \left( 235
    \right) 4 \right) \\
    &  & \\
    & = &  0\,.
  \end{eqnarray*}
  For $j = 19$ we have
  \begin{eqnarray*}
    &  & - 6 \left( \left( 125 \right)
    34 \right) - 6 \left( \left( 152 \right) 34 \right) - 6 \left( 3 \left(
    251 \right) 4 \right) - 12 \left( 3 \left( 125 \right) 4 \right)\\
    &  & \\
    & = &  6 \left( 3 \left( 152
    \right) 4 \right) + 6 \left( 3 \left( 521 \right) 4 \right)-6 \left( 3 \left( 125
    \right) 4 \right)\\
    &  & \\
    & = & - 6 \left( 3 \left( 215
    \right) 4 \right)-6 \left( 3 \left( 125
    \right) 4 \right)\\
    &  & \\
    & = &  0 \,.
  \end{eqnarray*}
We infer the required identity (\ref{KeyFrstIntegReduc}) in the statement of proposition \ref{Vanish}.
\end{proof}

\section{The almost complex structure associated to a connection over the tangent bundle}

This section is not needed for the proof of the results in the paper. We include it in order to clarify the integrability of an $M$-totally real almost
complex structure over $T_M$  associated to the horizontal distribution of a linear connection. 
We include it also to remind and to prove in modern terms a well known result due to Dombrowsky \cite{Dom}. 

It is well known (see \cite{Dom}) that we can construct an $M$-totally real almost
complex structure over $T_M$ by using the horizontal distribution $\mathcal{H}
\subset T_M$ associated to a linear connection $\nabla$ acting on the sections
of $T_M$. Indeed in this case we set $\alpha_{\eta} \assign H_{\eta}$ and
$B_{\eta} \assign \mathbbm{I}_{T_M, \pi \left( \eta \right)}$, where $\eta
\longmapsto H_{\eta}$ is the horizontal map associated to $\mathcal{H}$. We
will denote $J_{\mathcal{H}} \assign J_A$. If we define for any $\eta \in
T_{M, p}$ the vertical projection $\tmop{Vert}_{\eta} : T_{T_M, \eta}
\longrightarrow T_{T_{M, p}, \eta}$ as
\begin{eqnarray*}
  \tmop{Vert}_{\eta}  \assign  \mathbbm{I}_{T_{T_M, \eta}} - H_{\eta}
  d_{\eta} \pi\,,
\end{eqnarray*}
where $\pi : T_M \longrightarrow M$ is the canonical projection, then
\begin{eqnarray*}
  J_{\mathcal{H}, \eta} \assign  - H_{\eta}\, T^{- 1}_{\eta}
  \tmop{Vert}_{\eta} + \,T_{\eta}\, d_{\eta} \pi \,.
\end{eqnarray*}
If we decompose any vector $\xi \in \mathbbm{C}T_{T_M, \eta}$ in its
horizontal and vertical parts $\xi = \xi^h + \xi^v$ with $\xi^v \assign
\tmop{Vert}_{\eta} \left( \xi \right)$ then we have the expressions
\begin{eqnarray*}
  J_{\mathcal{H}, \eta} \xi & = & - H_{\eta} \,T^{- 1}_{\eta} \xi^v + T_{\eta}\,
  d_{\eta} \pi \,\xi^h,\\
  &  & \\
  \left( J_{\mathcal{H}, \eta} \xi \right)^h & = & - H_{\eta} \,T^{- 1}_{\eta}
  \xi^v,\\
  &  & \\
  \left( J_{\mathcal{H}, \eta} \xi \right)^v & = & T_{\eta} \,d_{\eta} \pi\,
  \xi^h .
\end{eqnarray*}
We infer
\begin{eqnarray*}
  \xi_{J_{\mathcal{H}}}^{0, 1} \left( \eta \right) & = & \frac{1}{2}  \left[
  \xi^h - i \,H_{\eta}  \, T^{- 1}_{\eta} \xi^v + \xi^v + i\, T_{\eta_{_{_{}}}}
  d_{\eta} \pi_{} \,\xi^h \right]\\
  &  & \\
  & = & \frac{1}{2}  \left[ \xi^h +_{_{_{_{}}}} H_{\eta}\, \mu + i\, T_{\eta}
  \left( d_{\eta} \pi_{} \xi^h + \mu \right) \right],
\end{eqnarray*}
with $\mu \assign - i\, T^{- 1}_{\eta} \xi^v$. We notice also the identity
\begin{equation}
  \label{CXConnDist} T_{T_M, J_{\mathcal{H}}, \eta}^{0, 1} = \frac{1}{2} 
  \left( H_{\eta} + i \,T_{\eta} \right) \mathbbm{C}T_{M, p}\,,
\end{equation}
for any any $\eta \in T_{X, p}$. The distribution $T_{T_M,
J_{\mathcal{H}}}^{0, 1}$ is horizontal, but the associated map does not
satisfies the condition (\ref{multHoriz}) of linear connections thanks to the
identity (\ref{transl}). Therefore this distribution does not identify a
linear connection. However its integrability implies that the vector bundle
$T_M$ is flat. Indeed we have the following well known lemma due to
Dombrowsky \cite{Dom}.

\begin{lemma}
  \label{affineLm}The torsion form $\tau^{J_{\mathcal{H}}}$ of the almost
  complex structure $J_{\mathcal{H}}$ satisfies at the point $\eta \in T_M$ in
  the directions $V_1, V_2 \in T_{T_M, J_{\mathcal{H}}, \eta}^{0, 1}$ the
  identity
  \begin{eqnarray*}
    8 \tau^{J_{\mathcal{H}}} \left( V_1, V_2 \right) \left( \eta \right) & = &
    - \;H_{\eta} \left[ \tau^{\nabla} \left( v_1, v_2 \right) + i\, R^{\nabla}
    \left( v_1, v_2 \right) \eta \right]\\
    &  & \\
    & + & T_{\eta} \left[ i \,\tau^{\nabla} \left( v_1, v_2 \right)_{_{_{_{}}}}
    - R^{\nabla} \left( v_1, v_2 \right) \eta \right],
  \end{eqnarray*}
  where $R^{\nabla} \assign \nabla^2$ is the complex linear extension of the
  curvature tensor of $\nabla$, where $\tau^{\nabla}$ is the torsion of the
  complex connection $\nabla$ and where $v_j \assign d_{\eta} \pi_{} V_j$, $j
  = 1, 2$. In particular $J_{\mathcal{H}}$ is a complex structure if and only
  if the linear connection $\nabla$ is flat and torsion free. 
\end{lemma}

\begin{proof}
  Let $\xi_j$ be vector field local extensions of $v_j$ such that $[\xi_1,
  \xi_2] \pi \left( \eta \right) = 0$. Then
  \begin{eqnarray*}
    \Xi_j  \assign  \frac{1}{2}  \left( H + i\, T \right) \xi_j\,,
  \end{eqnarray*}
  are local vector field extensions of $V_j$. We expand the bracket
  \begin{eqnarray*}
    4\, [\Xi_1, \Xi_2] \left( \eta \right) & = & \left( [H \xi_1, H \xi_2]
    \noplus + i\, [H \xi_1, T \xi_2] \noplus + i\, [T \xi_1, H \xi_2] \noplus - [T
    \xi_1, T \xi_2] \noplus_{_{_{_{}}}} \right) \left( \eta \right)\\
    &  & \\
    & = & H_{\eta} [\xi_1, \xi_2] - T_{\eta} [R^{\nabla} \left( v_1, v_2
    \right) \eta] + i\, T_{\eta} [\nabla_{\xi_1} \xi_2 - \nabla_{\xi_2} \xi_1]\, .
  \end{eqnarray*}
  The last equality follows from to the computation at the end of the proof of
  lemma \ref{TorsCurv} and thanks to the identity (\ref{BraketConn}) in the
  appendix. (We notice that $[T \xi_1, T \xi_2] \equiv 0$, since the vector
  fields $T \xi_1$ are tangent constant along the fibers). Thanks to the
  assumption $[\xi_1, \xi_2] \pi \left( \eta \right) = 0$, we infer the
  equality
  \begin{eqnarray*}
    4\, [\Xi_1, \Xi_2] \left( \eta \right)  =  T_{\eta} \left[ i\, \tau^{\nabla}
    \left( v_1, v_2 \right)_{_{_{_{}}}} - R^{\nabla} \left( v_1, v_2 \right)
    \eta \right] .
  \end{eqnarray*}
  The required formula follows from the identity
  \begin{eqnarray*}
    \xi_{J_{\mathcal{H}}}^{1, 0} \left( \eta \right)  =  \frac{1}{2}  \left[
    \xi^h + i\, H_{\eta}  \, T^{- 1}_{\eta} \xi^v + \xi^v - i\, T_{\eta}\, d_{\eta}
    \pi_{_{_{_{_{}}}}} \xi^h \right] .
  \end{eqnarray*}
  The fact that that the distribution $T_{T_M, J_{\mathcal{H}}}^{1, 0}$ is
  horizontal implies that $\tau^{J_{\mathcal{H}}} \left( V_1, V_2 \right)
  \left( \eta \right)$ vanishes for all $V_j$ if and only if the quantity
  \[ \tau^{\nabla} \left( v_1, v_2 \right) + i\, R^{\nabla} \left( v_1, v_2
     \right) \eta\,, \]
  vanishes for all $v_j$. In particular for real vectors $v_j$ this implies
  that $R^{\nabla}$ and $\tau^{\nabla}$ vanish at the point $\pi \left( \eta
  \right)$.
\end{proof}

We observe that a connection over $T_M$ is flat and torsion free if and only
if there exist local parallel frames with vanishing Lie brackets.

\section{Appendix}
In this appendix we provide some well known basic facts about the geometric theory of linear connections needed for the reading of the paper. (See also \cite{Gau}). We strongly recommend its reading even to experts.
\subsection{The horizontal distribution associated to a linear connection}

We start with the following fact.

\begin{lemma}
  \label{addLm}Let $\nabla$ be a linear connection acting on sections of a
  vector bundle $E$ over a manifold $M$. Then the linear map
  $$
    T_{M, p} \ni \xi \longmapsto H_{\eta} \left( \xi \right)  \assign  d_p
    \sigma \left( \xi \right) - T_{\eta} \nabla_{\xi} \sigma \in T_{E, \eta}\,,
  $$
  is independent of the sections $\sigma$ such that $\sigma \left( p \right) =
  \eta$.
\end{lemma}

\begin{proof}
  Let $e = \left( e_k \right)^r_{k = 1}$ be a local frame of $E$ over an open
  set $U \subset M$. We consider the local expression $\sigma = e \cdot f$
  with $f \in C^1 \left( U, \mathbbm{R}^r \right)$. Let $A \in C^{\infty}
  \left( U, T^{\ast}_M \otimes \tmop{Matrix}_{r \times r} \left( \mathbbm{R}
  \right) \right)$ be the connection form of $\nabla$ with respect to the
  local frame $e$, i.e $\nabla e = e \cdot A$. Then $\nabla \sigma = e \otimes
  (d f + A \cdot f)$. If we denote by $\theta_e : U \times \mathbbm{R}^r
  \longrightarrow E_{\mid U}$ then the differential of this map at the point
  $\left( p, f \left( p \right) \right)$ provides an isomorphism
  \begin{eqnarray*}
    d_{p, f \left( p \right)} \theta_e : T_{U, p} \oplus \mathbbm{R}^r &
    \longrightarrow & T_{E, \sigma \left( p \right)} .
  \end{eqnarray*}
  With respect to it, the equality hold
  \begin{eqnarray*}
    d_{p, f \left( p \right)} \theta_e  \left[ \xi \oplus d_p f \left( \xi
    \right) \right] & = & d_p \sigma \left( \xi \right) .
  \end{eqnarray*}
  We observe now the linear identity $d \tau_{\sigma \left( p \right)} \cdot
  d_{p, 0} \theta_{e \mid 0 \oplus \mathbbm{R}^r} = d_{p, f \left( p \right)}
  \theta_{e \mid 0 \oplus \mathbbm{R}^r}$. We infer
  \begin{equation}
    \label{Ttrivial} T_{\sigma \left( p \right)} \cdot \theta_{e \mid \{p\}
    \times \mathbbm{R}^r} = d_{p, f \left( p \right)} \theta_{e \mid 0 \oplus
    \mathbbm{R}^r},
  \end{equation}
  and
  \begin{eqnarray*}
    &&T_{\sigma \left( p \right)} \left[ e \left( p \right) \cdot \left( d_p f
    \left( \xi \right) + A \left( \xi \right) \cdot f \left( p
    \right)_{_{_{_{}}}} \right) \right] \\
    \\
    &=&
    d_{p, f \left( p \right)}
    \theta_e  \left[ 0 \oplus \left( d_p f \left( \xi \right) + A \left( \xi
    \right) \cdot f \left( p \right)_{_{_{_{}}}} \right) \right],\\
    &  & \\
    T_{\sigma \left( p \right)} \nabla_{\xi} \sigma & = & d_{p, f \left( p
    \right)} \theta_e  \left[ 0 \oplus \left( d_p f \left( \xi \right) + A
    \left( \xi \right) \cdot f \left( p \right)_{_{_{_{}}}} \right) \right] .
  \end{eqnarray*}
  Thus
  \begin{eqnarray*}
    H_{\sigma \left( p \right)} \left( \xi \right) & = & d_{p, f \left( p
    \right)} \theta_e  \left[ \xi \oplus \left( - A \left( \xi \right) \cdot f
    \left( p \right)_{_{_{_{}}}} \right) \right],
  \end{eqnarray*}
  i.e. if $\eta = e \cdot h$, then
  \begin{eqnarray*}
    H_{\eta} \left( \xi \right) & = & d_{p, h} \theta_e  \left[ \xi \oplus
    \left( - A \left( \xi \right) \cdot h_{_{_{_{}}}} \right) \right],
  \end{eqnarray*}
  which shows the required conclusion. 
\end{proof}

Let $\pi_E : E \longrightarrow M$ be the projection map and notice the
equality $\tmop{Ker} d_{\eta} \pi_E = T_{E_p, \eta}$, for any $\eta \in E_p$.
The identity $\pi_E \circ \sigma = \tmop{id}_M$ implies
\begin{eqnarray*}
  d_{\sigma \left( p \right)} \pi_E \circ d_p \sigma \left( \xi \right) & = &
  \xi \,.
\end{eqnarray*}
We deduce the identity $d_{\eta} \pi_E \circ H_{\eta} \left( \xi \right) =
\xi$. We define the horizontal distribution $\mathcal{H} \subset T_E$
associated to $\nabla$ as
\begin{eqnarray*}
  \mathcal{H}_{\eta} & \assign & H_{\eta} \left( T_{M, \pi_E \left( \eta
  \right)} \right) \subset T_{E, \eta} \,.
\end{eqnarray*}
We notice now that the tangent bundle of the vector bundle $E \oplus E$ is
given by the fibers
\begin{eqnarray*}
  T_{E \oplus E, \left( \eta_1, \eta_2 \right)} & = & \left\{ \left( v_1, v_2
  \right) \in T_{E, \eta_1} \oplus T_{E, \eta_2} \mid d_{\eta_1} \pi_E \left(
  v_1 \right) = d_{\eta_2} \pi_E \left( v_2 \right)_{_{_{_{_{}}}}} \right\},
\end{eqnarray*}
and that the differential of the sum bundle map $s m_{_E} : E \oplus E
\longrightarrow E$ satisfies
\begin{eqnarray*}
  d_{\left( \eta_1, \eta_2 \right)} \left( s m_{_E} \right) \left( v_1, v_2
  \right) & = & T_{\eta_1 + \eta_2} \left( T^{- 1}_{\eta_1} v_1 + T^{-
  1}_{\eta_2} v_2 \right),
\end{eqnarray*}
for any $\left( v_1, v_2 \right) \in T_{E, \eta_1} \oplus T_{E, \eta_2}$ such
that $d_{\eta_1} \pi_E \left( v_1 \right) = d_{\eta_2} \pi_E \left( v_2
\right) = 0$. We infer that for any sections $\sigma_j$ of $E$ such that
$\sigma_j \left( p \right) = \eta_j$, $j = 1, 2$, the equalities
\begin{eqnarray*}
  H_{\eta_1 + \eta_2} \left( \xi \right) & = & d_p  \left( \sigma_1 + \sigma_2
  \right) \left( \xi \right) - T_{\eta_1 + \eta_2} \nabla_{\xi}  \left(
  \sigma_1 + \sigma_2 \right)\\
  &  & \\
  & = & d_{\left( \eta_1, \eta_2 \right)} \left( s m_{_E} \right) \left( d_p
  \sigma_1 \left( \xi \right), d_p \sigma_2 \left( \xi \right)
  \right) - T_{\eta_1 + \eta_2} \nabla_{\xi} \sigma_1 - T_{\eta_1 + \eta_2}
  \nabla_{\xi} \sigma_2\\
  &  & \\
  & = & d_{\left( \eta_1, \eta_2 \right)} \left( s m_{_E} \right) \left( d_p
  \sigma_1 \left( \xi \right) - T_{\eta_1} \nabla_{\xi} \sigma_1, d_p \sigma_2
  \left( \xi \right) - T_{\eta_2} \nabla_{\xi} \sigma_2 \right)
  .
\end{eqnarray*}
hold.
We conclude that
\begin{equation}
  \label{aditivDistrib} H_{\eta_1 + \eta_2} \left( \xi \right) = d_{\left(
  \eta_1, \eta_2 \right)} \left( s m_{_E} \right) \left( H_{\eta_1} \left( \xi
  \right), H_{\eta_2} \left( \xi \right) \right) .
\end{equation}
\begin{lemma}
  \label{prdLm}For any section $\sigma \in C^1 \left( M, E \right)$ and for
  any function $u \in C^1 \left( M, \mathbbm{R} \right)$ the identity holds
  \begin{eqnarray*}
    d_p \left( u \sigma \right) & = & d_p u \otimes T_{u \sigma \left( p
    \right)} \sigma \left( p \right) \noplus \noplus + d_{\sigma \left( p
    \right)} [u \left( p \right) \mathbbm{I}_E] \cdot d_p \sigma\,,
  \end{eqnarray*}
  for any point $p \in M$.
\end{lemma}

\begin{proof}
  With the notation in the proof of lemma \ref{addLm}
  \begin{eqnarray*}
    d_p \left( u \sigma \right) \left( \xi \right) & = & d_{p, u f \left( p
    \right)} \theta_e  \left[ \xi \oplus d_p  \left( u f \right) \left( \xi
    \right)_{_{_{_{}}}} \right]\\
    &  & \\
    & = & d_{p, u f \left( p \right)} \theta_e \left\{ \xi \oplus \left[ d_p
    u \left( \xi \right) f \left( p \right) + u \left( p \right) d_p f \left(
    \xi \right)_{_{_{_{}}}} \right]_{_{_{}}} \right\}\\
    &  & \\
    & = & d_{p, u f \left( p \right)} \theta_e \left[ 0 \oplus d_p u \left(
    \xi \right) f \left( p \right)_{_{_{_{}}}} \right]\\
    &  & \\
    & + & d_{p, u f \left( p \right)} \theta_e \left[ \xi \oplus u \left( p
    \right) d_p f \left( \xi \right)_{_{_{_{}}}} \right]\\
    &  & \\
    & = & T_{u \sigma \left( p \right)} \theta_e \left( p, d_p u \left( \xi
    \right) f \left( p \right)_{_{_{_{}}}} \right) \noplus \noplus + d_p
    \left( u \left( p \right) \sigma \right) \left( \xi \right),
  \end{eqnarray*}
  thanks to (\ref{Ttrivial}). Using the identity
  \begin{equation}
    \label{difProd} d_p \left( \lambda \sigma \right) = d_{\sigma \left( p
    \right)} \left( \lambda \mathbbm{I}_E \right) \cdot d_p \sigma\,,
  \end{equation}
  for any $\lambda \in \mathbbm{R}$, we conclude
  \begin{eqnarray*}
    d_p \left( u \sigma \right) \left( \xi \right) & = & d_p u \left( \xi
    \right) T_{u \sigma \left( p \right)} \sigma \left( p \right) \noplus
    \noplus + d_{\sigma \left( p \right)} [u \left( p \right) \mathbbm{I}_E]
    \cdot d_p \sigma \left( \xi \right) .
  \end{eqnarray*}
\end{proof}

We observe also the elementary identity
\begin{equation}
  \label{transl} d_{\eta} \left( \lambda \mathbbm{I}_E \right) \cdot T_{\eta}
  = \lambda T_{\lambda \eta}\,,
\end{equation}
for all $\eta \in E$. We show now the identity
\begin{equation}
  \label{multHoriz} H_{\lambda \eta} = d_{\eta} \left( \lambda \mathbbm{I}_E
  \right) \cdot H_{\eta}\,,
\end{equation}
for all $\eta \in E$. Indeed let $\sigma$ be a section such that $\sigma
\left( p \right) = \eta$. Using (\ref{difProd}) and (\ref{transl}) we obtain
the equalities
\begin{eqnarray*}
  H_{\lambda \eta} & = & d_p \left( \lambda \sigma \right) - T_{\lambda \eta}
  \nabla \left( \lambda \sigma \right)\\
  &  & \\
  & = & d_{\sigma \left( p \right)} \left( \lambda \mathbbm{I}_E \right)
  \cdot d_p \sigma - \lambda T_{\lambda \eta} \nabla \sigma\\
  &  & \\
  & = & d_{\sigma \left( p \right)} \left( \lambda \mathbbm{I}_E \right)
  \cdot \left[ d_p \sigma - T_{\eta} \nabla \sigma \right]\\
  &  & \\
  & = & d_{\eta} \left( \lambda \mathbbm{I}_E \right) \cdot H_{\eta}\, .
\end{eqnarray*}
The property (\ref{multHoriz}) implies in particular $H_{0_p} = d_p 0_M$,
where $0_M$ is the zero section of $T_M$.

\begin{definition}
  A distribution $\mathcal{H} \subset T_E$, is called horizontal if the map
  $$
  d_{\eta} \pi_{E \mid \mathcal{H}_{\eta}} : \mathcal{H}_{\eta}
  \longrightarrow T_{M, \pi_E \left( \eta \right)}\,,
  $$ is an isomorphism for all
  $\eta \in E$.
\end{definition}

\begin{lemma}
  Any horizontal distribution $\mathcal{H} \subset T_E$, which satisfies the
  conditions $\left( \ref{aditivDistrib} \right)$ and $\left( \ref{multHoriz}
  \right)$ with $H_{\eta} \assign \left( d_{\eta} \pi_{E \mid
  \mathcal{H}_{\eta}} \right)^{- 1}$, determines a connection $\nabla$ over E
  with associated horizontal distribution $\mathcal{H}$.
\end{lemma}

\begin{proof}
  The connection $\nabla$ is defined by the formula
  \begin{eqnarray*}
    \nabla_{\xi} \sigma & = & T_{\sigma \left( p \right)}^{- 1} \cdot \left[
    d_p \sigma_{_{_{_{}}}} - H_{\sigma \left( p \right)} \right] \left( \xi
    \right),
  \end{eqnarray*}
  for any $\xi \in T_{M, p}$. The definition is well posed because
  \begin{eqnarray*}
    \left[ d_p \sigma_{_{_{_{}}}} - H_{\sigma \left( p \right)} \right] \left(
    \xi \right) & \in & T_{E_p, \sigma \left( p \right)}\,,
  \end{eqnarray*}
  which follows from the identity
  \begin{eqnarray*}
    d_{\sigma \left( p \right)} \pi_E \cdot \left[ d_p \sigma_{_{_{_{}}}} -
    H_{\sigma \left( p \right)} \right] \left( \xi \right) & = & 0 \,.
  \end{eqnarray*}
  It is obvious that the additive property of $\nabla$ is equivalent to the
  condition $\left( \ref{aditivDistrib} \right)$. We observe now that with the
  previous definition, the covariant Leibniz property
  \begin{eqnarray*}
    \nabla_{\xi}  \left( u \sigma \right) & = & d_p u \left( \xi \right)
    \sigma \left( p \right) + u \left( p \right) \nabla_{\xi} \sigma\,,
  \end{eqnarray*}
  is equivalent to the identity
  \begin{eqnarray*}
    &&d_p \left( u \sigma \right) \left( \xi \right) - H_{u \sigma \left( p
    \right)} \left( \xi \right) 
    \\
    \\& = & T_{u \sigma \left( p \right)}  \left\{
    d_p u \left( \xi \right) \sigma \left( p \right)_{_{_{_{_{}}}}} + u \left(
    p \right) T_{\sigma \left( p \right)}^{- 1} \cdot \left[ d_p \sigma \left(
    \xi \right)_{_{_{_{}}}} - H_{\sigma \left( p \right)} \left( \xi
    \right)_{_{_{_{}}}} \right]  \right\} .
  \end{eqnarray*}
  We develop the right hand side using (\ref{transl}). We infer that the
  previous identity is equivalent to the following one
  \begin{eqnarray*}
    d_p \left( u \sigma \right) \left( \xi \right) - H_{u \sigma \left( p
    \right)} \left( \xi \right) & = & d_p u \left( \xi \right) T_{u \sigma
    \left( p \right)} \sigma \left( p \right)\\
    &  & \\
    & + & d_{\sigma \left( p \right)} \left[ u \left( p \right)_{_{_{_{}}}}
    \mathbbm{I}_E \right] \cdot \left[ d_p \sigma \left( \xi
    \right)_{_{_{_{}}}} - H_{\sigma \left( p \right)} \left( \xi
    \right)_{_{_{_{}}}} \right] .
  \end{eqnarray*}
  The later hold true thanks to lemma \ref{prdLm} and the assumption
  (\ref{multHoriz}).
\end{proof}

The data of a smooth horizontal distribution over $E$ coincides with the one
of section 
$$
H \in C^{\infty} \left( E \nocomma, \pi^{\ast}_E T^{\ast}_M
\otimes T_E \right)
$$ 
such that $d \pi_E \cdot H =\mathbbm{I}_{\pi^{\ast}_E
T_E}$. (We notice that $d \pi_E \in C^{\infty} \left( E \nocomma, T^{\ast}_E
\otimes \pi^{\ast}_E T_M \right)$). Such type of section determines a
connection if and only if it satisfies the identity $\left( \ref{multHoriz}
\right)$.

For any vector $\Xi \in T_{E, \eta}$ we denote by 
$$
\gamma_{\eta}^{\mathcal{H}}
\left( \Xi \right) \assign \Xi - H_{\eta} \circ d_{\eta} \pi_E \left( \Xi
\right)\,,
$$ 
its vertical component with respect to the horizontal distribution
$\mathcal{H}$. In particular
\begin{eqnarray*}
  \gamma_{\sigma \left( p \right)}^{\mathcal{H}} \cdot d_p \sigma \left( \xi
  \right) = T_{\sigma \left( p \right)} [\nabla_{\xi} \sigma \left( p
  \right)] \,.
\end{eqnarray*}

\subsection{The induced connection}

Let $\psi : N \longrightarrow M$ be a smooth map. We define the vector bundle
$\psi^{\ast} E \assign N \times_{\psi} E$ over $N$. In explicit terms
\begin{eqnarray*}
  \psi^{\ast} E & = & \left\{ \left( y, \eta \right) \in N \times E \mid \psi
  \left( y \right) = \pi_E \left( \eta \right)_{_{_{_{}}}} \right\},
\end{eqnarray*}
and the projection over $N$ is given by the restriction of the projection to
the first factor. We will denote by $\Psi : \psi^{\ast} E \longrightarrow E$
the restriction of the projection to the second factor. The sections of
$\psi^{\ast} E$ are identified with the maps $\sigma : N \longrightarrow E$
such that $\pi_E \circ \sigma = \psi$. In this way, if $s$ is a section of $E$
then the section $\psi^{\ast} s \assign s \circ \psi$ is a section of
$\psi^{\ast} E$. More in general if $\alpha$ is a section of $\Lambda^p
T^{\ast}_M \otimes E$, we define the section $\psi^{\ast} \alpha \in \Lambda^p
T^{\ast}_N \otimes \psi^{\ast} E$ as
\begin{eqnarray*}
  \left( \psi^{\ast} \alpha \right) \left( y \right) & \assign & \left( \alpha
  \circ \psi \right) \left( y \right) \cdot \Lambda^p \left( d_y \psi \right)
  .
\end{eqnarray*}
We provide a generalization of lemma (\ref{prdLm}).

\begin{lemma}
  \label{GprdLm}For any section $\sigma \in C^1 \left( N, \psi^{\ast} E
  \right)$ and for any function $u \in C^1 \left( N, \mathbbm{R} \right)$ the
  identity holds
  \begin{eqnarray*}
    d_p \left( u \sigma \right) & = & d_p u \otimes T_{u \sigma \left( p
    \right)} \sigma \left( p \right) \noplus \noplus + d_{\sigma \left( p
    \right)} [u \left( p \right) \mathbbm{I}_E] \cdot d_p \sigma\,,
  \end{eqnarray*}
  for any point $p \in N$.
\end{lemma}

\begin{proof}
  A local frame $e$ of $E$ induces a local frame $\psi^{\ast} e$ of
  $\psi^{\ast} E$ over the open set $\psi^{- 1} \left( U \right)$. Then
  $\sigma = \psi^{\ast} e \cdot f$ with $f \in C^1 \left( \psi^{- 1} \left( U
  \right), \mathbbm{R}^r \right)$. We denote by $\theta_e : U \times
  \mathbbm{R}^r \longrightarrow E_{\mid U}$ the trivialization map induced by
  the local frame $e$ of $E$. Then the differential of this map at the point
  $\left( \psi \left( p \right), f \left( p \right) \right)$ provides an
  isomorphism
  \begin{eqnarray*}
    d_{\psi \left( p \right), f \left( p \right)} \theta_e : T_{U, \psi \left(
    p \right)} \oplus \mathbbm{R}^r & \longrightarrow & T_{E, \sigma \left( p
    \right)},
  \end{eqnarray*}
  and
  \begin{eqnarray*}
    d_p \sigma \left( \xi \right) & = & d_{\psi \left( p \right), u f \left( p
    \right)} \theta_e  \left[ d_p \psi \left( \xi \right) \oplus d_p f \left(
    \xi \right)_{_{_{_{}}}} \right] .
  \end{eqnarray*}
  for any $\xi \in T_{N, p}$ we have
  \begin{eqnarray*}
    d_p \left( u \sigma \right) \left( \xi \right) & = & d_{\psi \left( p
    \right), u f \left( p \right)} \theta_e  \left[ d_p \psi \left( \xi
    \right) \oplus d_p  \left( u f \right) \left( \xi \right)_{_{_{_{}}}}
    \right]\\
    &  & \\
    & = & d_{\psi \left( p \right), u f \left( p \right)} \theta_e \left\{
    d_p \psi \left( \xi \right) \oplus \left[ d_p u \left( \xi \right) f
    \left( p \right) + u \left( p \right) d_p f \left( \xi \right)_{_{_{_{}}}}
    \right]_{_{_{}}} \right\}\\
    &  & \\
    & = & d_{\psi \left( p \right), u f \left( p \right)} \theta_e \left[ 0
    \oplus d_p u \left( \xi \right) f \left( p \right)_{_{_{_{}}}} \right]\\
    &  & \\
    & + & d_{\psi \left( p \right), u f \left( p \right)} \theta_e \left[ d_p
    \psi \left( \xi \right) \oplus u \left( p \right) d_p f \left( \xi
    \right)_{_{_{_{}}}} \right]\\
    &  & \\
    & = & T_{u \sigma \left( p \right)} \cdot \theta_e  \left( \psi \left(
    \xi \right), d_p u \left( \xi \right) f \left( p \right)_{_{_{_{_{}}}}}
    \right) \noplus \noplus + d_p \left( u \left( p \right) \sigma \right)
    \left( \xi \right),
  \end{eqnarray*}
  thanks to (\ref{Ttrivial}). Using the equality
  \[ d_p \left( \lambda \sigma \right) = d_{\sigma \left( p \right)} \left(
     \lambda \mathbbm{I}_E \right) \cdot d_p \sigma, \]
  for any $\lambda \in \mathbbm{R}$, we conclude the required identity
  \begin{eqnarray*}
    d_p \left( u \sigma \right) \left( \xi \right) & = & d_p u \left( \xi
    \right) T_{u \sigma \left( p \right)} \sigma \left( p \right) \noplus
    \noplus + d_{\sigma \left( p \right)} [u \left( p \right) \mathbbm{I}_E]
    \cdot d_p \sigma \left( \xi \right) .
  \end{eqnarray*}
\end{proof}

The induced connection $\nabla^{\psi}$ over $\psi^{\ast} E$ is defined by the
formula
\begin{eqnarray*}
  \nabla_{\xi}^{\psi} \sigma & \assign & T^{- 1}_{\sigma \left( p \right)}
  \gamma_{\sigma \left( p \right)}^{\mathcal{H}} d_p \sigma \left( \xi
  \right)\\
  &  & \\
  & = & T^{- 1}_{\sigma \left( p \right)}  \left[ d_p \sigma \left( \xi
  \right) - H_{\sigma \left( p \right)} d_p \psi \left( \xi
  \right)_{_{_{_{}}}} \right],
\end{eqnarray*}
for any $\xi \in T_{N, p}$. It is obvious that the additive property of
$\nabla^{\psi}$ follows from the condition $\left( \ref{aditivDistrib}
\right)$. We show now that $\nabla^{\psi}$ satisfies the Leibniz property
\begin{eqnarray*}
  \nabla^{\psi}_{\xi}  \left( u \sigma \right) & = & d_p u \left( \xi \right)
  \sigma \left( p \right) + u \left( p \right) \nabla^{\psi}_{\xi} \sigma \,.
\end{eqnarray*}
Indeed using lemma \ref{GprdLm} and the identity (\ref{multHoriz}) we have
\begin{eqnarray*}
  \nabla^{\psi}_{\xi}  \left( u \sigma \right) & = & T^{- 1}_{u \sigma \left(
  p \right)} \gamma_{u \sigma \left( p \right)}^{\mathcal{H}} d_p  \left( u
  \sigma \right) \left( \xi \right)\\
  &  & \\
  & = & T^{- 1}_{u \sigma \left( p \right)} \gamma_{u \sigma \left( p
  \right)}^{\mathcal{H}} \left[ d_p u \left( \xi \right)_{_{_{_{}}}} T^{}_{u
  \sigma \left( p \right)} \sigma \left( p \right) + d_{\sigma \left( p
  \right)} [u \left( p \right) \mathbbm{I}_E] d_p \sigma \left( \xi
  \right)_{_{_{_{_{}}}}} \right]\\
  &  & \\
  & = & d_p u \left( \xi \right) \sigma \left( p \right) + T^{- 1}_{u \sigma
  \left( p \right)} \left[ d_{\sigma \left( p \right)} [u \left( p \right)
  \mathbbm{I}_E] d_p \sigma \left( \xi \right)_{_{_{_{_{}}}}} - H_{u \sigma
  \left( p \right)} d_p \psi \left( \xi \right)_{_{_{_{}}}} \right]\\
  &  & \\
  & = & d_p u \left( \xi \right) \sigma \left( p \right)\\
  &  & \\
  & + & T^{- 1}_{u \sigma \left( p \right)} \left[ d_{\sigma \left( p
  \right)} [u \left( p \right) \mathbbm{I}_E] d_p \sigma \left( \xi
  \right)_{_{_{_{_{}}}}} - d_{\sigma \left( p \right)} [u \left( p \right)
  \mathbbm{I}_E] H_{\sigma \left( p \right)} d_p \psi \left( \xi
  \right)_{_{_{_{}}}} \right]
  \end{eqnarray*}
  \begin{eqnarray*}
  & = & d_p u \left( \xi \right) \sigma \left( p \right) + T^{- 1}_{u \sigma
  \left( p \right)} \left[ d_{\sigma \left( p \right)} [u \left( p \right)
  \mathbbm{I}_E] \gamma_{\sigma \left( p \right)}^{\mathcal{H}} d_p \sigma
  \left( \xi \right)_{_{_{_{_{}}}}} \right] \\
  &  & \\
  & = & d_p u \left( \xi \right) \sigma \left( p \right) + T^{- 1}_{u \sigma
  \left( p \right)} \left[ u \left( p \right) \gamma_{\sigma \left( p
  \right)}^{\mathcal{H}} d_p \sigma \left( \xi \right)_{_{_{_{}}}} \right] \\
  &  & \\
  & = & d_p u \left( \xi \right) \sigma \left( p \right) + u \left( p \right)
  \nabla^{\psi}_{\xi} \sigma \,.
\end{eqnarray*}
We observe also that for any $s \in C^{\infty} \left( M, E \right)$ and $\xi
\in T_{N, p}$ we have the equalities
\begin{eqnarray*}
  \nabla_{\xi}^{\psi}  (\psi^{\ast} s) & = & T^{- 1}_{s \circ \psi \left( p
  \right)} \gamma_{s \circ \psi \left( p \right)}^{\mathcal{H}} d_{\psi \left(
  p \right)} s \cdot d_p \psi \left( \xi \right)\\
  &  & \\
  & = & \nabla s \left( \psi \left( p \right) \right) \cdot d_p \psi \left(
  \xi \right),
\end{eqnarray*}
in other terms the functorial formula
\begin{equation}
  \label{functConn} \nabla^{\psi}  (\psi^{\ast} s) = \psi^{\ast} \left( \nabla
  s \right),
\end{equation}
holds.

\subsubsection{The induced connection (second approach)}

We observe that the tangent space of $\psi^{\ast} E$ at the point $\left( y,
\eta \right)$ is given by the equality
\begin{eqnarray*}
  T_{\psi^{\ast} E, \left( y, \eta \right)} & = & \left\{ \left( \xi, \theta
  \right) \in T_{N, y} \oplus T_{E, \eta} \mid d_y \psi \left( \xi \right) =
  d_{\eta} \pi_E \left( \theta \right)_{_{_{_{}}}} \right\} .
\end{eqnarray*}
Given any horizontal distribution $H \in C^{\infty} \left( E \nocomma,
\pi^{\ast}_E T^{\ast}_M \otimes T_E \right)$ over $E$, we define the
horizontal distribution
\begin{eqnarray*}
  H^{\psi} \assign \Psi^{\ast} H & \in & C^{\infty} \left( \psi^{\ast} E
  \nocomma, \pi^{\ast}_{\psi^{\ast} E} T^{\ast}_N \otimes T_{\psi^{\ast} E}
  \right) .
\end{eqnarray*}
In explicit terms
\begin{eqnarray*}
  H^{\psi}_{\left( y, \eta \right)} & = & \mathbbm{I}_{T_{N, y}} \oplus
  H_{\eta} \cdot d_y \psi .
\end{eqnarray*}
If $H$ satisfies the identities $\left( \ref{aditivDistrib} \right)$ and
$\left( \ref{multHoriz} \right)$ then so does $H^{\psi}$. This follows indeed
from the identities
\begin{eqnarray*}
  d_{\left( y, \eta_1, \eta_2 \right)} \left( s m_{_{\psi^{\ast} E}} \right) &
  = & \mathbbm{I}_{T_{N, y}} \oplus d_{\left( \eta_1, \eta_2 \right)} \left( s
  m_{_E} \right),\\
  &  & \\
  d_{\left( y, \eta \right)} (\lambda \mathbbm{I}_{\psi^{\ast} E}) & = &
  \mathbbm{I}_{T_{N, y}} \oplus d_{\eta} \left( \lambda \mathbbm{I}_E \right)
  .
\end{eqnarray*}
By definition of \ $H^{\psi}$ we infer that the induced connection
$\nabla^{\psi}$ over $\psi^{\ast} E$ satisfies the formula
\begin{eqnarray*}
  \nabla_{\xi}^{\psi} \sigma & = & T^{- 1}_{\sigma \left( y \right)} \cdot
  \left[ d_y \sigma \left( \xi \right) - H_{\sigma \left( y \right)} \cdot d_y
  \psi \left( \xi \right)_{_{_{_{}}}} \right],
\end{eqnarray*}
for any $\xi \in T_{N, y}$.

The local frame $e$ induces a local frame $\eta \assign e \circ \psi$ of
$\psi^{\ast} E$ over $\psi^{- 1} \left( U \right)$. We compute the local
connection $A^{\psi}$ form of $\nabla^{\psi}$ with respect to such frame. We
notice that $\nabla^{\psi} \eta = \psi^{\ast} \left( e \cdot A \right) = \eta
\cdot \psi^{\ast} A$ by the previous remark. We infer the equality $A^{\psi} =
\psi^{\ast} A$.

\subsubsection{Parallel transport}

We consider a smooth curve $\gamma : \left( - \varepsilon, \varepsilon \right)
\longrightarrow M$ and a section $\sigma \in C^1 \left( \left( - \varepsilon,
\varepsilon \right), \gamma^{\ast} E \right)$ which satisfies the equation
\begin{eqnarray*}
  \nabla_{\frac{d}{d t}}^{\gamma} \sigma & = & 0,
\end{eqnarray*}
over $\left( - \varepsilon, \varepsilon \right)$ with $\sigma \left( 0 \right)
= \eta \in E_{\gamma \left( 0 \right)}$. If we write $\sigma \left( t \right)
= e \left( \gamma \left( t \right) \right) \cdot f \left( t \right)$ then
\begin{eqnarray*}
  \nabla_{\frac{d}{d t}}^{\gamma} \sigma & = & e \left( \gamma \left( t
  \right) \right) \cdot \left[ \dot{f} \left( t \right) + A \left(
  \dot{\gamma} \left( t \right) \right) \cdot f \left( t \right) \right] .
\end{eqnarray*}
We infer that the parallel transport map $\tau_{\gamma, t} : E_{\gamma \left( 0
\right)} \longrightarrow E_{\gamma \left( t \right)}$, $t \in \left( -
\varepsilon, \varepsilon \right)$ given by $\tau_{\gamma, t} \left( \eta
\right) = \sigma \left( t \right)$, is linear. We show the following fact.

\begin{lemma}
  For any smooth curve $\gamma : \left( - \varepsilon, \varepsilon \right)
  \longrightarrow M$ and for any section $\sigma \in C^1 \left( \left( -
  \varepsilon, \varepsilon \right), \gamma^{\ast} E \right)$, holds the identity
  \begin{equation}
    \label{covpar} \nabla_{\frac{d}{d t}}^{\gamma} \sigma \left( 0 \right) =
    \frac{d}{d t} _{\mid_{t = 0}}  \left[ \tau^{- 1}_{\gamma, t} \cdot \sigma
    \left( t \right)_{_{_{_{}}}} \right] .
  \end{equation}
\end{lemma}

\begin{proof}
  We notice first that the term $\tau^{- 1}_{\gamma, t} \cdot \sigma \left( t
  \right)$ is given by the intrinsic identities
  \begin{eqnarray*}
    \frac{d u_t}{d s} + A \left( \dot{\gamma} \left( s \right) \right) \cdot
    u_t \left( s \right) & = & 0,\\
    &  & \\
    u_t \left( t \right) & = & f \left( t \right),\\
    &  & \\
    e \left( \gamma \left( 0 \right) \right) \cdot u_t \left( 0 \right) & = &
    \tau^{- 1}_{\gamma, t} \cdot \sigma \left( t \right) .
  \end{eqnarray*}
  Integrating the first equation we infer
  \begin{eqnarray*}
    u_t \left( t \right) - u_t \left( 0 \right) & = & - \int_0^t A \left(
    \dot{\gamma} \left( s \right) \right) \cdot u_t \left( s \right) d s .
  \end{eqnarray*}
  Using the second equation we obtain
  \begin{eqnarray*}
    f \left( t \right) - u_t \left( 0 \right) & = & - \int_0^t A \left(
    \dot{\gamma} \left( s \right) \right) \cdot u_t \left( s \right) d s .
  \end{eqnarray*}
  Deriving with respect to the variable $t$ we obtain
  \begin{eqnarray*}
    \frac{d}{d t} u_t \left( 0 \right) & = & \dot{f} \left( t \right) + A
    \left( \dot{\gamma} \left( t \right) \right) \cdot u_t \left( t \right)\\
    &  & \\
    & = & \dot{f} \left( t \right) + A \left( \dot{\gamma} \left( t \right)
    \right) \cdot f \left( t \right) .
  \end{eqnarray*}
  Evaluating at $t = 0$ and multiplying both sides with $e \left( \gamma
  \left( 0 \right) \right)$ we infer the required conclusion.
\end{proof}

We consider now a $C^1$-vector field $\xi$ over $M$ and let $\varphi_{\xi, t}$
be the associated $1$-parameter subgroup of transformations of $M$. Let
$\Phi_{\xi, t} : E \longrightarrow E$ be the parallel transport map along the
flow lines of $\varphi_{\xi, t}$. It is obvious by definition, that the map
$\Phi_{\xi, t}$ satisfies $\pi_E \circ \Phi_{\xi, t} = \varphi_{\xi, t} \circ
\pi_E$.

The vector field $\Xi \assign \dot{\Phi}_{\xi, 0}$ over $E$ satisfies the
equality $\Xi \left( \eta \right) = H_{\eta} \left( \xi \right)$, for any
$\eta \in E$. This is a direct consequence of the definition of the induced
connection along the flow lines of $\xi$.

To any section $\sigma \in C^1 \left( M, E \right)$ we can associate a
$C^1$-vector field $\Sigma$ over $E$ defined as $\Sigma \left( \eta \right)
\assign T_{\eta} [\sigma \circ \pi_E \left( \eta \right)]$. Let $\Phi_{\Sigma,
t}$ be the associated $1$-parameter subgroup of transformations of $E$. In
explicit terms it satisfies $$
\Phi_{\Sigma, t} \left( \eta \right) = \eta + t
\sigma \circ \pi_E \left( \eta \right)\,.
$$
Then
\begin{eqnarray*}
  {}[\Xi, \Sigma] & = & \frac{d}{d t} _{\mid_{t = 0}}  \frac{d}{d s} _{\mid_{s
  = 0}}  \left( \Phi_{\xi, - t} \circ \Phi_{\Sigma, s} \circ \Phi_{\xi, t}
  \right) .
\end{eqnarray*}
The fact that the map $\Phi_{\xi, - t}$ is linear on the fibers implies
\begin{eqnarray*}
  \Phi_{\xi, - t} \circ \Phi_{\Sigma, s} \circ \Phi_{\xi, t} & = & \Phi_{\xi,
  - t} \left[ \Phi_{\xi, t} + s \sigma \circ \pi_E \circ \Phi_{\xi, t}
  \right]\\
  &  & \\
  & = & \mathbbm{I}_E \noplus + s \Phi_{\xi, - t} \cdot \sigma \circ \pi_E
  \circ \Phi_{\xi, t}\\
  &  & \\
  & = & \mathbbm{I}_E \noplus + s \Phi_{\xi, - t} \cdot \sigma \circ
  \varphi_{\xi, t} \circ \pi_E \,.
\end{eqnarray*}
Thus for any $\eta \in E_p$ holds
\begin{eqnarray*}
  \Phi_{\xi, - t} \circ \Phi_{\Sigma, s} \circ \Phi_{\xi, t} \left( \eta
  \right) & = & \eta \noplus + s \Phi_{\xi, - t} \cdot \sigma \circ
  \varphi_{\xi, t} \left( p \right) \in E_p \,.
\end{eqnarray*}
We conclude
\begin{eqnarray*}
  {}[\Xi, \Sigma] \left( \eta \right) & = & \frac{d}{d t} _{\mid_{t = 0}}
  T_{\eta} \left[ \Phi_{\xi, - t} \cdot \sigma \circ \varphi_{\xi, t} \left( p
  \right)_{_{_{_{}}}} \right]\\
  &  & \\
  & = & T_{\eta} \left[ \nabla_{\xi} \sigma \left( p \right) \right],
\end{eqnarray*}
i.e for any $\eta \in E$ the equality holds
\begin{equation}
  \label{BraketConn}  [\Xi, \Sigma] \left( \eta \right) = T_{\eta} \left[
  (\nabla_{\xi} \sigma) \circ \pi_E \left( \eta \right)_{_{_{_{}}}} \right] .
\end{equation}
Iterating twice we deduce the identity
\begin{equation}
  \label{Braket2Cov}  \left[ \Xi_1, [\Xi_2, \Sigma]_{_{_{_{}}}} \right] \left(
  \eta \right) = T_{\eta} \left[ (\nabla_{\xi_1} \nabla_{\xi_2} \sigma) \circ
  \pi_E \left( \eta \right)_{_{_{_{}}}} \right] .
\end{equation}
Moreover the fact that by (\ref{BraketConn}) the vector fields $[\Xi_j,
\Sigma]$, $j = 1, 2$ are tangent to the fibers of $E$ and constant along them
implies
\begin{equation}
  \label{vanishComm}  \left[ [\Xi_1, \Sigma], [\Xi_2, \Sigma]_{_{_{_{}}}}
  \right] = 0 \,.
\end{equation}

\subsection{The geometric meaning of the curvature tensor}

\begin{lemma}
  \label{TorsCurv}Let $R \assign \nabla^2$ be the curvature tensor of the
  connection $\nabla$. Then for any vector fields $\xi_1, \xi_2$ over $M$ and
  for any $\eta \in E$ the identity holds
  \begin{eqnarray*}
    \gamma_{\eta}^{\nabla} \left( \left[ \Xi_1, \Xi_2 \right] \left( \eta
    \right)_{_{_{_{}}}} \right) = T_{\eta} [R \left( \xi_2, \xi_1 \right)
    \eta] \,.
  \end{eqnarray*}
\end{lemma}

\begin{proof}
  Let $\sigma$ be a local section of $E$ such that $\sigma \left( p \right) =
  \eta$. By definition of horizontal lift $\Xi$ of a vector field $\xi$ we
  have
  \begin{eqnarray*}
    \Xi \left( \eta \right) & = & [d \sigma \left( \xi \right)] \circ \pi_E
    \left( \eta \right) - T_{\eta} \left[ (\nabla_{\xi} \sigma) \circ \pi_E
    \left( \eta \right)_{_{_{_{}}}} \right] .
  \end{eqnarray*}
  We infer by (\ref{BraketConn}) the identity
  \begin{eqnarray*}
    {}[d \sigma \left( \xi \right)] \circ \pi_E & = & \Xi + [\Xi, \Sigma] \,.
  \end{eqnarray*}
  We infer $\sigma_{\ast} \xi = \Xi + [\Xi, \Sigma]$ over $\tmop{Im} \sigma$.
  Thus
  \begin{eqnarray*}
    \sigma_{\ast} [\xi_1, \xi_2] & = & [\sigma_{\ast} \xi_1, \sigma_{\ast}
    \xi_2]\\
    &  & \\
    & = & [\Xi_1, \Xi_2] + \left[ \Xi_1, [\Xi_2, \Sigma]_{_{_{_{}}}} \right]
    + \left[ [\Xi_1, \Sigma]_{_{_{_{}}}}, \Xi_2 \right],
  \end{eqnarray*}
  thanks to (\ref{vanishComm}). We rewrite the previous equality as
  \begin{eqnarray*}
    {}[\Xi_1, \Xi_2] & = & \left[ \Xi_2, [\Xi_1, \Sigma]_{_{_{_{}}}} \right] -
    \left[ \Xi_1, [\Xi_2, \Sigma]_{_{_{_{}}}} \right] - \sigma_{\ast} [\xi_2,
    \xi_1] \,.
  \end{eqnarray*}
  Using (\ref{Braket2Cov}) we deduce
  \begin{eqnarray*}
    {}[\Xi_1, \Xi_2] \left( \eta \right) & = & T_{\eta} \left[ (\nabla_{\xi_2}
    \nabla_{\xi_1} \sigma - \nabla_{\xi_1} \nabla_{\xi_2} \sigma) \left( p
    \right)_{_{_{_{}}}} \right] - d_p \sigma ([\xi_2, \xi_1]) \\
    &  & \\
    & = & T_{\eta} \left[ (\nabla_{\xi_2} \nabla_{\xi_1} \sigma -
    \nabla_{\xi_1} \nabla_{\xi_2} \sigma - \nabla_{[\xi_2, \xi_1]} \sigma)
    \left( p \right)_{_{_{_{}}}} \right] - H_{\eta} ([\xi_2, \xi_1])\\
    &  & \\
    & = & T_{\eta} [R \left( \xi_2, \xi_1 \right) \sigma \left( p \right)]
    \noplus + H_{\eta} ([\xi_1, \xi_2]) \,.
  \end{eqnarray*}
  We infer the required conclusion.
\end{proof}

\vspace{1cm}
\noindent
Nefton Pali
\\
Universit\'{e} Paris Sud, D\'epartement de Math\'ematiques 
\\
B\^{a}timent 307 F91405 Orsay, France
\\
E-mail: \textit{nefton.pali@math.u-psud.fr}

\end{document}